\documentclass[11pt]{amsart}

\usepackage{hyperref}
\usepackage{ucs}
\usepackage[utf8]{inputenc}
\usepackage{amsmath}
\usepackage{amssymb}
\usepackage{amsthm}
\usepackage[english]{babel}
\usepackage{fontenc}
\usepackage{graphicx, tikz}
\usepackage{tikz-cd}
\usetikzlibrary{arrows}
\usepackage{amscd}

\author{Johannes Schmitt}
\title[Compactification of self-maps space of $\mathbb{CP}^1$ using stable maps]{A compactification of the moduli space of self-maps of $\mathbb{CP}^1$ using stable maps}
\address{Departement Mathematik, ETH Z\"urich, R\"amistrasse 101, 8092 Z\"urich, Switzerland}
\email{johannes.schmitt@math.ethz.ch}
\date{\today}

\newcommand{\dimension}{\textnormal{dim}}
\newcommand{\codim}{\textnormal{codim}}

\newcommand{\C}{\ensuremath{\mathbb{C}}}

\newcommand{\CP}{\ensuremath{\mathbb{CP}}}
\newcommand{\PP}{\ensuremath{\mathbb{P}}}
\renewcommand{\L}{\ensuremath{\mathcal{L}}}

\renewcommand{\dim}{\text{dim}}

\newcommand{\sslash}{/\!\!/}

\newcommand{\todo}[1]{}
\newcommand{\todoOld}[1]{}
\newcommand{\todoAlt}[1]{}
\newcommand{\todoFin}[1]{}
\newcommand{\taumo}[2]{\ensuremath{\tau_{#1}^{#2}}}  
\newcommand{\compomo}{\ensuremath{\mathfrak{c}}}

\newcommand{\comment}[1]{}

\newcommand{\detex}[1]{}  

\newcommand{\mapsfrom}{\mathrel{\reflectbox{\ensuremath{\mapsto}}}}

\newcommand{\Gene}[2]{\ensuremath{\mathfrak{G}_{#1,#2}}}  
\newcommand{\Basi}[2]{\ensuremath{\mathfrak{B}_{#1,#2}}}  

\begin{document}
\begin{abstract}
 We present a new compactification $M(d,n)$ of the moduli space of self-maps of $\mathbb{CP}^1$ of degree $d$ with $n$ markings. It is constructed via GIT from the stable maps moduli space $\overline M_{0,n}(\mathbb{CP}^1 \times \mathbb{CP}^1, (1,d))$. We show that it is the coarse moduli space of a smooth Deligne-Mumford stack and we compute its rational Picard group. Using the recursive boundary structure inherited from the stable maps space, we give an explicit algorithm for computing top-intersection numbers of divisors on $M(d,n)$. We also study the $m$-fold iteration map $M(d,n) \dashrightarrow M(d^m,n)$ and we give a geometric way to extend this rational map to parts of the boundary of $M(d,n)$.
\end{abstract}

 \maketitle

 \newtheoremstyle{test}
  {}
  {}
  {}
  {}
  {\bfseries}
  {.}
  { }
  {}
 
 \theoremstyle{test}
\newtheorem{Def}{Definition}[section]
\newtheorem{Exa}[Def]{Example}
\newtheorem{Rmk}[Def]{Remark}
\newtheorem{Exe}[Def]{Exercise}
\newtheorem{Theo}[Def]{Theorem}
\newtheorem{Lem}[Def]{Lemma}
\newtheorem{Cor}[Def]{Corollary}
\newtheorem{Pro}[Def]{Proposition}
\newtheorem*{Rmk*}{Remark}   
\newtheorem*{Exa*}{Example}   
\newtheorem*{Pro*}{Proposition} 
\newtheorem*{Def*}{Definition}
\newtheorem*{Cor*}{Corollary}
\newtheorem*{Lem*}{Lemma}
\newtheorem*{Theo*}{Theorem}

\section{Introduction}
\subsection{The moduli space of self-maps}
Self-maps of the projective line constitute a rich subject with connections to many branches of mathematics, including, among others, complex and arithmetic dynamics (\cite{MR2193309},\cite{MR2316407}), Hurwitz theory (\cite{hurwitz}) and enumerative geometry (\cite{MR2199225}).

Recall that a degree $d$ morphism $\varphi: \CP^1 \to \CP^1$ is given by two homogeneous polynomials of degree $d$ with no common zeros. A natural parameter space is thus the complement $\text{Rat}_d$ of the resultant hypersurface in $\CP^{2d+1}$. The automorphism group $\text{Aut}(\CP^1) = \text{PGL}_2(\mathbb{C})$ acts on $\text{Rat}_d$ by conjugation and we can form the GIT quotient $M_d = \text{Rat}_d \sslash \text{PGL}_2(\mathbb{C})$. 
%
%
%
This quotient $M_d$ is a moduli space of degree $d$ self-maps $C \to C$ of curves $C$ (abstractly) isomorphic to $\CP^1$.

To study how such maps degenerate in families, when the two defining polynomials obtain common zeroes, it is necessary to compactify $M_d$. 
In \cite{silverman}, Silverman obtains a compactification $M_d^{ss}$ of $M_d$ as a GIT quotient of $\CP^{2d+1}$ by $\text{PGL}_2(\mathbb{C})$. 
The boundary consists of self-maps $\tilde \varphi : C \to C$ of curves $C \cong \CP^1$ with degree $d'<d$ together with an effective divisor $D=\sum_{i} k_i [q_i]$ on $C$ of degree $d-d'$, where we ask that $k_i \leq (d+1)/2$ and moreover $k_i \leq (d-1)/2$ if $q_i$ is a fixed point of $\tilde \varphi$. Such a point $(\tilde \varphi, D)$ is the limit of a family of self-maps $\varphi : C \to C$ where the defining polynomials acquire common zeroes $q_i$ of multiplicities $k_i$.

\subsection{A new compactification}
Building on the work of Silverman, we construct a different compactification $M(d,0)$ (for $d \geq 2$ even), where we allow $C$ to degenerate into a nodal curve. More precisely, the points of $M(d,0)$ correspond to stable self-maps $\varphi: C \to C$ where $C$ is a connected genus $0$ curve with at worst nodal singularities, such that $\varphi$ has image in a unique irreducible component $C_0 \subset C$ and the total degree of $\varphi$ (summed over the components of the domain) is $d$. 
Here $\varphi$ is called stable if
\begin{itemize}
 \item every component $C' \neq C_0$ of $C$ contracted by $\varphi$ contains at least three nodes,
 \item all connected components $C_i$ of the complement $C \setminus C_0^{sm}$ of the smooth points of $C_0$ map with degree at most $(d+1)/2$. If in addition $C_i \cap C_0$ is a fixed point of $\varphi|_{C_0}$, we require the degree to be at most $(d-1)/2$.
\end{itemize}
By allowing the curve $C$ to have $n$ marked smooth points $p_1, \ldots, p_n$, we define the space $M(d,n)$ similarly.

To construct $M(d,n)$, we consider the stable maps moduli space 
\[Y_{d,n}=\overline M_{0,n}(\mathbb{CP}^1 \times \mathbb{CP}^1, (1,d)).\]
It parametrizes maps \[f=(\pi, \varphi): (C;p_1, \ldots, p_n) \to \CP^1 \times \CP^1\] of degree $(1,d)$ from an $n$-marked at worst nodal genus $0$ curve $C$, such that each component of $C$ contracted by $f$ contains at least three special points, i.e. nodes or markings. 
Observe that there is exactly one component $C_0$ of $C$ mapping isomorphically to $\CP^1$ via $\pi$. 

The data of $f$ is equivalent to a degree $d$ self-map 
\[C \xrightarrow{\varphi} \CP^1 \xrightarrow{\pi^{-1}} C_0 \subset C\]
of $C$ with image contained in a single component $C_0$, together with an isomorphism $\pi: C_0 \to \CP^1$. Note that when $C$ is irreducible, we have $C=C_0$ and this identifies $\text{Rat}_d \subset Y_{d,0}$ as the open subset of points $f$ with smooth source curve. We can thus see $Y_{d,0}$ as a compactification of $\text{Rat}_d$.

In order to forget the additional information of the isomorphism $\pi:C_0 \to \CP^1$ above to obtain the space $M(d,n)$, we let $\text{PGL}_2(\mathbb{C})$ act on $Y_{d,n}$ by postcomposing $\pi$ with $\psi \in \text{PGL}_2(\mathbb{C})$. The orbits of $\text{PGL}_2(\mathbb{C})$ parametrize exactly all possible ways to choose the identification $C_0 \cong \CP^1$. Then we can take the quotient using GIT.

A priori, for technical reasons, this GIT quotient only works well for $d$ even (see Lemma \ref{Lem:preimstable} and Corollary \ref{Cor:deven}). However, if we allow ourselves to put nonnegative rational weights $d_i$ on the points $p_i$, which affect the GIT stability condition, we are able to define quotient spaces $M(d|d_1, \ldots, d_n)$ also in the case $d$ odd (provided we have at least one marking). More precisely we ask that there exists $k \in \mathbb{Z}_{\geq 1}$ such that the numbers $\tilde d_i = k d_i$ are integers and such that $k(d+1)+\sum_i \tilde d_i$ is odd. The corresponding tuples $(d|d_1, \ldots, d_n)$ are called admissible. When $d$ is even and all $d_i=0$, we recover the previous description of $M(d,n)$.

\begin{Theo*}[see Corollary \ref{Cor:quotexist2}]
 For $(d|d_1, \ldots, d_n)$ admissible, there exists a GIT quotient $M(d|d_1, \ldots, d_n)$ of $Y_{d,n}$ by $\text{PGL}_2(\C)$ which is a normal projective variety. 
\end{Theo*}
For later use we mention here that the natural map $Y_{d,n} \to \overline M_{0,n}$ forgetting the map $f$ and stabilizing the curve $(C;p_1, \ldots, p_n)$ is $\text{PGL}_2(\mathbb{C})$-invariant and thus induces forgetful maps $M(d|d_1, \ldots, d_n) \to \overline M_{0,n}$.

Let us point out some differences between the compactifications $M(d,0)$ and $M_d^{ss}$ of $M_d$. When a map $\varphi \in M_d$ degenerates, instead of introducing a base point $q_i$ and recording its multiplicity $k_i$ as in $M_d^{ss}$, we insert a new component $C'$ of $C$ over the point $q_i$ together with a map $C' \to C$ of degree $k_i$. Thus in this map $C' \to C$ we can hope to preserve information about the behaviour of $\varphi$ around $q_i$ as it degenerates. Also note that the locus of maps $\tilde \varphi$ in $M_d^{ss}$ with a base point of multiplicity $k$ has codimension $2k-1$. On the other hand, the corresponding locus of $\varphi:C \to C$ in $M(d,0)$ having a component $C' \neq C_0$ with degree $k$ under $\varphi$ is a divisor. Thus it seems easier to study such degenerations in $M(d,0)$. 
\subsection{Modular interpretation}
We can also give a rigorous modular interpretation for $M(d|d_1, \ldots, d_n)$. It is the coarse moduli space of a smooth Deligne-Mumford stack $\mathcal{M}(d|d_1, \ldots, d_n)$, which is an open substack of the quotient stack $\overline{ \mathcal{M}}_{0,n}(\CP^1 \times \CP^1,(1,d))/\text{PGL}_2(\C)$. While every single geometric point of this stack can be interpreted as a self-map of an $n$-pointed nodal curve $C$ as described above, this does not generalize well to self-maps of families of nodal curves. Instead, it is necessary to work with two different families of curves.
\begin{Theo*}[see Theorem \ref{Theo:modprop}]
 Given a scheme $S$, the objects of the  stack $\overline{ \mathcal{M}}_{0,n}(\CP^1 \times \CP^1, (1,d))/\text{PGL}_2(\mathbb{C})$ over $S$ are diagrams
\begin{equation*}  
\begin{tikzpicture}
 \draw (0,2) node{$\mathcal{C}$};
 \draw[->] (-0.1,1.7) -- (-0.1,1) node[left]{$\pi$} -- (-0.1,0.3);
 \draw[->] (0.1,0.3) -- (0.1,1) node[right]{$\sigma_i$} -- (0.1,1.7);
 \draw (0,0) node{$S$};
 \draw[->] (0.25,2.05) --(1,2.05) node[above]{$\mu$} -- (1.8,2.05);
 \draw[->] (0.25,1.95) --(1,1.95) node[below]{$\varphi$} -- (1.8,1.95);
 \draw (2,2) node{$\mathcal{\tilde C}$};
 \draw[->] (1.7,1.7) -- (1,1) node[right]{$\tilde \pi$} -- (0.3,0.3);
\end{tikzpicture}
\end{equation*}
where $\pi, \tilde \pi$ are flat, projective families of quasi-stable genus $0$ curves, all geometric fibres of $\tilde \pi$ are isomorphic to $\CP^1$ and the maps $\mu, \varphi$ satisfy $\mu_*([\mathcal{C}_s]) = [\mathcal{\tilde C}_s]$ and $\varphi_*([\mathcal{C}_s]) = d [\mathcal{\tilde C}_s]$ for all geometric points $s \in S$ (plus the stable maps condition for the map $(\mu_s,\phi_s) : (\mathcal{C}_s;\sigma_1(s), \ldots, \sigma_n(s)) \to \mathcal{\tilde C}_s \times \mathcal{\tilde C}_s$).
\end{Theo*}
Over the points $s$ with $\mathcal{C}_s \cong \CP^1$, the morphism $\mu^{-1} \circ \phi$ becomes a self-map of the family $\mathcal{C}$ of degree $d$ as expected.
While it might seem that the use of two different families of curves no longer allows an interpretation as a self-map, the second family $\tilde C$ will be an important ingredient in defining self-composition and iteration maps on $M(d|d_1, \ldots, d_n)$ in Section \ref{Sect:Iterate}.
\subsection{General properties and the Picard group of \texorpdfstring{$M(d|d_1, \ldots, d_n)$}{M(d|d1, ..., dn)}}
From a result of Levy (\cite{levy}), we conclude that the space $M(d|d_1, \ldots, d_n)$ is rational. As it is the coarse moduli space of smooth Deligne-Mumford stack, it has finite quotient singularities. This allows us to compute its rational Picard group.
\begin{Theo*}[see Theorem \ref{Theo:Generators}, Theorem \ref{Theo:relations} and Corollary \ref{Cor:preBasis}]
 For $d \geq 2$ even and $n \geq 0$, the rational Picard group of $M(d,n)$ is generated by boundary divisors together with the divisor class $\mathcal{H}$ descending from the evaluation class $\text{ev}_1^* \mathcal{O}_{\CP^1\times \CP^1}(1,0)$ on $Y_{d,n}$ for $n=1,2$. The relations between these generators are  generated by the pullbacks of the relations
 \begin{equation*}\sum_{\substack{i,j \in A\\ k,l \in B}} D(A;B) = \sum_{\substack{i,k \in A\\ j,l \in B}} D(A;B) \end{equation*}
 between boundary divisors in  $\overline M_{0,n}$ (for $\{i,j,k,l\} \subset \{1, \ldots, n\}$) under the forgetful map $M(d,n) \to \overline M_{0,n}$. 
\end{Theo*}
This solves the case when all $d_i=0$. If some of them are not, we have for $d=0$ an additional generator $\mathcal{G}$ descending from $\text{ev}_1^* \mathcal{O}_{\CP^1\times \CP^1}(0,1)$. On the other hand, we can have additional relations: we need to divide by the classes of all codimension one loci which are not GIT-stable with respect to $(d|d_1, \ldots, d_n)$ (see Corollary \ref{Cor:preBasis}).

The relations above are determined using a method adapted from \cite{pandhaintersect}, which is based on intersecting possible relations with test curves $C_{B,k}$. In Proposition \ref{Pro:iddivisors}, we give an algorithm for computing the class of a divisor $D$ in terms of the generators above from its intersection numbers with such test curves. We also determine the classes of some interesting divisors, such as the locus $D_{i=\text{fix}}$ where the $i$-th marking is a fixed point of the self-map (\ref{Def:fixdiv}). 
\subsection{Iteration maps}
Other geometrically significant divisors can be obtained by studying composition and iteration of maps on the moduli spaces $Y_{d,n}$ and $M(d|d_1, \ldots, d_n)$. We give a geometric interpretation for an extension
\[\compomo : Y_{d_1,0} \times Y_{d_2,0} \dashrightarrow Y_{d_1 d_2,0}\]
of the composition morphism 
\[\text{Rat}_{d_1} \circ \text{Rat}_{d_2} \to \text{Rat}_{d_1 d_2}, (f_1, f_2) \mapsto f_2 \circ f_1\]
to parts of the boundary. This allows us to define an $m$-fold self-composition map
\[\mathfrak{sc}_m : M(d|d_1, \ldots, d_n) \dashrightarrow M(d^m|d_1, \ldots, d_n).\]
Using this map, we define divisors $\text{Per}_m(\lambda) \subset M(d|d_1, \ldots, d_n)$ parametrizing self-maps $\varphi: C \to C$ having an $m$-periodic point $p$ with a given multiplier $d \varphi^{\circ m}|_p = \lambda$. We relate these cycles to similar varieties studied by Milnor (\cite{milnorquaddyn}).
\subsection{Intersection theory of \texorpdfstring{$M(d|d_1, \ldots, d_n)$}{M(d|d1, ..., dn)}}
Another advantage of using the stable maps spaces for the compactification is that their intersection theory is very well-understood and many of the results carry over to the quotients $M(d|d_1, \ldots, d_n)$. In Section \ref{Sect:psi} we define $\psi$-classes and show that their intersection numbers satisfy recursions corresponding to the usual String, Dilaton and Divisor equations. But we are also able to compute general top-intersections of divisors $D_1, \ldots, D_{2d-2+n}$ on $M(d|d_1, \ldots, d_n)$. When $D_1$ is a boundary divisor, it is the quotient of a fibre product of smaller-dimensional moduli spaces. In Section \ref{Sect:EquivarProduct} we look at the equivariant intersection theory of such products in some more generality. Using these results, we can determine the intersection of the restrictions of the divisors $D_2, \ldots, D_{2d-2+n}$ on $D_1$. This leads to the following result.
\begin{Theo*}[see Section \ref{Sect:algo}]
 There is a recursive algorithm computing top-intersection numbers of divisors on $M(d|d_1, \ldots, d_n)$, using as input the base cases $d=0, n=2,3$ and $d=1,n=1$ together with the torus-equivariant intersection theory of divisors on $\overline M_{0,n'}(\CP^1,d')$.
\end{Theo*}
This algorithm has been implemented and tested using \verb+SAGE+.
\subsection{Plan of the paper}
The paper is organized as follows: in Section \ref{Sect:Construction} we present the construction of the spaces $M(d|d_1, \ldots, d_n)$. Their general properties as well as some key examples for small $d$ and $n$ are studied in Section \ref{Sect:Properties}. The rational Picard groups of $Y_{d,n}$ and $M(d|d_1, \ldots, d_n)$ are computed in Section \ref{Sect:Picard}. Section \ref{Sect:Iterate} studies composition and iteration maps and loci of maps with given multipliers at periodic points. Finally, in Section \ref{Sect:IntTh} we give a recursive algorithm for computing intersection numbers of divisors on $M(d|d_1, \ldots, d_n)$.

Appendix \ref{App:notations} contains  a list of notations used throughout the text. A collection of small technical results, which we include due to the lack of a good reference, can be found in Appendix \ref{App:generalities}. In Appendix \ref{App:stackaction} we treat group actions on stacks and the corresponding quotient stacks, based on the work of Romagny (\cite{romagny2005}). 
\section*{Acknowledgements}
I would like to express my gratitude towards my advisor Rahul Pandharipande, who proposed the construction of $M(d|d_1, \ldots, d_n)$ to me and guided me in my research. I also want to thank 
Felix Janda, Georg Oberdieck, Christoph Schie\ss{}l, Junliang Shen and Qizheng Yin
for their advice and for helpful discussions about various points in this paper. I am indebted to Emily Clader and Javier Fres\'an for helping me improve the quality of this text considerably. Finally, I want to thank Jakob Oesinghaus, who taught me much about stacks and was always willing to help me with the finer and coarser technical points concerning them. 

I am supported by the grant SNF-200020162928.

\section*{Conventions}
Throughout the paper we will work over the complex numbers. Hence in the future, notations such as $\PP^1, \text{GL}_n, \text{SL}_n,\text{PGL}_n$ will always denote the corresponding objects defined over $\mathbb{C}$.

By a coarse moduli space of a stack $\mathcal{X}$ we mean a morphism $\mathcal{X} \to X$ from $\mathcal{X}$ to an algebraic space $X$, which is initial among morphisms to algebraic spaces and induces a bijection on geometric points.

\tableofcontents


\section{Construction of the moduli space of self-maps} \label{Sect:Construction}
We want to study self-maps of $\mathbb{P}^1$ of a fixed degree $d$ modulo the conjugation action of $\text{Aut}(\PP^1)=\text{PGL}_2$. For this we will start with a moduli space which was already studied by Silverman in \cite{silverman} and chose a different way to compactify it, also allowing markings now. 
\subsection{The space of degree \texorpdfstring{$d$}{d} maps} \label{Sect:degdmaps}
A morphism of degree $d\geq 0$ from $\PP^1$ to itself is given by two sections $F_a,F_b$ of $\mathcal{O}_{\PP^1}(d)$ not vanishing simultaneously. Identifying global sections of $\mathcal{O}_{\PP^1}(d)$ with homogeneous polynomials in $X,Y$ of degree $d$, the morphism is of the form
\begin{align*}
 [X:Y] \mapsto &[F_a(X,Y),F_b(X,Y)]\\ = &[a_0 X^d + a_1 X^{d-1}Y + \ldots + a_d Y^d:b_0 X^d  + \ldots + b_d Y^d],
\end{align*}
where the two polynomials $F_a, F_b$ have no common factor. This last condition can be rephrased as
\[\text{Res}(F_a,F_b) \neq 0.\]
Here $\text{Res}(F_a,F_b)$ is the resultant of $F_a,F_b$, which is a homogeneous polynomial of degree $2d$ in the coefficients $a_0, \ldots, b_d$. As the pair $\lambda F_a, \lambda F_b$ for $\lambda \in \C \setminus \{0\}$ gives the same map, we see that the degree $d$ maps from $\PP^1$ to itself are exactly given by elements 
\[[a_0:a_1: \ldots :a_d : b_0 : \ldots : b_d] \in \PP^{2d+1} \setminus \{\text{Res}(F_a,F_b) = 0\}.\]
Until now we have only repeated the definitions used in \cite{silverman}. The numbers $a_i,b_i$ are interpreted as the parameters of the map they describe. Instead we want to interpret a morphism $\varphi: \PP^1 \to \PP^1$ as its graph $\Gamma_\varphi \subset \PP^1\times \PP^1$. This graph will be the zero-locus of a section $\gamma \in H^0(\PP^1 \times \PP^1, \mathcal{O}(d,1))$. Indeed, by the Kunneth formula we have
\[H^0(\PP^1 \times \PP^1, \mathcal{O}(d,1)) = H^0(\PP^1, \mathcal{O}(d)) \otimes_\C H^0(\PP^1, \mathcal{O}(1)).\]
So $\gamma$ can uniquely be written as $\gamma = f \otimes S + g \otimes T$ where $f,g \in H^0(\PP^1, \mathcal{O}(d))$ and $S,T$ denote coordinates on the second factor of $\PP^1\times \PP^1$. Then one checks that $\gamma$ vanishes exactly on the graph of the morphism $\phi = [g:-f]$. Let $Z_d=\PP(H^0(\PP^1 \times \PP^1, \mathcal{O}(d,1)))$, then (the graphs of) rational maps of degree $d$ are parametrized by an open subset $\text{Rat}_d \subset Z_d$. Indeed we have the following modular interpretation for $\text{Rat}_d$.
\begin{Lem} \label{Lem:Umodinterpret}
 The variety $\text{Rat}_d$ carries a universal family 
 \[
  \begin{CD}
   \PP^1_{\text{Rat}_d} @>>> \PP^1_{\text{Rat}_d}\\
   @VVV @VVV\\
   \text{Rat}_d @= \text{Rat}_d
  \end{CD}
 \]
 and it represents the functor 
 \begin{align*}
  \underline{\text{Rat}}_d : \textbf{Sch}_\C &\to \textbf{Sets},\\
   S &\mapsto \{S\text{-morphisms } \phi: \PP^1_S \to \PP^1_S\text{ with } \phi^* \mathcal{O}(1) = \mathcal{O}(d)\}.
 \end{align*}
\end{Lem}
\begin{proof}
 See \cite[Theorem 3.1]{silverman}.
\end{proof}
We also want to find an interpretation for the elements of $Z_d \setminus \text{Rat}_d$. They correspond to $f,g$ having some common roots, say we have $p_i = [c_i:d_i] \in \PP^1$ distinct roots of multiplicity $e_i > 0$ for $i=1, \ldots, k$. Then
\[f = \prod_{i=1}^k (d_iX-c_iY)^{e_i} \tilde f, g = \prod_{i=1}^k (d_iX-c_iY)^{e_i} \tilde g\]
with $\tilde f, \tilde g$ homogeneous polynomials of degree $\tilde d= d - \sum_{i=1}^k e_i$ without common factor. Then the zero locus of $\gamma = f \otimes S + g \otimes T$ in $\PP^1 \times \PP^1$ consists of one \emph{horizontal section} $h$, namely the graph of the degree $\tilde d$ map given by $[\tilde g: -\tilde f]$, together with the \emph{vertical sections} $v_i=\{p_i\} \times \PP^1$, on which $\gamma$ vanishes with multiplicity $e_i$. This corresponds to a rational map of degree $\tilde d$ with locus of indeterminacy formed by the points $p_i$. After cancellation, this gives a morphism of degree $\tilde d$. We will see soon that the vertical sections fit naturally with our new geometric compactification of $\text{Rat}_d$.
\subsection{The conjugation action and stability}
Now we want to study self-maps of $\PP^1$ up to change of coordinates on $\PP^1$, that is we want to divide by the conjugation action of $\text{Aut}(\PP^1)=\text{PGL}_2$. For technical reasons we will instead divide by the action of $G=\text{SL}_2$. We have a natural action of $G$ on $\PP^1$, given by 
\[\begin{pmatrix}
   a & b\\ c &d
  \end{pmatrix} [X:Y] = [aX+bY: cX + dY].
\]
Then for a map $\varphi: \PP^1 \to \PP^1$ and an element $g \in G$ we define
\[(g \varphi)(x) = g \varphi(g^{-1}x).\]
Via the identification above, this also induces a (projectively) linear action of $G$ on $Z_d$. The action leaves the subset $\text{Rat}_d$ invariant. When viewing elements of $Z_d$ as the corresponding graphs $\Gamma \subset \PP^1 \times \PP^1$, the action comes from the diagonal action $g(p,q) = (gp, gq)$ of $\text{SL}_2$ on $\PP^1 \times \PP^1$. 

Now in order to define a quotient of $Z_d$ by $G$, we will use Geometric Invariant Theory as described by Mumford in \cite{git}. As $Z_d \cong \PP^{2d+1}$, it has a natural line bundle $\L = \mathcal{O}_{\PP^{2d+1}}(1)$, which is ample and carries a canonical $G$-linearization (see \cite[1.§3]{git}). In \cite[Proposition 2.2]{silverman}, Silverman describes the set of (semi)stable points $Z_d^{ss}, Z_d^s$ corresponding to the action above. His results have the following convenient interpretation in terms of vertical sections.
\begin{Lem} \label{Lem:numstability}
 An element $\Gamma \in Z_d$ consisting of a horizontal section $h$ and vertical sections $v_i=\{p_i\} \times \PP^1$ of multiplicities $e_i$ ($i=1, \ldots, k$) is
 \begin{itemize}
  \item semistable, iff $e_i \leq \frac{d-1}{2}$ or we have $e_i \leq \frac{d+1}{2}$ and $p_i$ is not a fixed point of the induced morphism $\PP^1 \to \PP^1$ of lower degree for $i=1, \ldots, k$,
  \item stable, iff $e_i < \frac{d-1}{2}$ or we have $e_i < \frac{d+1}{2}$ and $p_i$ is not a fixed point of the induced morphism $\PP^1 \to \PP^1$ of lower degree for $i=1, \ldots, k$.
 \end{itemize}
\end{Lem}
\begin{proof}
 By \cite[Proposition 2.2]{silverman}, the rational map given by $[a_0 X^d + \ldots + a_d Y^d:b_0 X^d + b_1 X^{d-1}Y + \ldots + b_d Y^d]$ is unstable if and only if after conjugation by $G$ we can reach 
 \begin{equation}
 a_i = 0\text{ for }i \leq \frac{d-1}{2}\text{ and }b_i = 0\text{ for }i \leq \frac{d+1}{2}. \label{eqn:stability1}
 \end{equation}
 Condition (\ref{eqn:stability1}) is equivalent to saying that
 \begin{align}
 &F_a\text{ vanishes at }[1:0]\text{ to order }>\frac{d-1}{2}\nonumber \\
 \text{ and }&F_b\text{ vanishes at }[1:0]\text{ to order }> \frac{d+1}{2}. \label{eqn:stability2}
 \end{align}
 One checks that the negation of (\ref{eqn:stability2}) is equivalent to the semistability condition given above for $p_i = [1:0]$. Here one uses that $p=[1:0]$ is a fixed point of the induced morphism iff $F_b$ vanishes at $p$ with strictly higher order than $F_a$. But now our formulation of the stability condition no longer refers to the specific point $[1:0]$ and as $G$ acts transitively on $\PP^1$, we have the desired criterion for semistability of $\Gamma$. The stable case follows similarly.
\end{proof}
From the above one sees easily, that for $d$ even the situation is particularly nice.
\begin{Cor} \label{Cor:deven}
 For $d$ even we have $Z_d^{ss} = Z_d^s$.
\end{Cor}
It is also obvious that $\text{Rat}_d \subset Z_d^s$. By \cite[Theorem 2.1]{silverman} we can now define quotients $M_d = \text{Rat}_d \sslash G$, $M_d^{ss} = Z_d^{ss} \sslash G$ , $M_d^s = Z_d^s \sslash G$. We will not work with those in the future, but mention them for completeness.

Finally, we want to add a remark that was mostly proved in Silverman's paper and we include the full result here.
\begin{Pro}
 The space $M_d$ is a coarse moduli space for the functor
\begin{align*}
  \underline{M}_d : \textbf{Sch}_\C &\to \textbf{Sets}\\
   S &\mapsto \underline{\text{Rat}}_d(S) / \sim,
 \end{align*}
 where two $S$-morphisms $\phi,\psi: \PP^1_S \to \PP^1_S$ satisfy $\phi \sim \psi$ if there exists an isomorphism $f : \PP^1_S \to \PP^1_S$ over $S$ such that $\phi \circ f = f \circ \psi$.
\end{Pro}
\begin{proof}
 This is essentially the content of Theorem 3.2 in \cite{silverman}. There, a natural transformation $\psi: \underline{M}_d \to \text{Hom}(-,M_d)$ is constructed such that the diagram 
 \[\begin{CD} 
    \underline{\text{Rat}}_d @= \text{Hom}(-,\text{Rat}_d)\\
    @VVV @V\circ \pi VV\\
    \underline{M}_d @>\psi >> \text{Hom}(-,M_d)
   \end{CD}
\]
 commutes, where $\pi: \text{Rat}_d \to M_d$ is the quotient map. It is also shown that $\psi$ induces a bijection on geometric points. To complete the proof that $\psi$ is initial among such natural transformations, one can use the fact that $\text{Rat}_d \to M_d$ is a categorical quotient. Indeed, assume that $\tilde \psi : \underline{M}_d \to \text{Hom}(-,Z)$ is another natural transformation. Then composing with the morphism $\underline{\text{Rat}}_d \to \underline{M}_d$ we obtain a functor $\underline{\text{Rat}}_d \to \text{Hom}(-,Z)$. But this morphism of functors is equivalent to a morphism $\varphi: \text{Rat}_d \to Z$ of schemes. We would like to show that this morphism is $G$-invariant. Indeed, consider the maps $\sigma, \pi_2: G \times \text{Rat}_d \to \text{Rat}_d$, where $\sigma$ is the action map and $\pi_2$ is the projection on the second factor. We need to show $\varphi \circ \sigma = \varphi \circ \pi_2$. But these maps are the pullbacks of $\varphi \in \text{Hom}(\text{Rat}_d,Z)$ under $\sigma, \pi_2$. By functoriality, they come from the pullbacks 
 \[\sigma^* \mathcal{U}, \pi_2^* \mathcal{U} \in \underline{\text{Rat}}_d(G \times \text{Rat}_d)\]
 of the universal self-map 
 \[\mathcal{U} = (\PP^1_{\text{Rat}_d} \to \PP^1_{\text{Rat}_d}) \in \underline{\text{Rat}}_d(\text{Rat}_d)\]
 over $\text{Rat}_d$. 
 However, from their definition one sees immediately that $\sigma^* \mathcal{U} \cong \pi_2^* \mathcal{U}$, so they induce the same maps $\varphi \circ \sigma = \varphi \circ \pi_2$. Now as $\text{Rat}_d \to M_d$ is a categorical quotient, the map $\varphi$ factors through a unique map $M_d \to Z$. One checks that this gives the desired natural transformation $\text{Hom}(-,M_d) \to \text{Hom}(-,Z)$ using that $\underline{\text{Rat}}_d \to \underline{M}_d$ is surjective on objects over all schemes $S$. 
\end{proof}
\subsection{Parametrized graphs}
We have already interpreted $Z_d$ as the space of (generalized) graphs of rational self maps of $\PP^1$ of degree $d$. Now we will consider parametrizations of these graphs by trees of $\PP^1$s, that is, nodal curves of genus $0$.

To be more precise, consider the stable maps space 
\[Y_{d,n} = \overline M_{0,n}(\PP^1 \times \PP^1, (1,d)).\]
The notation $Y_{d,n}$ is not standard but as the space will be used frequently in the future, we use $Y_{d,n}$ for brevity. From \cite[Theorem 2]{fultonpandha} we see that it is a normal, projective variety of pure dimension $2d+1+n$ and it is locally the quotient of a nonsingular variety by a finite group.  

We now show that there is a natural action of $\text{PGL}_2$ on $Y_{d,n}$. At this point it will be advantageous to construct this action as an action of the group scheme $\text{PGL}_2$ on the moduli stack $\overline{ \mathcal{M}}_{0,n}(\PP^1 \times \PP^1, (1,d))$ inducing an action of $\text{PGL}_2$ on the coarse moduli space $Y_{d,n}$. Here we use the definition of \cite{romagny2005} for a group action on a stack (see also Appendix \ref{App:stackaction}).
\begin{Cor}
 There exists a natural strict action of the group $\text{PGL}_2$ on the stack $\overline{\mathcal{M}}_{0,n}(\PP^1 \times \PP^1, (1,d))$ induced by the diagonal action of $\text{PGL}_2$ on $\PP^1 \times \PP^1$. For $g \in \text{PGL}_2$ and $(f : C \to \PP^1 \times \PP^1; p_1 , \ldots, p_n) \in \overline{\mathcal{M}}_{0,n}(\PP^1 \times \PP^1, (1,d))$ a pair of $\mathbb{C}$-valued points, we have
 \[g (f;p_1, \ldots, p_n) = (g_* \circ f; p_1, \ldots, p_n),\] 
 where $g_* : \PP^1 \times \PP^1, (p,q) \mapsto (gp, gq)$.
\end{Cor}
\begin{proof}
This follows from Lemma \ref{Lem:stackyaction} applied to the diagonal action of $\text{PGL}_2$ on $X= \PP^1 \times \PP^1$.
\end{proof}
We immediately obtain an induced action of $\text{PGL}_2$ on the coarse moduli space $Y_{d,n}$ of $\overline{\mathcal{M}}_{0,n}(\PP^1 \times \PP^1, (1,d))$. For taking the GIT quotient, we will work with the induced action of $G=\text{SL}_2$.
\todoOld{Insert consequences, action on the coarse moduli space etc.} \todoOld{What technical reasons?}
\begin{Lem} \label{Lem:jConstruction}
 There exists a natural $G$-equivariant morphism $j : Y_{d,n} \to Z_d$ uniquely defined by requiring that the image 
 \[s=j((f;p_1, \ldots, p_n)) \in Z_d = \PP(H^0(\PP^1 \times \PP^1, (d,1)))\]
 of the class of the morphism $f : C \to \PP^1 \times \PP^1$ satisfies
 \[\text{div}(s) = f_* [C].\]
 In particular we see that $j$ does not depend on the marked points. 
\end{Lem}
\begin{proof}
 The map $j$ was constructed by Givental in \cite[The Main Lemma]{givental} for the case $n=0$. From the proof there it is clear that on the closed points it is given by the set theoretic map described above. We define it for arbitrary $n$ by precomposing the map for $n=0$ with the forgetful map of all marked points. Considering the action of $G$ on closed points of $Y_{d,n}$, one sees immediately that the forgetful map $Y_{d,n} \to Y_{d,0}$ is $G$-equivariant. As the $G$-action on $Z_d$ is given by the componentwise action of $G$ on $\PP^1 \times \PP^1$ when considering elements of $Z_d$ as their corresponding vanishing schemes in $\PP^1 \times \PP^1$, one verifies that the map $j$ is $G$-equivariant in the $n=0$ case by considering the action on closed points. 
\end{proof}
In the case $n \geq 1$, the map $j$ loses information by forgetting the marked points. However, by using the evaluation maps, we can preserve this information in many cases. 
\begin{Lem} \label{Lem:Jconstruction}
 Let $\text{ev}_i:Y_{d,n}=\overline M_{0,n}(\PP^1 \times \PP^1, (1,d)) \to \PP^1 \times \PP^1$ be the evaluation associated to the $i$th marked point, and let $\pi_1 : \PP^1 \times \PP^1 \to \PP^1$ be the projection to the first component. Let $\Delta \subset (\PP^1)^n$ be the union of all the diagonals, that is
 \[\Delta = \{(p_1, \ldots, p_n) \in (\PP^1)^n; \exists i \neq j \text{ such that } p_i = p_j\}.\]
 Then the map 
 \begin{equation}
  J : Y_{d,n} \to Z_d  \times (\PP^1)^n , J = j \times (\pi_1 \circ \text{ev}_1) \times \ldots \times (\pi_1 \circ \text{ev}_n)
 \end{equation}
 is equivariant with respect to the diagonal action of $G$ on $(\PP^1)^n$. Furthermore, it is an isomorphism over $\text{Rat}_d \times ((\PP^1)^n \setminus \Delta)$.
\end{Lem}
\begin{proof}
 The $G$-equivariance follows directly from the equivariance of the map $j$ and the definition of the $G$-action on $Y_{d,n}$. 
 If a closed point $(f:C \to \PP^1 \times \PP^1; p_1, \ldots, p_n) \in Y_{d,n}$ maps to $([s], q_1, \ldots, q_n) \in \text{Rat}_d \times ((\PP^1)^n \setminus \Delta)$, then all components of $C$ must map with degrees $(1,d)$ or $(0,0)$, as $s$ has no vertical sections. But as also the marked points $p_1, \ldots, p_n$ map to distinct points in $\PP^1 \times \PP^1$, there can be no contracted components for stability reasons. Hence $C$ is irreducible and $\pi_1 \circ f : C \to \PP^1$ is an isomorphism. But then we can uniquely reconstruct the positions of the markings $p_i \in C$ from the points $q_i = \pi_1(f(p_i)) \in \PP^1$. All in all, this shows that the restriction of $J$ to the preimage of $\text{Rat}_d \times ((\PP^1)^n \setminus \Delta)$ induces a bijection on closed points. Thus it is an isomorphism by Proposition \ref{Pro:bijectiso}. 
\end{proof}
For an algebraic group $H$ acting on a scheme $X$, we denote by $\text{Pic}^H(X)$ the group of $H$-linearized line bundles on $X$ (c.f. \cite[1.§3]{git}). We want to obtain an ample $G$-linearized line bundle $\mathcal{M}'$ on $Y_{d,n}$ allowing us to form a quotient of $Y_{d,n}$ by $G$. Additionally, we want the stable and semistable loci of $\L'$ to coincide and to be given by the preimage of the stable locus of $\L$ in $Z_d$ via the map $j$. For this, we use the following construction adapted from \cite[Proposition 2.18]{git}.
\begin{Lem} \label{Lem:preimstable}   
 Let $G=\text{SL}_2$ act on projective varieties $Y,Z$ and let $j:Y \to Z$ be a $G$-equivariant map. Assume we have $\L \in \text{Pic}^G(Z)$,  $\mathcal{M}\in \text{Pic}^G(Y)$ both ample, $G$-linearized line bundles and assume $Z^{ss}=Z^{s}$ ((semi-)stability with respect to $\mathcal{L}$). Then there exists $N>0$ such that for all $n\geq N$ we have that $\mathcal{M}'=\mathcal{M}\otimes (j^*\L)^{\otimes n}$ is an ample $G$-linearized line bundle on $Y$ with $Y^{ss}=Y^s=j^{-1}(Z^s)$.
\end{Lem}
\begin{proof}
 As $\mathcal{M}$ is ample and $\L$ is base-point free, it is clear that $\mathcal{M}'$ is ample and $G$-linearized. To identify (semi)stable points, we want to use the Hilbert Mumford numerical criterion (see \cite[Theorem 2.1]{git}). We note that for a given $y \in Y$ and a one-parameter subgroup $\lambda$ in $G$ we have
 \[\mu^{\mathcal{M}\otimes (j^*\L)^{\otimes n}}(y,\lambda) = \mu^{\mathcal{M}}(y,\lambda) + n\mu^{\L}(j(y),\lambda).\]
 Now by \cite[Proposition 7.5]{mukai}, all one-parameter subgroups of $\text{SL}_2$ are conjugate to a multiple of the standard diagonal one-parameter subgroup
 \[T : \mathbb{G}_m \to \text{SL}_2, t \mapsto \left( \begin{matrix}
                                                         t & 0\\
                                                         0 & t^{-1}
                                                        \end{matrix}
 \right)\]
 For the analysis of stability, we may restrict ourselves to those conjugate to $T$ or $T^{-1}$. Note further that as $Z^{ss}=Z^s$, we either have $\mu^{\L}(j(y),\lambda) >0$ if $y \in j^{-1}(Z^s)$ or $\mu^{\L}(j(y),\lambda)<0$ otherwise. Thus if we can uniformly bound $\mu^{\mathcal{M}}(y,\lambda)$ over all $y \in Y$ and $\lambda = T, T^{-1}$, we can choose $N$ larger and then the stability of $y$ is only determined by the stability of $j(y)$ as desired. Note that here we use $\mu^L(z,g^{-1} T g) = \mu^L(gz, T)$. But that such a bound exists follows immediately from \cite[Proposition 2.14]{git}.
\end{proof}
We now apply this result to the map $j$ constructed in Lemma \ref{Lem:jConstruction}. In order to achieve $Z_d^{ss}=Z_d^s$, we have to restrict to the case $d$ even. 
\begin{Cor} \label{Cor:quotexist1}
 Let $\mathcal{M}$ be any $G$-linearized ample line bundle on $Y_{d,n}$, let $N>0$ as in Lemma \ref{Lem:preimstable} and $\mathcal{M}'=\mathcal{M}\otimes (j^*\L)^{\otimes N}$. Then $Y_{d,n}^{ss}=Y_{d,n}^{s}$ is independent of the choice of $\mathcal{M}$ and it admits a uniform geometric quotient $\phi: Y_{d,n}^s \to M(d,n)$. Here $\phi$ is affine and universally submersive and $M(d,n)$ is a normal projective variety over $\C$. 
\end{Cor}
\begin{proof}
 Let $\mathcal{M}_0$ be an ample line bundle on the projective, normal variety $Y_{d,n}$, then by \cite[Corollary 1.6]{git} some power $\mathcal{M}$ is $G$-linearizable. It is clear from Lemma \ref{Lem:preimstable} that $Y_{d,n}^{ss}=Y_{d,n}^{s}=j^{-1}(Z_d^s)$ is independent of the constructed $\mathcal{M'}$. By \cite[Theorem 1.10]{git} we conclude the existence of the affine, universally submersive uniform geometric quotient $\phi$, reducedness and normality of $M(d,n)$ follow from \cite[0.§2 (2)]{git} and projectivity from the remark above \cite[Converse 1.12]{git}.  
\end{proof}
The spaces $M(d,n)$ (and their generalizations $M(d|d_1, \ldots, d_n)$ defined in the following section) will now be our main object of study. Note that the $G$-equivariant map $j:Y_{d,n} \to Z_d$ descends to a map $\tilde j: M(d,n) \to M_d^s$.

\subsection{Weighted points}
In order to analyze the recursive boundary structure of $M(d,n)$, it will be necessary to generalize the construction above. We will do so by attributing nonnegative rational weights to the marked points in $\overline M_{0,n}(\PP^1 \times \PP^1, (1,d))$. This will not change the action, but it will affect the set of (semi)stable points and hence the geometric quotient. 

Let $d,n \geq 0$, $k \geq 1$ and $\tilde d_1, \ldots, \tilde d_n \in \mathbb{Z}_{\geq 0}$ nonnegative integers such that $d_S= k(d+1) + \sum_{i=1}^n \tilde d_i$ is positive and odd. We write $\textbf{d}^\sim=(d,k|\tilde d_1, \ldots, \tilde d_n)$. 
Consider the map 
\[J : Y_{d,n} \to Z_d  \times (\PP^1)^n , J = j \times (\pi_1 \circ \text{ev}_1) \times \ldots \times (\pi_1 \circ \text{ev}_n)\]
defined in Lemma \ref{Lem:Jconstruction}. On the target 
\[Z_d  \times (\PP^1)^n = \PP(H^0(\PP^1 \times \PP^1, (d,1)))  \times (\PP^1)^n\]
we have the basepoint free line bundle
\[\mathcal{L}_{\textbf{d}^\sim} = \mathcal{O}_Z(k) \boxtimes \mathcal{O}(\tilde d_1) \boxtimes \ldots \boxtimes \mathcal{O}(\tilde d_n).\]
We want to apply Lemma \ref{Lem:preimstable} to obtain an ample linearized line bundle $\mathcal{M}'$ on $Y_{d,n}$ such that a point in $Y_{d,n}$ is (semi)stable iff it maps to a (semi)stable point in $Z_{d} \times (\PP^1)^n$ with respect to the line bundle $\mathcal{L}_{\textbf{d}^\sim}$\footnote{Note in the following that Mumford's numerical criterion and its consequence Lemma \ref{Lem:preimstable} were formulated for ample line bundles. Therefore if some of the numbers $d$ or $d_i$ are zero, we have to modify the map $J$ to leave out the corresponding factors $Z_{d}$ or $\PP^1$ in the target to make the modified line bundle $\mathcal{L}_{\textbf{d}^\sim}$ ample. However, the analysis will not be affected by this so we will ignore this technicallity henceforth.}. Therefore we need to analyze (semi)stability for the action of $\text{SL}_2$ on $Z_{d} \times (\PP^1)^n$ with respect to $\mathcal{L}_{\textbf{d}^\sim}$.
\begin{Lem} \label{Lem:semistableadvanced}
 A point $q=([s],p_1, \ldots, p_n) \in Z_{d} \times (\PP^1)^n$ is semistable with respect to $\mathcal{L}_{\textbf{d}^\sim}$ iff for all $p \in \PP^1$ we have
 \begin{equation} \label{eqn:advancedsemistable}
  \nu_p([s]) + \delta_{p=\text{fix}}([s]) + \sum_{i: p_i = p}  \frac{\tilde d_i}{k} \leq \frac{d+1+\sum_i  \frac{\tilde d_i}{k}}{2},
 \end{equation}
 Here $\nu_p([s])$ is the order of vanishing of $s$ on the cycle $\{p\} \times \PP^1$ or equivalently the multiplicity of a potential vertical section over $p$. The number $\delta_{p=\text{fix}}([s])$ is $1$ if $p$ is a fixed point of the underlying map $\tilde \varphi$ from the horizontal section of $s$ and $0$ otherwise. The point $q$ is stable iff the inequality above is strict for all $p$.
\end{Lem}
\begin{proof}
 The Lemma follows by applying the Hilbert-Mumford numerical criterion \cite[Theorem 2.1]{git}. As in the proof of Lemma \ref{Lem:numstability}, we can restrict to compute $\mu=\mu^{\mathcal{L}_{\textbf{d}^\sim}}(q,\lambda)$ for the diagonal one-parameter subgroup $\lambda \subset \text{SL}_2$. From the proof of \cite[Proposition 2.2]{silverman}, we see that the line bundle $\mathcal{O}_{Z_d}(k)$ on the factor $Z_d$ contributes a summand
 \[k\left(d+1-2\nu_{[1:0]}([s]) + 2\delta_{[1:0]=\text{fix}}([s])\right)\]
 to $\mu$. On the other hand, the line bundle $\mathcal{O}(\tilde d_i)$ on the $i$-th factor $\PP^1$ contributes $\tilde d_i$ for $p_i \neq [1:0]$ and $-\tilde d_i$ for $p_i = [1:0]$. This can be seen by using \cite[Proposition 2.3]{git}. The conditions $\mu \geq 0$ for semistability and $\mu > 0$ for stability then translate to the claimed result by rearranging the terms and dividing by $2k$.
\end{proof}
We remark that multiplying the inequality (\ref{eqn:advancedsemistable}) by $k$, the left side is an integer. So for $k(d+1) + \sum_i \tilde d_i$ odd, the inequality is satisfied iff it is satisfied strictly. Thus semistability is the same as stability. Then we can apply Lemma \ref{Lem:preimstable} as described above to obtain an ample linearized line bundle $\mathcal{M}'$ on $Y_{d,n}$ to define semistability. For brevity in the later text we make the following definition.
\begin{Def} \label{Def:admissible}
 A tuple $\textbf{d}=(d|d_1, \ldots, d_n) \in \mathbb{Z}_{\geq 0} \times (\mathbb{Q}_{\geq 0})^n$ is called admissible if there exists an integer $k \geq 1$ such that all numbers $\tilde d_i = k d_i$ are integers and such that $k(d+1) + \sum_i \tilde d_i$ is odd. 
\end{Def}
We are now able to define the space $M(d|d_1, \ldots, d_n)$.
\begin{Cor} \label{Cor:quotexist2}
 Let $\textbf{d}=(d|d_1, \ldots, d_n)$ be admissible, then the set 
 \[Y_{d,n}^{ss,\textbf{d}}=Y_{d,n}^{s,\textbf{d}} = J^{-1}((Z_{d}\times (\PP^1)^n)^{s,\mathcal{L}_{\textbf{d}^\sim}})\] admits a uniform geometric quotient $\phi: Y_{d,n}^{s,\textbf{d}} \to M(d|d_1, \ldots, d_n)$. Here $\phi$ is affine and universally submersive and $M(d|d_1, \ldots, d_n)$ is a normal projective variety over $\C$. 
\end{Cor}
\begin{proof}
 This is exactly the same proof as for Corollary \ref{Cor:quotexist1}.
\end{proof}
\todoOld{Check rationality makes no trouble}
\section{Examples and properties} \label{Sect:Properties}
\subsection{Examples}
In the following we will study some of the spaces $\overline{M}_{0,n}(\PP^1 \times \PP^1, (1,d))$ and the corresponding quotient spaces $M(d|d_1, \ldots, d_n)$ in more detail. For $B \subset \{1, \ldots, n\}$ and $0 \leq k \leq d$ with $k \geq 1$ or $|B|\geq 2$, we denote by 
\[D_{B,k} = D(\{1, \ldots, n\} \setminus B, (1,d-k)|B,(0,k))  \subset Y_{d,n}\]
the boundary divisor with general point $(f: C \to \PP^1 \times \PP^1; p_1, \ldots, p_n)$ for $C$ having two irreducible components $C_1, C_2$ carrying the markings $\{1, \ldots, n\} \setminus B$ and $B$ and mapping with degrees $(1,d-k), (0,k)$, respectively.
\subsubsection{\texorpdfstring{$d=0, n\leq2$}{d=0, n<=2}}
It is clear that we have an isomorphism
\[\overline{M}_{0,n}(\PP^1 \times \PP^1, (1,0)) \cong \PP^1 \times \overline{M}_{0,n}(\PP^1,1).\]
For small $n$ the space $\overline{M}_{0,n}(\PP^1,1)$ is easy to describe: 
\[\overline{M}_{0,1}(\PP^1,1) \cong \PP^1, \overline{M}_{0,2}(\PP^1,1) \cong (\PP^1)^2,\]
both via the evaluation maps of the markings.
\detex{
\begin{Lem} \label{Lem:M(0,n)}
 The map $\text{ev}= \text{ev}_1 \times \ldots \times \text{ev}_n: \overline{M}_{0,n}(\PP^1,1) \to (\PP^1)^n$ is an isomorphism for $n \leq 2$.
\end{Lem}
\begin{proof}
 For $n \leq 1$ we have $\overline{M}_{0,n}(\PP^1,1) = {M}_{0,n}(\PP^1,1)$ and in general it holds that 
 \[{M}_{0,n}(\PP^1,1) \cong (\PP^1)^n \setminus \Delta,\]
 where $\Delta$ is the union of all diagonals in $(\PP^1)^n$. This finishes the proof for $n=0,1$.
 
 For $n=2$ we have the open set $(\PP^1)^n \setminus \Delta$ corresponding to maps with smooth source curve and the boundary divisor $D(\emptyset,1; \{1,2\},0)$. Then one sees easily that a geometric point $f: C \to \PP^1$ of this boundary divisor is uniquely determined by $\text{ev}_1(f) = \text{ev}_2(f)$. Hence $\text{ev}$ is a bijection on geometric points and by Proposition \ref{Pro:bijectiso} it is an isomorphism.
\end{proof}
For $n \leq 2$ one sees that $\overline{M}_{0,n}(\PP^1 \times \PP^1, (1,0)) \cong (\PP^1)^{n+1}$ and one sees that the action of $\text{SL}_2$ (or $\text{PGL}_2$) is given by the usual componentwise action.} 

Concerning the spaces $M(0|d_1, \ldots, d_n)$ with $n \leq 2$, one checks using Lemma \ref{Lem:semistableadvanced} that the only cases with nonempty semistable sets occur for $n=2$ and the semistable set does not depend on the choice of $d_1,d_2$. Thus these nonempty spaces $M(0|d_1,d_2)$ are all isomorphic to $M(0|1,1)$. 
\begin{Lem} \label{Lem:M(0;1,1)}
 The space $M(0|1,1)$ is isomorphic to a single point $\text{Spec}(\mathbb{C})$, without isotropy.
\end{Lem}
\begin{proof}
 The semistable points in $\overline{M}_{0,2}(\PP^1 \times \PP^1, (0,1)) \cong (\PP^1)^3$ are exactly those where the two marked points have different first coordinates $p_1, p_2$ and both of them are not fixed points of the induced degree $0$ morphism, which is constant equal to $q \in \PP^1$. Under the above ismorphism, this corresponds exactly to the point $(q,p_1, p_2)$. Hence 
 \[\overline{M}_{0,2}(\PP^1 \times \PP^1, (0,1))^{ss,(1,1)} = (\PP^1)^3 \setminus \Delta.\]
 As the action of $\text{PGL}_2$ on $\PP^1$ is strictly $3$-transitive, the quotient is isomorphic to a single point without isotropy.
\end{proof}

\subsubsection{\texorpdfstring{$d=1, n=1$}{d=1,n=1}} \label{Sect:d1n0}
First we note that the only nonempty moduli spaces $M(1|d_1)$ are those with $0<d_1<2$ and all of them are isomorphic to $M(1|1)$ (by an analysis of stable loci using Lemma \ref{Lem:semistableadvanced}).

For $\overline M_{0,1}(\PP^1 \times \PP^1 , (1,1))$ we first look at the locus of stable maps with smooth source curve and find
\[M_{0,1}(\PP^1 \times \PP^1 , (1,1)) \cong \text{PGL}_2 \times \PP^1\]
by Lemma \ref{Lem:Jconstruction}. Here $\text{PGL}_2$ acts on itself by conjugation and on $\PP^1$ in the usual way. There are two boundary divisors on $Y_{1,1}$, namely $D_{\emptyset,1}$ and $D_{\{1\},1}$. 
Then we have the following. \todoOld{Identify as toric stack. also identify universal families}
\begin{Lem} \label{Lem:M(1,1)}
 Consider coordinates $[X:Y],[S:T]$ on $\PP^1 \times \PP^1$ and the rational map 
 \begin{align*}
 \varphi: Z_1=\PP(H^0(\mathcal{O}_{\PP^1 \times \PP^1}(1,1)) &\dashrightarrow \PP^1\\
 [aXS+bXT+cYS+dYT] &\mapsto  [-bc+da:b^2-2bc+c^2].
 \end{align*}
 It is $\text{PGL}_2$-invariant and the image of 
 \[j:\overline M_{0,1}(\PP^1 \times \PP^1 , (1,1))^{ss,(1|1)} \to Z_1\]
 lies in the domain of definition of $\varphi$. Hence $\varphi \circ j$ induces a map $\psi: M(1|1) \to \PP^1$ and this map is an isomorphism. All points in $M(1|1)$ have trivial $\text{PGL}_2$-stabilizers, except for the point corresponding to the orbit of 
 \[f_0 = \left( \left[\begin{matrix}
                 -1 & 0 \\ 0 & 1
                \end{matrix} \right], [1:1]
\right) = ([z \mapsto -z], z=1) \in \text{PGL}_2 \times \PP^1\]
 which has a $\mathbb{Z}/2 \mathbb{Z}$-isotropy given by
 \[B_0=\left[\begin{matrix}
                 0 & 1 \\ 1 & 0
                \end{matrix} \right] = [z \mapsto 1/z] \in \text{PGL}_2.\]
\end{Lem}
\begin{proof}
 From the semistability condition in Lemma \ref{Lem:semistableadvanced} it follows that
 \[M_{0,1}(\PP^1 \times \PP^1 , (1,1))^{ss,(1|1)} = \text{PGL}_2 \times \PP^1 \setminus \underbrace{\{([A],p): Ap = p\}}_{=\text{Fix}}.\]
 We will first study the conjugation action of $\text{PGL}_2$ on itself. Here it is clear that in each orbit of a given $[A]\in \text{PGL}_2$ there is a matrix in Jordan canonical form and by scaling we may pick the representative $A'$ of this matrix in $\text{PGL}_2$ to have entries $\alpha,1$ on the diagonal. 
 
 Now we determine the stabilizer under the conjugation action. If $B \in \text{GL}_2$ satisfies $B A' B^{-1} = \lambda A'$ for some $\lambda \in \mathbb{C}^*$ then taking trace and determinant, we have 
 \begin{align*}
  \alpha + 1 = \text{tr}(A') = \text{tr}(B A' B^{-1})  =  \text{tr}(\lambda A') = \lambda (\alpha +1),\\
  \alpha = \text{det}(B A' B^{-1}) = \text{det}(\lambda A') = \lambda^2 \alpha.
 \end{align*}
 Thus $\lambda = \pm 1$ and $\lambda = 1$ for $\alpha \neq -1$.
 
 We first look for the solutions $B$ of $BA'=A'B$, so the equation for $\lambda=1$. If $A$ is diagonalizable, the matrix $A'$ has the form 
 \[A'=\left( \begin{matrix}
       \alpha & 0\\ 0 & 1
      \end{matrix}\right).
 \]
 Note that for $\alpha= 1$, the induced map $[A]\in \text{PGL}_2$ is the identity. Therefore all points are fixed and thus this case does not occur in the semistable set above. Excluding this case let $B= \left(\begin{matrix}
       a & b\\ c & d
      \end{matrix} \right)$ with $B A' = A' B$. Spelling this out means exactly $\alpha b = b$, $\alpha c = c$. So 
 as $\alpha \neq 1$, we have $b=c=0$, so $B$ is a diagonal matrix. Note that the diagonal matrices in $\text{PGL}_2$ are  isomorphic to $\mathbb{C}^*$. But now we also want to take into account the additional marked point above. The complement of the fixed points of $A'$ are the points $p$ in $\PP^1 \setminus \{[0:1],[1:0]\}$. This eliminates the remaining stabilizing elements in $\mathbb{C}^*$, for instance moving the marked point to $[1:1]$.
 
 On the other hand, for $\alpha=-1$ we have to check the case $BA' = - A'B$. This amounts to $-a=a, -d=d$. Thus in this case the stabilizer of $[A']$ also contains antidiagonal matrices. As above, we can use the $\mathbb{C}^*$-part of the stabilizer to move the marking to $[1:1]$. Of the antidiagonal matrices, exactly the class of the element $B_0=\left( \begin{matrix}
       0 & 1\\ 1 & 0
      \end{matrix}\right)$ fixes the marked point $[1:1]$. This gives precisely the stabilizer of 
 the point $f_0$ above.\\
 Finally there is the case where $A$ is not diagonalizable. Necessarily its eigenvalues have to coincide, amounting to $\alpha=1$ (and thus $\lambda=1$) above, and the matrix $A'$ has the form
 \[A'=\left( \begin{matrix}
       1 & 1\\ 0 & 1
      \end{matrix}\right).
 \]
 The equation $BA' = A'B$ gives $c=0$, $a=d$, so the stabilizer consists of matrices $B= \left(\begin{matrix}
       a & b\\ 0 & a
      \end{matrix} \right)$. By scaling the diagonal to $a=1$, we see that this is exactly isomorphic to
 $\mathbb{C}$. But now the only fixed point of $A'$ is $[1:0]$, so we see that the marked point $p \in \mathbb{C} = \PP^1 \setminus \{[1:0]\}$ eliminates the remaining stabilizer. 
 
 We also have to analyze the boundary. Note that by Lemma \ref{Lem:semistableadvanced}, the entire divisor $D_{\{1\},1}$ is unstable. 
 On the other hand we know from \cite{fultonpandha} that there is a bijective gluing morphism
 \[\overline{M}_{0,2}(\PP^1 \times \PP^1, (1,0)) \times_{\PP^1 \times \PP^1} \underbrace{\overline{M}_{0,1}(\PP^1 \times \PP^1, (0,1))}_{=\PP^1 \times \PP^1} \to D_{\emptyset,1}.\]
 We see that the left side is actually isomorphic to $\overline{M}_{0,2}(\PP^1 \times \PP^1, (1,0))$. One checks that the preimage of the (semi)stable locus $D_{\emptyset,1}^{ss,(1|1)}$ is precisely $\overline{M}_{0,2}(\PP^1 \times \PP^1, (1,0))^{ss,(0|1,1)}$. 
 As seen in the proof of Lemma \ref{Lem:M(0;1,1)} this is exactly the variety $\text{PGL}_2$ and the corresponding action is given by left-multiplication. Hence the boundary divisor $D_{\emptyset,1}$ in $M(1|1)$ is isomorphic to a point (without stabilizer).
 
 Now we are ready to prove that the map $\psi$ is bijective on closed points. An element of the boundary divisor $D_{\emptyset, 1}$ corresponds to a horizontal section $[e:f]\times \PP^1$ and a vertical section $\PP^1 \times [g:h]$ with $[e:f] \neq [g:h]$. Under $j$ it maps to sections of the form
 \[(Xf-Ye)(Sh-Tg) = fh XS - fg XT - eh YS + eg YT \in H^0(\PP^1 \times \PP^1, \mathcal{O}(1,1)).\]
 Thus the image under $\psi$ is 
 \begin{align*}&[-(-fg)(-eh)+(eg)(fh):(fg)^2 - 2 (-fg)(-eh) + (eh)^2]\\ = &[0:(fg-eh)^2] = [0:1],\end{align*}
 where we use $fg-eh \neq 0$ as $[e:f] \neq [g:h]$. 
 
 On the other hand we look at $(\text{PGL}_2 \times \PP^1) \setminus \text{Fix}$.
 Note that $[A] = [(a_{i,j})_{i,j=1}^2]\in \text{PGL}_2$ corresponds to the map 
 \[\PP^1 \to \PP^1 \times \PP^1, [x:y] \mapsto ([x:y],[a_{1,1}x + a_{1,2}y: a_{2,1}x + a_{2,2}y]).\]
 Under $j$ the point $([A],p)$ maps to $[S(a_{2,1}X + a_{2,2}Y)-T(a_{1,1}X + a_{1,2}Y)]$. Hence our proposed map $\psi$ here simply takes the form
 \[\psi([A],p) = [a_{1,1}a_{2,2}-a_{2,1}a_{1,2}:a_{1,1}^2 + 2a_{1,1}a_{2,2} + a_{2,2}^2]= [\text{det}(A): \text{tr}(A)^2].\]
 Note that this map is independent of the choice of representative $A$ in the class $[A] \in \text{PGL}_2$ and also invariant under the conjugation action. Because $j$ is dominant and $\text{PGL}_2$-equivariant, we also obtain that $\varphi$ is $\text{PGL}_2$ invariant.
 
 We know $\text{det}(A)\neq 0$ and by scaling we may reach $\text{det}(A)=1$. The scaling constant is unique up to a sign. On the other hand if we know $\psi([A],p)$ then $\text{tr}(A)$ is determined up to a sign, giving us the eigenvalues of $A$ up to a common factor $\pm 1$. But this specifies the Jordan canonical form (as the identity matrix is excluded). Thus $\psi$ in injective on closed points. On the other hand one sees quickly that every point $[1:t]$ is in the image of $\psi$. Thus by Proposition \ref{Pro:bijectiso}, the space $M(1|1)$ is isomorphic to $\PP^1$ via the map $\psi$. 
\end{proof}
\subsubsection{\texorpdfstring{$d=2, n=0$}{d=2,n=0}}
\begin{Lem} \label{Lem:dn20}
 For $d=2,n=0$ the map $j:Y_{d,n} \to Z_d$ is an isomorphism over $Z_d^s$ and thus the map $\tilde j : M(2,0) \to M_{2}^s$ is also an isomorphism. Hence, the space $M(2,0)$ is isomorphic to $\mathbb{P}^2$.
\end{Lem}
\begin{proof}
 We will show that the map $j:Y_{d,n}^s \to Z_d^s$ is bijective on closed points and then apply Proposition \ref{Pro:bijectiso}. We already know that over $\text{Rat}_d \subset Z_d^s$, this morphism is an isomorphism, hence bijective, by Lemma \ref{Lem:Jconstruction}. By Lemma \ref{Lem:numstability} all remaining points in $Z_d^s \setminus \text{Rat}_d$ are classes $[s]$ of sections $s \in H^0(\PP^1 \times \PP^1, (d,1))$ corresponding to graphs $\Gamma=V(s) \subset \PP^1\times\PP^1$ with one or two vertical sections $v_i$ over $p_i \in \PP^1$ of multiplicity exactly $1$ such that $p_i$ is not a fixed point of the induced map $\tilde \varphi: \PP^1 \to \PP^1$ of degree  $1$ or $0$. But one sees easily that the inclusion $\Gamma \hookrightarrow \PP^1 \times \PP^1$ is a stable map of genus $0$ curves corresponding to a point in $Y_{d,n}^s$ mapping to $[s]$. Thus $j: Y_{d,n}^s \to Z_d^s$ is surjective. On the other hand if $f:C \to \PP^1 \times \PP^1$ is a stable map such that $j(f) = [s]$, one sees that it defines an isomorphism $C \to \Gamma$. 
 Thus up to isomorphism of stable maps, the inclusion $C=\Gamma \hookrightarrow \PP^1 \times \PP^1$ is the only curve in the preimage of $[s]$, which shows injectivity.
\end{proof}

\subsection{Isotropy and Singularities}
In this section we first show that the spaces $M(d|d_1, \ldots, d_n)$ we constructed have nice singularities. This will be important for analyzing their Picard group. In the following, let $\textbf{d}=(d|d_1, \ldots, d_n)$ be admissible.
\begin{Lem} \label{Lem:codimauto}
  Let $d\geq 0, n\geq 0$ with $(d,n) \neq (0,0), (0,1), (1,0)$. Then for the action of $\tilde G= \text{PGL}_2$ let $A_{d,n} \subset \overline M_{0,n}(\PP^1 \times \PP^1, (1,d))$ be the locus of points where $\tilde G$ acts with nontrivial stabilizer. Then $A_{d,n}$ is of codimension at least $1$ for all $d,n$ as above. Even more, for $d \geq 1$ and $(d,n) \neq (1,1), (2,0)$ the set $A_{d,n}$ is of codimension at least $2$ and for $d=0$ we have that 
  \[A_{d,n} \setminus D_{\{1, \ldots, n\}, 0} \subset \overline M_{0,n}(\PP^1 \times \PP^1, (1,d)) \setminus D_{\{1, \ldots, n\}, 0} \]
  is of codimension at least $2$. 
\end{Lem}
\begin{proof}
 Given a fixed $d$, our proof will basically be an induction on $n$. For this, consider the forgetful map 
 \[F: \overline M_{0,n+1}(\PP^1 \times \PP^1, (1,d)) \to \overline M_{0,n}(\PP^1 \times \PP^1, (1,d)).\]
 The fibre over a closed point point $[(f:C \to \PP^1 \times \PP^1;p_1, \ldots, p_n)]$ is exactly isomorphic to $C / \text{Aut}(C)$, corresponding to the possibilites to add another marked point, possibly having to add additional components to $C$ afterwards. It is clear that $F(A_{d,n+1}) \subset A_{d,n}$. By standard theorems on fibre dimensions this shows that 
 \[\text{codim}(A_{d,n+1}) \leq \text{codim}(A_{d,n})\]
 with a strict inequality if for some $(d,n)$ only a finite number of points in the fibre over a general point $[(f:C \to \PP^1 \times \PP^1;p_1, \ldots, p_n)] \in A_{d,n}$ have a $\text{PGL}_2$-isotropy.
 
 Using this we see that for $d\geq 1$ it suffices to prove the assertions above if 
 \[(d,n)=(1,1), (1,2), (2,0),(2,1), (3,0), (4,0) \ldots.\] 
 Our analysis of $A_{d,n}$ will be split in analyzing the interior of the moduli space $\overline M_{0,n+1}(\PP^1 \times \PP^1, (1,d))$ and its boundary. By \cite[Lemma 12]{fultonpandha} the boundary divisor $D_{B,k}$ admits a surjective, birational map from
 \begin{equation}\overline M_{0,n-|B|+1}(\PP^1\times \PP^1, (1,d -k)) \times_{\PP^1 \times \PP^1} \overline M_{0,|B|+1}(\PP^1 \times \PP^1, (0,k)), \label{eqn:interdiv}\end{equation}
 which amounts to gluing two curves and maps at a single point. One sees easily that this gluing map is $\text{PGL}_2$-equivariant with respect to the natural actions. Note that $D_{B,k}$ is irreducible by \cite[Corollary 2]{pandhaconnected} and already has codimension $1$. Thus in order to show that the boundary part of $A_{d,n}$ has codimension at least $2$, we only have to show that a general point in $D_{B,k}$ has trivial stabilizer. But this immediately follows if we know that a general point of $\overline M_{0,n-|B|+1}(\PP^1\times \PP^1, (d -k,1))$ has trivial stabilizer. For almost all $B,k$ this will be implied by our inductive argument below, where we first induce over $d$ and then for fixed $d$ over $n$.  
 The only exceptional case is $B=\{1, \ldots, n\}$ and $k=d$. Here \detex{using Lemma \ref{Lem:M(0,n)}} we have
 \begin{align*}
  D_{B,k} &\cong \overline M_{0,1}(\PP^1\times \PP^1, (1,0)) \times_{\PP^1 \times \PP^1} \overline M_{0,n+1}(\PP^1 \times \PP^1, (0,d))\\
  &\cong (\PP^1 \times \PP^1) \times_{\PP^1 \times \PP^1} (\PP^1 \times \overline M_{0,n+1}(\PP^1, d))\\
  &\cong \PP^1 \times \overline M_{0,n+1}(\PP^1, d).
 \end{align*}
 Under these identifications, $\text{PGL}_2$ acts on the first factor in the usual way and on the second factor by postcomposition. By another inductive argument using a forgetful morphism, it is enough to consider the case $n=0$. We claim that for $d \geq 2$, a general point of the space above has trivial stabilizer.
 
 Indeed let $(q, [\varphi: \PP^1 \to \PP^1; p])$ be a general element of $\PP^1 \times M_{0,n+1}(\PP^1, d)$. Then we may assume that $q \neq \varphi(p)$ and we can use our $\text{PGL}_2$ action to move those points to $[0:1],[1:0]$, respectively. Now assume $\psi \in \text{PGL}_2$ fixes $[0:1],[1:0]$ and satisfies $[\psi \circ \varphi]=[\varphi] \in M_{0,0}(\PP^1,d)$. In other words, there exists $B \in \text{Aut}(\PP^1)$ such that $\psi \circ \varphi = \varphi \circ B$. Then $B$ must fix the preimages $\varphi^{-1}([0:1])$ and $\varphi^{-1}([1:0])$. As $\varphi$ was supposed to be general, these are two sets of $d$ points, all distinct. For $d\geq 2$ one sees that this implies $B=\text{id}$ so $\psi = \text{id}$ as desired.
 
 To conclude, we will now show the claimed results in the order
 \[(d,n) = (0,2), (0,3), (0,4), \ldots; (1,1), (1,2), (2,0), (2,1), (3,0), (4,0), \ldots\]
 where for $d \geq 2$ we only have to consider the points of $A_{d,n}$ in the interior $M_{0,n}(\PP^1 \times \PP^1, (1,d))$ of $Y_{d,n}$.\\
 \underline{$d= 0$ and $n \geq 2$}
 
 It is clear that
 \[M_{0,n}(\PP^1 \times \PP^1, (1,0)) \cong \PP^1 \times ((\PP^1)^n \setminus \Delta)\]
 where $\text{PGL}_2$ acts in the usual way on each component.
 Hence for $n \geq 2$, a general point of this space has trivial stabilizer. For $n \geq 3$ we even have that no point in this open locus has any stabilizer, so it remains to consider the boundary components. For a proper subset $B \subsetneq \{1, \ldots, n\}$ with $|B| \geq 2$, a general element of $D_{B,0}$ will consist of a horizontal inclusion $\PP^1 \to \PP^1 \times \{q\}$ with at least one marking from $\{1, \ldots, n\} \setminus B$ mapping to the point $(p,q)$  and a component of the source curve containing the marks $B$ contracted to a point $(p',q)$. For $q,p,p'$ pairwise distinct this element has trivial stabilizer as desired.\\
 \underline{$d= 1$ and $n =1$}
 
 This case was analyzed in great detail in Lemma \ref{Lem:M(1,1)}, where in particular it is shown that a general point of $M_{0,1}(\PP^1 \times \PP^1, (1,1))$ has trivial stabilizer.\\
 \underline{$d= 1$ and $n =2$}
 
 Here we will look at the fibres of the forgetful map 
 \[F:\overline M_{0,2}(\PP^1 \times \PP^1, (1,1)) \to \overline M_{0,1}(\PP^1 \times \PP^1, (1,1)).\]
 On $M_{0,1}(\PP^1 \times \PP^1, (1,1)) = \text{PGL}_2 \times \PP^1$ we have seen that away from the locus $\text{Fix} \subset \text{PGL}_2 \times \PP^1$, the only orbit with nontrivial stabilizer is the orbit of $f_0 = [(z \mapsto -z, 1)]$, which has finite stabilizer. Therefore all but finitely many points in the fibre $\PP^1 / \text{Aut}(f_0)$ are not fixed points of this stabilizer as desired.
 
 In the closed subset $\text{Fix}$ there are two types of elements: firstly we have $([A],p)$ with $[A]\neq [\text{id}]$ and $p$ one of the finitely many fixed point of $A$. Here as above we see that adding an additional marked point in a sufficiently general position (i.e. different from the fixed points) removes the remaining stabilizer. The other remaining case are the points $([\text{id},p])$ with $p \in \PP^1$ arbitrary. But this locus is already itself of codimension $3$ so its preimage under $F$ also has codimension $3$.
 
 In Lemma \ref{Lem:M(1,1)} we have already seen that a general point on the boundary divisor $D_{\emptyset,1}$ has no isotropy and similarly it is easy to see the same for $D_{\{1\},1}$. Hence $A_{1,1}$ intersected the boundary of $\overline M_{0,1}(\PP^1 \times \PP^1, (1,1))$ has codimension at least $2$.\\
 \underline{$d\geq 2$ and $n =0$}
 
 By our preparations, we need to show that $A_{d,n} \cap M_{0,n}(\PP^1 \times \PP^1, (1,d))$ has codimension at least $1$ for $d=2$ and at least $2$ for $d \geq 3$. But by \cite[Corollary 4]{autolocus} $A_{d,0}$ has codimension $d-1$ in $M_{0,n}(\PP^1 \times \PP^1, (1,d))$. Here we use that $j$ is an isomorphism over $\text{Rat}_d$ by Lemma \ref{Lem:jConstruction}.\\
 \underline{$d = 2$ and $n =1$}
 
 All points in $M_{0,0}(\PP^1 \times \PP^1, (1,2)) \cong \text{Rat}_2$ have finite $\text{PGL}_2$-isotropy by \cite[Proposition 4.65]{MR2316407}. Hence as above we only have finitely many points in each fibre of $F$ with nontrivial isotropy. This finishes the proof.
\end{proof}
\begin{Cor} \label{Cor:codimfree}
 Let $\textbf{d}=(d|d_1, \ldots, d_n)$ be admissible.
 Then $\tilde G= \text{PGL}_2$ acts on $Y_{d,n}^{ss,\textbf{d}}$ with finite\detex{, reduced} stabilizers at geometric points and for $(d,n) \neq (2,0),(1,1)$, the action is free on an invariant open set with complement of codimension at least $2$.
\end{Cor}
\begin{proof}
 By \cite[1.§4 (1)]{git} the function $y \mapsto \dimension(\text{Stab}(y))$ is locally constant on $Y_{d,n}$. But note that $Y_{d,n}$ is connected (see \cite[Corollary 1]{pandhaconnected}). Hence for showing that the action has finite stabilizers on the semistable set, it suffices to find any semistable point with finite stabilizer. One sees quickly that for all $\textbf{d}$ as above such that there are semistable points at all, the locus $Y_{d,n}^{ss,\textbf{d}} \setminus A_{d,n}$, where $\tilde G$ acts freely, is nonempty by Lemma \ref{Lem:codimauto}. Thus every geometric point has finite\detex{, reduced} stabilizer\detex{ \footnote{To be very precise, we showed that the stabilizer $\text{Stab}(y)$ of a geometric point $y$ is quasi-finite over $\text{Spec}(\mathbb{C})$. But every quasi-finite scheme over the spectrum of a field is finite and every separated, finite type group scheme over a field of characteristic $0$ is reduced.}}.
 
 Moreover, Lemma \ref{Lem:codimauto} immediately implies that $A_{d,n} \cap Y_{d,n}^{ss,\textbf{d}}$ is of codimension at least $2$ for $(d,n) \neq (2,0),(1,1)$ and $d \geq 1$. Finally for $d=0$ we see that $D_{\{1, \ldots, n\}, 0}$ is always disjoint from the locus of semistable points by Lemma \ref{Lem:semistableadvanced}, so again we can apply Lemma \ref{Lem:codimauto}.
\end{proof}

Denote by
\[\mathcal{Y}_{d,n} = \overline{ \mathcal{M}}_{0,n}(\PP^1 \times \PP^1, (1,d))\]
the moduli stack of stable maps to $\PP^1 \times \PP^1$ with coarse moduli space $\mathcal{Y}_{d,n} \to Y_{d,n}$. Let $\mathcal{Y}_{d,n}^{ss,\textbf{d}}$ be the preimage of the locus of semistable points $Y_{d,n}^{ss,\textbf{d}}$. 
\begin{Lem} \label{Lem:Mfinitequotsing}
 The space $M(d|d_1, \ldots, d_n)$ is the coarse moduli space of the smooth Deligne-Mumford stack
 \begin{equation*}
  \mathcal{M}(d|d_1, \ldots, d_n) = \mathcal{Y}_{d,n}^{ss,\textbf{d}} / \text{PGL}_2.
 \end{equation*}
 In particular, it has finite quotient singularities.
\end{Lem}
\begin{proof}
 As $M(d|d_1, \ldots, d_n)$ is a GIT quotient of $Y_{d,n}^{ss,\textbf{d}}=Y_{d,n}^{s,\textbf{d}}$ and as $Y_{d,n}^{ss,\textbf{d}}$ is a coarse moduli space for $\mathcal{Y}_{d,n}^{ss,\textbf{d}}$, we obtain that $M(d|d_1, \ldots, d_n)$ is a coarse moduli space for the quotient stack $\mathcal{Y}_{d,n}^{ss,\textbf{d}} / \text{PGL}_2$ by Lemma \ref{Lem:quotcoarsemod} and Remark \ref{Rmk:quotcoarsemod}. To show that $\mathcal{M}(d|d_1, \ldots, d_n)$ is a smooth Deligne-Mumford stack, we want to apply Proposition \ref{Pro:smoothDMquotient}. The fact that $\mathcal{Y}_{d,n}$ is an orbifold is proved in \cite{fultonpandha} and \cite{pandhaconnected}. The action of $\text{PGL}_2$ on $\mathcal{Y}_{d,n}^{ss,\textbf{d}}$ has finite\detex{, reduced} stabilizers by Corollary \ref{Cor:codimfree}. Thus, the conditions of Proposition \ref{Pro:smoothDMquotient} are verified and the proof of the first part is finished. Finally, $M(d|d_1, \ldots, d_n)$ has finite quotient singularities as it is normal and the coarse moduli space of a smooth Deligne-Mumford stack (see \cite[Proposition 2.8]{vistoliintersection}).
\end{proof}
\begin{Cor} \label{Cor:QCartier}
 Every Weil divisor on $M(d|d_1, \ldots, d_n)$ is $\mathbb{Q}$-Cartier.
\end{Cor}
\begin{proof}
 This follows as $M(d|d_1, \ldots, d_n)$ has at most finite quotient singularities (see \cite[Proposition 5.15]{moribirat}). 
\end{proof}

\subsection{Modular interpretations}
Let still $\textbf{d} = (d|d_1, \ldots, d_n)$ be admissible. We want to find a way to interpret the quotient $M(d|d_1, \ldots, d_n)$ as a moduli space for a functor of stable self-maps, in a sense yet to be defined. Our strategy is as follows: first we identify the stack $\mathcal{Y}_{d,n}/ \text{PGL}_2$ with an explicit category fibred in groupoids. Its objects over a scheme $S$ are a natural generalization of families with fibres $\PP^1$ together with $n$ markings and a degree $d$ self-map of the family. We can identify $\mathcal{M}(d|d_1, \ldots, d_n)$ as an open substack defined by requiring that the self-maps satisfy a stability condition. By passing from this stack to its coarse moduli space $M(d|d_1, \ldots, d_n)$, we obtain the functor from schemes to sets, which associates to $S$ the set of isomorphism classes of families in $\mathcal{M}(d|d_1, \ldots, d_n)(S)$. Finally, we see that over the locus $M^*$ of points $[p] \in M(d|d_1, \ldots, d_n)$ without $\text{PGL}_2$-isotropy or automorphisms (of $p \in Y_{d,n}$), we have a universal family of curves with a self-map (in the category of schemes).
Let $\mathcal{M}^{d,n}$ be the category fibred in groupoids over $\text{Sch}_{\mathbb{C}}$, whose objects over a scheme $S$ are diagrams
\begin{equation}  \label{eqn:selfmapfam}
\begin{tikzpicture}
 \draw (0,2) node{$\mathcal{C}$};
 \draw[->] (-0.1,1.7) -- (-0.1,1) node[left]{$\pi$} -- (-0.1,0.3);
 \draw[->] (0.1,0.3) -- (0.1,1) node[right]{$\sigma_i$} -- (0.1,1.7);
 \draw (0,0) node{$S$};
 \draw[->] (0.25,2.05) --(1,2.05) node[above]{$\mu$} -- (1.8,2.05);
 \draw[->] (0.25,1.95) --(1,1.95) node[below]{$\varphi$} -- (1.8,1.95);
 \draw (2,2) node{$\mathcal{\tilde C}$};
 \draw[->] (1.7,1.7) -- (1,1) node[right]{$\tilde \pi$} -- (0.3,0.3);
\end{tikzpicture}
\end{equation}
where 
\begin{itemize}
 \item $\pi, \tilde \pi$ are flat, projective families of quasi-stable genus $0$ curves, all geometric fibres of $\tilde \pi$ are isomorphic to $\PP^1$
 \item $\mu, \varphi$ satisfy $\mu_*([\mathcal{C}_s]) = [\mathcal{\tilde C}_s]$ and $\varphi_*([\mathcal{C}_s]) = d [\mathcal{\tilde C}_s]$ for all geometric points $s \in S$
 \item every component of a geometric fibre $\mathcal{C}_s$, which is contracted by both $\mu_s, \varphi_s$ contains at least three special points (nodes or markings)
\end{itemize}
The morphisms between $\mathcal{F} \in \mathcal{M}^{d,n}(S)$ and $\mathcal{F}' \in \mathcal{M}^{d,n}(S')$ lying over a morphism $\psi: S \to S'$ of schemes are exactly the pullback-diagrams identifying $\mathcal{F}$ as the base change of the diagram $\mathcal{F}'$ by $\psi$.
\begin{Theo} \label{Theo:modprop}
 There is a natural isomorphism between the quotient stack $\mathcal{Y}_{d,n}/ \text{PGL}_2$ and $\mathcal{M}^{d,n}$. Here, an element of $\left( \mathcal{Y}_{d,n}/ \text{PGL}_2 \right)(S)$ coming from a family 
 \[(\pi: \mathcal{C} \to S; \sigma_1, \ldots, \sigma_n : S \to \mathcal{C}; (f_1, f_2): \mathcal{C} \to \PP^1 \times \PP^1) \in \mathcal{Y}_{d,n}(S)\]
 is identified with the diagram
 \begin{equation*} 
\begin{tikzpicture}
 \draw (0,2) node{$\mathcal{C}$};
 \draw[->] (-0.1,1.7) -- (-0.1,1) node[left]{$\pi$} -- (-0.1,0.3);
 \draw[->] (0.1,0.3) -- (0.1,1) node[right]{$\sigma_i$} -- (0.1,1.7);
 \draw (0,0) node{$S$};
 \draw[->] (0.25,2.05) --(1,2.05) node[above]{$\pi \times f_1$} -- (1.8,2.05);
 \draw[->] (0.25,1.95) --(1,1.95) node[below]{$\pi \times f_2$} -- (1.8,1.95);
 \draw (2.5,2) node{$S \times \PP^1$};
 \draw[->] (1.7,1.7) -- (1,1) node[right]{} -- (0.3,0.3);
\end{tikzpicture}
\end{equation*}
in $\mathcal{M}^{d,n}(S)$.
\end{Theo}
\begin{proof}
 We give a natural identification of the objects of $\mathcal{Y}_{d,n}/ \text{PGL}_2$ and $\mathcal{M}^{d,n}$ over a scheme $S$, which respects isomorphisms of these objects. Let $G=\text{PGL}_2$ for the remainder of the proof.
 
 Using \cite[Theorem 4.1]{romagny2005}, we know that an object of $\mathcal{Y}_{d,n}/ \text{PGL}_2$ over $S$ is nothing but a $G$-torsor $\psi: T \to S$ over $S$ together with a morphism $(f, \sigma): T \to \mathcal{Y}_{d,n}$ of $G$-stacks. Note that as $G$ is quasi-affine, the torsor $T$ is actually a scheme (not an algebraic space), see \cite[VII, Corollaire 7.9]{sga1}. The map $f: T \to \mathcal{Y}_{d,n}$ is just a usual (1)-morphism, so it is equivalent to specifying the data 
 \begin{equation} \label{eqn:equivarfamiliy}
  \begin{tikzpicture}
 \draw (0,2) node{$\mathcal{C}$};
 \draw[->] (-0.1,1.7) -- (-0.1,1) node[left]{$\pi$} -- (-0.1,0.3);
 \draw[->] (0.1,0.3) -- (0.1,1) node[right]{$s_i$} -- (0.1,1.7);
 \draw (0,0) node{$T$};
 \draw[->] (0.25,2) --(1,2) node[above]{$(f_1,f_2)$} -- (1.8,2);
 \draw (2.5,2) node{$\PP^1 \times \PP^1$};
\end{tikzpicture}
 \end{equation}
 of a family of stable maps to $\PP^1 \times \PP^1$ with degree $(1,d)$ and $n$ markings. The 2-morphism $\sigma$ relates the two paths in the commutative diagram
 \[\begin{CD}
    G \times T @> \text{id}_G \times \psi >> G \times \mathcal{Y}_{d,n}\\
    @V \mu_T VV @V \mu_Y VV\\
    T @> \psi >> \mathcal{Y}_{d,n}
   \end{CD},
 \]
 where the two vertical arrows are the corresponding actions. Now the morphism $\psi \circ \mu_T$ corresponds to the data of the family 
 \[(\mu_T^* \pi: \mu_T^* \mathcal{C} \to G \times T; \mu_T^* s_i : G \times T \to \mu_T^* \mathcal{C}; (f_1, f_2): \mu_T^* \mathcal{C} \to \PP^1 \times \PP^1)\]
 of stable maps. On the other hand $\mu_Y \circ (\text{id}_G \times \psi)$ corresponds to the data
 \[(\text{id}_G \times \pi: G \times \mathcal{C} \to G \times T; \text{id}_G \times  s_i : G \times T \to G \times \mathcal{C}; (g,c) \mapsto (g.f_1(c), g.f_2(c)) ).\]
 Giving a 2-morphism $\sigma$ between $\psi \circ \mu_T$ and $\mu_Y \circ (\text{id}_G \times \psi)$ is thus equivalent to an isomorphism between these two families. One sees that the necessary data is a morphism
 \[\eta : G \times \mathcal{C} \to \mathcal{C}\]
 making the diagram
 \begin{equation} \label{eqn:equiCdiag}
 \begin{CD}
    G \times \mathcal{C} @> \eta >> \mathcal{C}\\
    @V \text{id}_G \times \pi VV @V \pi VV\\
    G \times T @> \mu_T >> T
   \end{CD}
\end{equation}
 cartesian. Moreover, the map $\eta$ should be compatible with the sections $s_i$ and the maps to $\PP^1 \times \PP^1$.
 
 Now for $(f,\sigma)$ to be a 1-morphism of $G$-stacks, the map $\sigma$ needs to satisfy a compatibility condition. For any scheme $Z$ and $g,h \in G(Z), x \in T(Z)$ let 
 \[\sigma_g^x : g. \psi(x) \to \psi(g.x)\]
 be the isomorphism of families of stable maps over $Z$ given by $\sigma$. Then we require
 \[\sigma_g^{h.x} \circ \left(g . \sigma_h^x \right) = \sigma_{gh}^x.\]
 Unwinding the definitions, one sees that this is exactly equivalent to the map $\eta$ satisfying $\eta(g,\eta(h,x)) = \eta(gh,x)$. As the diagram (\ref{eqn:equiCdiag}) is cartesian, we know that the identity element of $G$ acts by an isomorphism. Hence, the compatibility condition exactly requires $\eta$ to be a group action. Before we summarize, we note that assuming $\eta$ is a group action, the diagramm (\ref{eqn:equiCdiag}) is automatically cartesian if it is commutative.
 
 We have seen above that an object of $\mathcal{Y}_{d,n}/ \text{PGL}_2$ over $S$ is equivalent to a $G$-torsor $\psi : T \to S$, a family (\ref{eqn:equivarfamiliy}) of stable maps and an action of $G$ on $\mathcal{C}$ making all the arrows in the diagram (\ref{eqn:equivarfamiliy}) $G$-equivariant. We now rearrange this data slightly as the following diagram
 \begin{equation} \label{eqn:diagone} \begin{tikzpicture}
 \draw (0,2) node{$\mathcal{C}$};
 \draw[->] (-0.1,1.7) -- (-0.1,1) node[left]{$\pi$} -- (-0.1,0.3);
 \draw[->] (0.1,0.3) -- (0.1,1) node[right]{$s_i$} -- (0.1,1.7);
 \draw (0,0) node{$T$};
 \draw[->] (0.25,2.05) --(1,2.05) node[above]{$\pi \times f_1$} -- (1.8,2.05);
 \draw[->] (0.25,1.95) --(1,1.95) node[below]{$\pi \times f_2$} -- (1.8,1.95);
 \draw (2.4,2) node{$T \times \PP^1$};
 \draw[->] (1.7,1.7) -- (1,1) node[right]{$\tilde \pi$} -- (0.3,0.3);
\end{tikzpicture}
 \end{equation}
 and then take the quotient by $G$ at every point of the diagram. As all $G$-actions are free, the quotients are again algebraic spaces. Also, as $T$ is a $G$-torsor over $S$, we have $T/G=S$. Thus the quotient diagram looks like
 \begin{equation} \label{eqn:diagtwo}\begin{tikzpicture}
 \draw (0,2) node{$\mathcal{C}'$};
 \draw[->] (-0.1,1.7) -- (-0.1,1) node[left]{$\pi'$} -- (-0.1,0.3);
 \draw[->] (0.1,0.3) -- (0.1,1) node[right]{$s_i'$} -- (0.1,1.7);
 \draw (0,0) node{$S$};
 \draw[->] (0.25,2.05) --(1,2.05) node[above]{$g_1$} -- (1.8,2.05);
 \draw[->] (0.25,1.95) --(1,1.95) node[below]{$g_2$} -- (1.8,1.95);
 \draw (2,2) node{$\mathcal{\tilde C}$};
 \draw[->] (1.7,1.7) -- (1,1) node[right]{$\tilde \pi'$} -- (0.3,0.3);
\end{tikzpicture}\end{equation}
 and all maps between the quotients keep their respective fppf-local properties (flat, proper, finite type, etc.) because of Remark \ref{Rmk:quotinherit}. To prove that $\pi', \tilde \pi'$ are still projective, we will construct $G$-linearized relatively ample line bundles on $\mathcal{C},T \times \PP^1$ and then use a descent argument. The ingredients for these line bundles are the following:
 \begin{itemize}
  \item The relative dualizing sheaf $\omega_{\pi'}$ carries a natural $G$-action (this follows easily from the compatibility of the formation of relative dualizing sheaves with base change).
  \item The line bundle $\mathcal{O}_{\PP^1}(2)$ has a natural $G$-linearization and hence its pullbacks via $f_1, f_2$ and via the projection $T \times \PP^1 \mapsto \PP^1$ have induced $G$-linearizations.
  \item On $\mathcal{C}$ we also want to have line bundles corresponding to the Weyl divisors $s_i(T)$. To assure that they carry a $G$-linearization, we will make a detour via the quotient $\mathcal{C}'=\mathcal{C}/G$, because $\text{Pic}_G(\mathcal{C}) = \text{Pic}(\mathcal{C}')$. Now we have that $s_i' : S \to \mathcal{C}'$ is a section of the finite type, flat, separated morphism $\pi'$. If we knew that $\mathcal{C}'$ was a scheme, then Proposition \ref{Pro:SectionEmbedding} would imply that $s_i'$ is a Cartier divisor, hence it would give rise to our desired G-linearized line bundle $\mathcal{O}(s_i)$ on $\mathcal{C}$. However, being a Cartier divisor can be checked fpqc-locally. This is because a Cartier divisor is the same as a Koszul-regular immersion of codimension $1$ (see \cite[\href{http://stacks.math.columbia.edu/tag/061T }{Tag 061T} ]{stacks} for a definition) and this notion is fpqc-local on the target (see  \cite[\href{http://stacks.math.columbia.edu/tag/0694}{Tag 0694} ]{stacks}). But to check this, we again pull back via the map $\mathcal{C} \to \mathcal{C}'$ and now Proposition \ref{Pro:SectionEmbedding} really shows that $s_i$, which are the pullbacks of the maps $s_i'$, are Cartier divisors.
 \end{itemize}
 Clearly $\mathcal{O}_{\PP^1}(2)$ is $\tilde \pi$-relatively ample on $T \times \PP^1$. On the other hand, as the family $\pi: \mathcal{C} \to T$ was supposed to be stable, the bundle $\omega_{\pi}(s_1+ \ldots + s_n) \otimes f_1^* \mathcal{O}(l) \otimes f_2^* \mathcal{O}(m)$ is $\pi$-ample for $l,m$ sufficiently large.
 
 But then by \cite[VIII,Proposition 7.8]{sga1}, the spaces $\mathcal{C}', \mathcal{\tilde C}$ are actually schemes and the line bundles above descend to relatively ample line bundles on them for the morphisms $\pi', \tilde \pi'$ \todoOld{Correct interpretation?}. As these morphisms are also finite type and proper, they are projective.
 
 Note further that the geometric fibres of $\tilde \pi'$ are isomorphic to $\PP^1$. Hence we have obtained a family in $\mathcal{M}^{d,n}(S)$. Conversely, for such a family as above, we can consider the morphism
 \[\psi : T=\text{Isom}_S(\mathcal{\tilde C}, S \times \PP^1) \to S.\]
 Here for schemes $X,Y$ over $S$, the scheme $\text{Isom}_S(X,Y)$ represents the functor
 \[\text{Sch}_S \to \text{Sets}, S' \mapsto \{\psi : X_{S'} \to Y_{S'} : \psi_s\text{ isomorphism } \forall s \in S'\}.\]
 The existence of this scheme follows from the fact that $\pi', \tilde \pi'$ are proper and that $\pi'$ is  flat\detex{\footnote{For a good overview see Proposition 2.2 and Lemma 2.3 in \cite{ossermansemquot}.}}. We want to show that $\psi$ is a $G$-torsor. It clearly has a $G$-action by postcomposition. On the other hand, the map $\tilde \pi'$ is smooth. Thus it has sections \'etale locally and hence is a trivial $\PP^1$-bundle after an \'etale base change $S' \to S$. But then the pullback of $T$ to $S'$ is simply $S' \times \text{PGL}_2$, a trivial $G$-torsor.
 
 Now after pulling back the diagram (\ref{eqn:diagtwo}) above by the map $\psi$, we can clearly trivialize the $\PP^1$-bundle $\psi^* \mathcal{\tilde C}$. Indeed, by definition the scheme $T$ parametrizes all possible ways to perform this trivialization fibrewise. Now we have again a diagram as in (\ref{eqn:diagone}) and the action of $G$ on $T$ lifts to an action of the other spaces under the pullback $T \to S$. This is not canonical, but it exactly mirrors the choice we had when translating from the 2-morphism $\sigma$ above to an action map $\eta: G \times \mathcal{C} \to \mathcal{C}$. 
 
 One verifies that the two operations we described are inverse to one-another and one checks, that they respect the isomorphisms inherent in the objects of $\mathcal{Y}_{d,n}/ \text{PGL}_2$ and $\mathcal{M}^{d,n}$ over $S$. Here one uses that starting with a diagram (\ref{eqn:diagone}), the $G$-torsors $T$ and $\text{Isom}_S(\mathcal{\tilde C}, S \times \PP^1)$ over $S$ are isomorphic. This is because there exists a $G$-equivariant map $T \to  \text{Isom}_S(\mathcal{\tilde C}, S \times \PP^1)$ by the universal property of  $\text{Isom}_S(\mathcal{\tilde C}, S \times \PP^1)$. 
 
 The explicit description of the image of a family in $\mathcal{Y}_{d,n}/ \text{PGL}_2$ coming from a family in $\mathcal{Y}_{d,n}$ is obvious from the above construction. Hence, the proof is finished.
\end{proof}

\begin{Cor}
 For $\textbf{d}=(d|d_1, \ldots, d_n)$ admissible, we have that the scheme $M(d|d_1, \ldots, d_n)$ is a coarse moduli space for the functor $\text{Sch}_{\mathbb{C}} \to \text{Sets}$, associating to a scheme $S$ the set of isomorphism classes of diagrams (\ref{eqn:selfmapfam}) in $\mathcal{M}^{d,n}$ such that for all $s \in S$ the map 
 \[(\mathcal{C}_s; \sigma_1(s), \ldots, \sigma_n(s)) \to \mathcal{\tilde C}_s \times \mathcal{\tilde C}_s\]
 is semistable with respect to $\textbf{d}$ as in Lemma \ref{Lem:semistableadvanced}.
\end{Cor}
\begin{proof}
 The substack $\mathcal{Y}_{d,n}^{ss,\textbf{d}} \subset \mathcal{Y}_{d,n}$ is open and $G$-invariant. By Lemma \ref{Lem:quotcoarsemod}, the coarse moduli space of its quotient $\mathcal{Y}_{d,n}^{ss,\textbf{d}}/G$ is exactly $M(d|d_1, \ldots, d_n)$. On the other hand, Theorem \ref{Theo:modprop} identifies this quotient with the open substack of $\mathcal{M}^{d,n}$ whose objects are exactly those described above. But then the functor sending $S$ to the set of isomorphism classes of such objects naturally has $M(d|d_1, \ldots, d_n)$ as a coarse moduli space. \todoOld{Check with Jakob}
\end{proof}

\begin{Rmk}
 If in the diagram (\ref{eqn:selfmapfam}), the fibres $\mathcal{C}_s$ of $\mathcal{C}$ are smooth, the map $\mu$ is an isomorphism and thus $\mathcal{C} \cong \mathcal{\tilde C}$. Thus in this case, the scheme $S$ really parametrizes a family of self-maps $\mathcal{C}_s \to \mathcal{C}_s$ of degree $d$. Thus the locus $M(d|d_1, \ldots, d_n)^o$ of classes of maps with smooth source curve is a coarse moduli space for the functor 
 \begin{align*}
  \text{Sch}_{ \mathbb{C}} &\to \text{Sets}\\
  S &\mapsto \left\{ \vcenter{ \hbox{\begin{tikzpicture}
 \draw (0,2) node{$\mathcal{C}$};
 \draw[->] (0.3,2.05) to [out=390,in=330,looseness=12] (0.3, 1.95);
 \draw (0.9,2) node{$\phi$};
 \draw[->] (-0.1,1.7) -- (-0.1,1) node[left]{$\pi$} -- (-0.1,0.3);
 \draw[->] (0.1,0.3) -- (0.1,1) node[right]{$\sigma_i$} -- (0.1,1.7);
 \draw (0,0) node{$S$};
\end{tikzpicture} }}  : \begin{array}{l}\pi \text{ flat, projective family of}\\\text{smooth genus }0\text{ curves,} \\ \sigma_1, \ldots, \sigma_n \text{ disjoint sections of }\pi\\ \phi\text{ self-map over }S\text{ with }\\ \phi_* [\mathcal{C}_s] = d [\mathcal{C}_s]\text{ for all }s \in S \end{array}  \right\}/\text{iso}.
\end{align*}
Here, an isomorphism of families as above is an isomorphism $\mathcal{C}_1 \to \mathcal{C}_2$ making all diagrams commute.

Note that if some of the $d_i$ are greater than $(d-1+\sum_{j=1}^n d_i)/2$, we need to require above that $\sigma_i(s)$ is not a fixed point of $\phi_2$ for all $s \in S$.
\end{Rmk}
\begin{Rmk} \label{Rem:coarsefamily}
 The forgetful map $F: Y_{d,n+1} \to Y_{d,n}$ of the last marked point is $G$-equivariant and the preimage of $Y_{d,n}^{ss,\textbf{d}}$ is $Y_{d,n+1}^{ss,(\textbf{d},0)}$, with $(\textbf{d},0) = (d|d_1, \ldots, d_n,0)$. Thus $F$ induces a map 
 \[F: M(d|d_1, \ldots, d_n,0) \to M(d|d_1, \ldots, d_n).\] 
 Similarly one constructs sections 
 \[\sigma_i : M(d|d_1, \ldots, d_n) \to M(d|d_1, \ldots, d_n,0)\]
 of $F$. We note that over the locus $M^*$ of points $[f] \in M(d|d_1, \ldots, d_n)$ with trivial $G$-isotropy and such that $f \in Y_{d,n}$ has no automorphisms, the family $F$ and the sections $\sigma_i$ are exactly the universal family $\pi: \mathcal{C}|_{M^*} \to M^*$ from above. This follows because by \cite{fultonpandha}, the locus $Y_{d,n}^*$ of points without automorphisms in $Y_{d,n}$ carries the restriction of the forgetful map from $Y_{d,n+1}$ as a universal family. On the other hand, the geometric quotient of a scheme by a free action is isomorphic to the corresponding quotient stack.
\end{Rmk}

\subsection{Rationality}
In this section, we show that the spaces $M(d|d_1, \ldots, d_n)$ are rational. First note that for different weights $d_1, \ldots, d_n$, all (nonempty) spaces $M(d|d_1, \ldots, d_n)$  are canonically birational. Indeed, they are categorical quotients for varying open, invariant subsets of $Y_{d,n}$, hence they are isomorphic over the image of the intersection of those subsets. Thus, given $d,n$ we can choose the weights arbitarily without changing the birational class of $M(d|d_1, \ldots, d_n)$. This also shows that they are birational to a categorical quotient of some nonempty, $G=\text{PGL}_2$-invariant subset of $\text{Rat}_d \times ((\PP^1)^n \setminus \Delta)$ by $G$, if this quotient exists.
\begin{Theo} \label{Theo:rational}
 The spaces $M(d|d_1, \ldots, d_n)$ are rational.
\end{Theo}
\begin{proof}
 Our argument will basically be an induction on $n$ for every fixed $d \geq 0$. In the case $n=0$, Levy shows that $\text{Rat}_d \sslash G$ is rational for $d \geq 2$ (see \cite[Theorem 4.1]{levy}).
 
 For $n=1$ we must now also cover the case $d=1$, but then we have seen that all nonempty moduli spaces $M(1|d_1)$ are isomorphic to $M(1|1) \cong \PP^1$, which is rational.
 
 For $d \geq 2$ consider the open subset
 \[\mathcal{U} = \{(f,p): p \neq f(p) \neq f(f(p)) \neq p\} \subset \text{Rat}_d \times \PP^1.\]
 By the definition of the action of $G$, this subset is invariant. But on $\mathcal{U}$ we can use the action of $G$ to move $p$ to $0=[0:1]$, $f(p)$ to $\infty=[1:0]$ and $f(f(p))$ to $1=[1:1]$ in a unique way. More precisely, let
 \[\mathcal{U}_0 = \{(f,0): f(0)= \infty, f(\infty) = 1\} \subset \mathcal{U}.\]
 Then the action map restricted to $G \times \mathcal{U}_0$ is an isomorphism onto $\mathcal{U}$. Thus $\mathcal{U} \sslash G = \mathcal{U}_0$. But the conditions $f(0)=\infty, f(\infty)=1$ give linear conditions on the coefficients of $f \in \text{Rat}_d \subset \PP^{2d+1}$ and hence $\mathcal{U}_0$ is rational.
 
 Finally, for $n \geq 2$ we note that the case $d=0,n=2$ is clear, as the only nonempty moduli spaces here are isomorphic to $M(0|1,1)$ (which is a point by Lemma \ref{Lem:M(0;1,1)}). By the discussion preceding the Theorem, we may restrict to the cases $M=M(0|1,1,0, \ldots, 0)$, $M=M(d|0, \ldots, 0)$ for $d$ even and $M=M(d|1,0,\ldots, 0)$ for $d$ odd. But by the forgetful map $F : M \to \tilde M$ of the last marked point (which carries weight $0$), these spaces map to the corresponding spaces $\tilde M$ with one mark less. These are rational by induction. Using Remark \ref{Rem:coarsefamily}, the forgetful maps are flat, projective $\PP^1$-fibrations over the locus in $\tilde M$ parametrizing self-maps of smooth curves without isotropy or automorphisms. Moreover, as $n\geq 2$, the map $F$ has a section $\sigma_1$ (whose image is the divisor $D_{\{1,n\},0}$). But any flat $\PP^1$-fibration with a section is actually locally trivial (see \cite[Proposition 25.3]{hartdeformation}), so indeed $M$ is rational.
%
\end{proof}

\section{The Picard group of \texorpdfstring{$M(d|d_1, \ldots, d_n)$}{M(d|d1, ..., dn)}} \label{Sect:Picard}
\begin{Lem} \label{Lem:Picisomorph}
 Let $\tilde G = \text{PGL}_2$ act freely on a normal variety $X$ and let $\phi: X \to X'$ be a geometric quotient. Then the map \[\phi^* : \text{Pic}(X') \otimes_{\mathbb{Z}} \mathbb{Q} \to \text{Pic}(X) \otimes_{\mathbb{Z}} \mathbb{Q}\] is an isomorphism.
\end{Lem}
\begin{proof}
 The map $\phi^*$ factors as the composition 
 \[\text{Pic}(X') \otimes_{\mathbb{Z}} \mathbb{Q} \to \text{Pic}^{\tilde G}(X) \otimes_{\mathbb{Z}} \mathbb{Q} \to \text{Pic}(X) \otimes_{\mathbb{Z}} \mathbb{Q}\]
 By \cite[1.§3]{git} the first map is an isomorphism (even before tensoring with $\mathbb{Q}$). By \cite[Thereom 7.1, Exercise 7.2]{dolgachev} the second map is injective (using that $\text{PGL}_2$ has only trivial characters as the same is true for $\text{SL}_2$). Finally \cite[Corollary 7.2]{dolgachev} implies that the right map is surjective.
\end{proof}

\begin{Cor} \label{Cor:Picidentification}
 Let $\textbf{d}=(d|d_1, \ldots, d_n)$ be admissible with $(d,n) \neq (2,0)$, $(1,1)$. Then the quotient map 
 \[\phi: Y_{d,n}^{ss,\textbf{d}} \to M(d|d_1, \ldots, d_n)\]
 induces an isomorphism 
 \[\phi^* : \text{Pic}(M(d|d_1, \ldots, d_n)) \otimes_{\mathbb{Z}} \mathbb{Q} \to \text{Pic}(Y_{d,n}^{ss,\textbf{d}}) \otimes_{\mathbb{Z}} \mathbb{Q}\]
 and we have $\text{Pic}(Y_{d,n}^{ss,\textbf{d}}) \otimes_{\mathbb{Z}} \mathbb{Q} \cong \text{Cl}(Y_{d,n}^{ss,\textbf{d}}) \otimes_{\mathbb{Z}} \mathbb{Q}$.
\end{Cor}
\begin{proof}
 By Corollary \ref{Cor:codimfree} we know that there is an open, $\tilde G$-invariant subset $Y^f \subset Y_{d,n}^{ss,\textbf{d}}$ with complement of codimension at least $2$, on which $\tilde G$ acts freely. Let $M^f \subset M(d|d_1, \ldots, d_n)$ be the image of $Y^f$ under $\phi$. Then as $\phi$ is a geometric quotient, this set is open with $Y^f=\phi^{-1}(M^f)$. But this implies that $M^f$ has complement of codimension at least $2$, too. As $Y_{d,n}$ and $M(d,n)$ only have finite quotient singularities, their corresponding rational Class groups and Picard groups coincide, and using that the Class group does not change when removing sets of codimension at least $2$ we obtain
 \begin{align*}
  \text{Pic}(Y_{d,n}^{ss,\textbf{d}}) \otimes \mathbb{Q} &\cong \text{Cl}(Y_{d,n}^{ss,\textbf{d}}) \otimes \mathbb{Q} \cong \text{Cl}(Y^f) \otimes \mathbb{Q}\cong \text{Pic}(Y^f) \otimes \mathbb{Q}\\
  \text{Pic}(M(d|d_1, \ldots, d_n)) \otimes \mathbb{Q} &\cong \text{Cl}(M(d|d_1, \ldots, d_n)) \otimes \mathbb{Q} \cong \text{Cl}(M^f) \otimes \mathbb{Q} \\ &\cong \text{Pic}(M^f) \otimes \mathbb{Q}.
 \end{align*}
 Note further that the restricted map $\phi: Y^f \to M^f$ is not only a geometric quotient for the action of $G= \text{SL}_2$, but also for the induced action of $\tilde G$ on $Y^f$. Using Lemma \ref{Lem:Picisomorph} we conclude the desired statement.
\end{proof}
We see that we can reduce the computation of $\text{Pic}(M(d|d_1, \ldots, d_n)) \otimes \mathbb{Q}$ to the computation of $\text{Cl}(Y_{d,n}^{ss,\textbf{d}}) \otimes \mathbb{Q}$. As $Y_{d,n}^{ss,\textbf{d}} \subset Y_{d,n}$ is an open set whose complement is a union of divisors, we will first compute generators and relations for the whole group $\text{Cl}(\overline M_{0,n}(\PP^1 \times \PP^1, (1,d))) \otimes \mathbb{Q}$ and then obtain the desired group as a quotient by the classes coming from the unstable loci.

\subsection{The Picard group of \texorpdfstring{$\overline M_{0,n}(\PP^1 \times \PP^1, (1,d))$}{\textbackslash bar M0n(P1xP1,(1,d))}}
\todoOld{Reread with respect to $d=0$ matters} We will use methods adapted from \cite{pandhaintersect} to compute the Picard group of $Y_{d,n}=\overline M_{0,n}(\PP^1 \times \PP^1, (1,d))$.  First we find a set of divisors generating the group. Then we construct test curves, which we intersect with those divisors. As the intersection numbers only depend on the class of the divisors in the Picard group,  we will be able to show linear independence for some subsets of our generators by comparing intersection numbers. We mention that in \cite{opreadivisors}, Oprea has computed a system of generators for the rational Picard group of $\overline M_{0,n}(X, \beta)$ for $X=G/P$ a projective homogeneous space, which of course also covers $X=\PP^1 \times \PP^1$. 
\subsubsection{Generators}
Consider the open set $Y_{d,n}^o = M_{0,n}(\PP^1 \times \PP^1, (1,d)) \subset Y_{d,n}$ corresponding to maps with a smooth domain and its complement $Y_{d,n}^\partial = Y_{d,n} \setminus Y_{d,n}^o$, the boundary of $Y_{d,n}$. Then we have an exact sequence
\begin{equation}
 A_{2d+n} Y_{d,n}^\partial \to A_{2d+n} Y_{d,n} \to A_{2d+n} Y_{d,n}^o \to 0.
\end{equation}
of groups of algebraic $(2d+n)$-cycles modulo rational equivalence (see \cite[Proposition 1.8]{inttheory}). We note that as $\dim(Y_{d,n}) = \dim(Y_{d,n}^o) = 2d+1+n$, we have $A_{2d+n} Y_{d,n} = \text{Cl}(Y_{d,n})$, $A_{2d+n} Y_{d,n}^o = \text{Cl}(Y_{d,n}^o)$. Furthermore by \cite[§3.1]{fultonpandha}, $Y_{d,n}^\partial$ is of pure dimension $2d+n$. Recall that by \cite[Corollary 2]{pandhaconnected} we know that its irreducible components are the divisors 
\[D_{B,k} = D(\{1, \ldots, n\} \setminus B, (1,d-k); B, (0,k)),\]
where $B \subset \{1, \ldots, n\}$, $0 \leq k \leq d$ with $|B|\geq 2$ for $k=0$. In our interpretation as pointed graphs of rational maps, this is the divisor of graphs, where the vertical section contains the markings $B$ and maps with degree $k$. We conclude that $A_{2d+n} Y_{d,n}^\partial$ is the free abelian group generated by these divisors. 

It remains to find generators for the Class group of $Y_{d,n}^o$. For this we use the fact that the map $J$ defined in Lemma \ref{Lem:Jconstruction} is an isomorphism over this locus.
\begin{Cor} \label{Cor:PicYo}
 For $d \geq 1$, $n \geq 0$, the rational Picard group of $Y_{d,n}^o$ is given by
 \begin{equation*}
  \text{Pic}(Y_{d,n}^o) \otimes \mathbb{Q} = \begin{cases}
                                        0 &\text{for } n=0,\\
                                        \mathbb{Q} \pi_{\PP^1}^* \mathcal{O}(1) &\text{for } n=1,2,\\
                                        0 &\text{for } n\geq 3,
                                       \end{cases}
 \end{equation*}
 where for $n=1,2$ the map $\pi_{\PP^1} : Y_{d,n}^o \cong \text{Rat}_d \times ((\PP^1)^n \setminus \Delta) \to \PP^1$ is the projection on the first factor $\PP^1$.
 
 In the case $d=0$, we have $Y_{d,n}^o = M_{0,n}(\PP^1 \times \PP^1, (0,1)) = M_{0,n}(\PP^1,1) \times \PP^1$ and an induced projection $p: Y_{d,n}^o \to  \PP^1$ on the last factor. This gives an additional direct summand $\mathbb{Q} \mathcal{G}$ for $\mathcal{G} = p^*(\mathcal{O}_{\PP^1}(1))$ to the formula above.
\end{Cor}
\begin{proof}
 From Lemma \ref{Lem:Jconstruction} we see that $Y_{d,n}^o \cong \text{Rat}_d \times ((\PP^1)^n \setminus \Delta)$. By \cite{picardproduct}, the Picard group of the product of a rational variety with another variety is the direct sum of their corresponding Picard groups. Now for $d \geq 1$, the set $\text{Rat}_d$ is the complement of a hypersurface in $\PP^{2d+1}$ and thus has a finite Picard group, which vanishes after tensoring with $\mathbb{Q}$. For $d=0$ the degree $d$ maps from $\PP^1$ to itself are exactly the constant maps, so $\text{Rat}_d=\PP^1$ giving the additional summand $\mathbb{Q} \mathcal{G}$. 
 
 For the other factor we note that the rational Picard group of $(\PP^1)^n$ is freely generated by the classes 
 \[H_i=\pi_i^* \mathcal{O}_{\PP^1}(1),\]
 where $\pi_i : (\PP^1)^n \to \PP^1$ is the projection on the $i$-th factor. 
 The set $\Delta$ we remove is the union of irreducible divisors 
 \[\Delta_{ij} = \{(p_1, \ldots, p_n); p_i = p_j\}.\]
 Their divisor class equals $[\Delta_{ij}] = H_i + H_j$. By the excision exact sequence we have 
 \[\text{Pic}( (\PP^1)^n \setminus \Delta) \otimes \mathbb{Q} = \bigoplus_{i=1}^n \mathbb{Q} H_i / \bigoplus_{i \neq j} \mathbb{Q} [\Delta_{ij}].\]
 For $n=0$ this is obviously trivial, for $n=1$ we have exactly $\mathbb{Q} H_1$. For $n=2$ we obtain $\mathbb{Q} H_1 \oplus \mathbb{Q} H_2 / \mathbb{Q} (H_1+H_2) \cong \mathbb{Q} H_1$. For $n \geq 3$ we note that 
 \[[\Delta_{1,2}] + [\Delta_{1,3}] - [\Delta_{2,3}] = H_1+H_2+H_1+H_3-H_2-H_3 = 2 H_1.\]
 Similarly we can represent all other classes $H_i$ and thus one sees that the Picard group of $(\PP^1)^n \setminus \Delta$ is trivial in this case.
\end{proof}
\begin{Theo} \label{Theo:Generators}
 For $d \geq 0$ and $n \geq 0$, the rational Picard group of $Y_{d,n}$ is generated by $D_{B,k}$ for $B \subset \{1, \ldots, n\}$, $0 \leq k \leq d$ and $|B|\geq 2$ if $k=0$, together with the divisor class 
 \[\mathcal{H} = (\pi_1 \circ \text{ev}_1)^*(\mathcal{O}_{\PP^1}(1))\]
 in the cases $n=1,2$ and the class $\mathcal{G} = \pi^*(\mathcal{O}_{\PP^1}(1))$ for $d=0$, where 
 \[\pi: \overline M_{0,n}(\PP^1 \times \PP^1, (0,1)) = \overline M_{0,n}(\PP^1,1) \times \PP^1 \to \PP^1\]
 is the projection on the factor $\PP^1$. 
\end{Theo}
\begin{proof}
 The first part of the Theorem is a combination of the discussion at the beginning of this section together with Corollary \ref{Cor:PicYo}. Here we note that the divisor $\mathcal{H}$ above obviously restricts to $\pi_{\PP^1}^* \mathcal{O}(1)$ on $Y_{d,n}^o$ and similarly for $\mathcal{G}$. 
\end{proof}
In the following, we denote the set of generators above by $\Gene{d}{n}$.
\subsubsection{Relations}
Now we want to find all relations between the generators of the rational Picard group of $Y_{d,n}$ found above. Using techniques adapted from \cite{pandhaintersect} we construct curves in $Y_{d,n}$ and intersect them with the divisors above. Relations among the classes of the divisors would imply relations between these intersection numbers. Hence using test curves for which the vectors of intersection numbers are linearly independent, we can show the linear independence of some of the generators above. The construction of the curves will rely on the following explicit family of stable maps.

Consider the variety $S=\PP^1 \times \PP^1$ with coordinates $([z:w],[x_0:x_1])$ and the projections $\pi_1, \pi_2$ to the two factors. Let $d \geq 0$, $0 \leq k \leq d$ and 
 \[\mathcal{N} = \mathcal{O}_S(d,k) = \pi_1^*(\mathcal{O}_{\PP^1}(d)) \otimes \pi_2^*(\mathcal{O}_{\PP^1}(k)).\]
 Two global sections $s_1, s_2 \in H^0(S,\mathcal{N})$  are given by
 \begin{align*}
  s_1 &= (z-a_1 w)(z-a_2 w)\ldots (z-a_{d-k}w)(x_1 z - b_1 x_0 w)  \ldots (x_1 z - b_k x_0 w)\\
  s_2 &= (z-c_1 w)(z-c_2 w)\ldots (z-c_{d-k}w)(x_1 z - d_1 x_0 w)  \ldots (x_1 z - d_k x_0 w),
 \end{align*}
 where $a_i, b_j, c_i, d_j \in \mathbb{C}^*$ for $i=1, \ldots, d-k$, $j=1, \ldots, k$ are sufficiently general. In the case $d=0$ we have $\mathcal{N} = \mathcal{O}$ and we set $s_1=a, s_2=c$ with $(a,c) \neq (0,0)$. These sections define a rational map
 \begin{align*}
  \mu: &S \, - \rightarrow \PP^1 \times \PP^1,\\
  &([z:w],[x_0:x_1]) \mapsto ([z:w], [s_1:s_2]).
 \end{align*}
 For $k = 0$, this is a morphism and induces a constant map of degree $d$ on the fibres of $\pi_2$. For $k \geq 1$ its \textit{base points} are exactly 
 \begin{itemize}
  \item the $k(d-k)$ points $([c_i:1],[c_i/b_j:1])$ for $i=1, \ldots, d-k$, $j=1, \ldots, k$, 
  \item the $k(d-k)$ points $([a_i:1],[a_i/d_j:1])$ for $i=1, \ldots, d-k$, $j=1, \ldots, k$, 
  \item the $2$ points $([0:1],[0:1])$ and $([1:0],[1:0])$.
 \end{itemize}
 To illustrate this, consider the following picture. We draw for $d=4, k=2$ the vanishing sets of $s_1, s_2$ on $\PP^1 \times \PP^1$. The horizontal lines of these sets come from the zeroes at $[z:w]=[a_i:1]$ and $[z:w]=[c_i:1]$, respectively. The curved parts depict the graphs of the linear maps $[x_0:x_1] \mapsto [b_j x_0:x_1]$ and $[x_0:x_1] \mapsto [d_j x_0:x_1]$, respectively. The base points are exactly the intersection points of $V(s_1)$ with $V(s_2)$. In particular, we see the two special intersection points over $[0:1]$ and $[1:0]$, where $k$ zeroes of $s_1$ and $s_2$ come together.
 
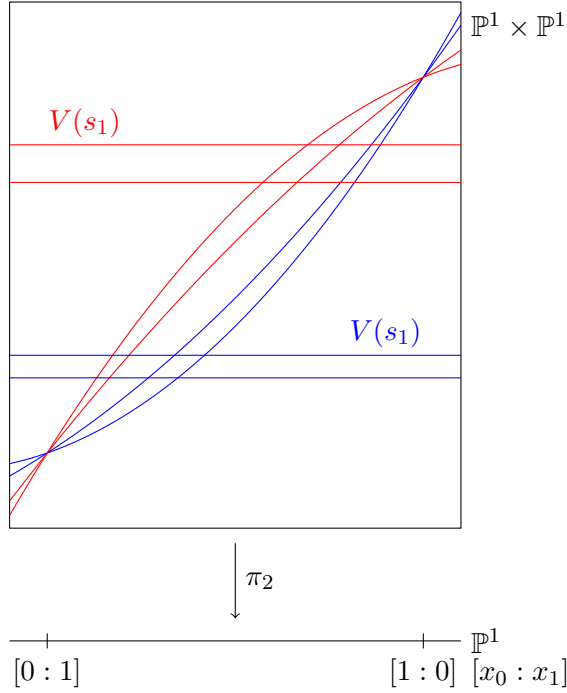
\begin{center}
 \begin{tikzpicture}
  \draw (-3,-3.5) rectangle (3,3.5) node[below right]{$\PP^1 \times \PP^1$};
  \draw[->] (0,-3.7) -- (0,-4.2) node[right]{$\pi_2$} -- (0,-4.7);
  \draw (-3,-5) -- (3,-5) node[right]{$\PP^1$} ;
  \draw (3,-5.1) node[below right]{$[x_0:x_1]$};
  
  \draw (-2.5,-4.9) -- (-2.5,-5.1) node[below]{$[0:1]$};
  \draw (2.5,-4.9) -- (2.5,-5.1) node[below]{$[1:0]$};
  
  \draw[domain=-3:3,smooth,variable=\x,blue] plot ({\x},{\x + 0.13*(\x+2.5)*(\x-2.5)});
  \draw[domain=-3:3,smooth,variable=\x,blue] plot ({\x},{\x + 0.07*(\x+2.5)*(\x-2.5)});
  \draw[domain=-3:3,smooth,variable=\x,red] plot ({\x},{\x + -0.12*(\x+2.5)*(\x-2.5)});
  \draw[domain=-3:3,smooth,variable=\x,red] plot ({\x},{\x + -0.05*(\x+2.5)*(\x-2.5)});
  
  \draw[domain=-3:3,smooth,variable=\x,red] plot ({\x},{1.6});  
  \draw[domain=-3:3,smooth,variable=\x,red] plot ({\x},{1.1});
  \draw[domain=-3:3,smooth,variable=\x,blue] plot ({\x},{-1.5});  
  \draw[domain=-3:3,smooth,variable=\x,blue] plot ({\x},{-1.2});  
  
  \draw[red] (-2,1.9) node{$V(s_1)$};
  \draw[blue] (2,-0.9) node{$V(s_1)$};
 \end{tikzpicture}
 \end{center}
 From the formulas above, we conclude that for sufficiently general $a_i, b_j, c_i, d_j$ all base points have distinct $[x_0:x_1]$-coordinates. 
 
 Now let $n \geq 0$ and assume we have sections $\sigma_1, \ldots, \sigma_n : \PP^1 \to S$ of $\pi_2 : S \to \PP^1$. An \textit{intersection point} is a point $p \in S$ lying in the image of at least two of the sections $\sigma_i$. The set of \textit{special points} is the union of base points and intersection points.
\begin{Lem} \label{Lem:testcurves}
  Assume that there are finitely many intersection points and that at each of them, the sections $\sigma_i$ meeting there have distinct tangent directions. Furthermore, assume that the $[x_0:x_1]$-coordinates of all special points are pairwise distinct.
  
  Then the blow-up $\overline S$ of $S$ at all special points resolves the indeterminacies of $\mu$ and the induced maps $\mu: \overline S \to \PP^1 \times \PP^1$, $\sigma_1, \ldots, \sigma_n : \PP^1 \to \overline S$ give a Kontsevich stable family of $n$-pointed genus $0$ curves
  \[\mathcal{C} = (\pi_2 : \overline S \to \PP^1; \sigma_1, \ldots, \sigma_n; \mu : \overline S \to \PP^1 \times \PP^1)\]
  over $\PP^1$ which induces a map $\psi: \PP^1 \to Y_{d,n} = \overline M_{0,n}(\PP^1 \times \PP^1, (1,d))$.
  
  Moreover, for a sufficiently general choice of the parameters $a_i, b_i, c_i, d_i$, the map $\psi$ intersects the boundary of $Y_{d,n}$ transversally and a point $x=[x_0:x_1]$ maps to the boundary if and only if there is a special point $p \in S$ with $x=\pi_2(p)$. Let 
  \[B = \{i: \sigma_i \text{ passes through }p\} \subset \{1, \ldots, n\}\]
  be the set of indices of sections through $p$ and let $m=0$ if $p$ is not a base point, $m=1$ if $p$ is a base point different from $([1:0],[1:0])$ and $([0:1],[0,1])$ and $m=k$ otherwise. Then at $x$ the map $\psi$ intersects exactly the boundary divisor $D_{B,m}$.
\end{Lem}
\begin{proof}
 We first show that the blow-up of the special points resolves the indeterminacies of our rational map. Note that all base points except $([1:0],[1:0])$ (which is easily seen to behave similar to $([0:1],[0:1])$) lie in $\mathbb{A}^2 \subset \PP^1 \times \PP^1$. Thus we will set $w=x_1=1$ and use affine coordinates $z,x_0$.
 
 Around the point $p=(c_i, c_i/b_j)$ we can identify
 \[\text{Bl}_p \mathbb{A}^2 = \{((z,w),[T:S]): (z-c_i)S - (x_0 - c_i/b_j)T\} \subset \mathbb{A}^2 \times \PP^1.\]
 Then the second component of the map the map $\mu$ extends around $p$ by sending $((z,x_0),[T:S])$ to 
 \[\frac{(z-a_1) \ldots (z - b_{j-1} x_0)(T - b_{j} S)(z - b_{j+1} x_0) \ldots (z - b_k x_0)}{(z-c_1) \ldots (z-c_{i-1}) T (z-c_{i+1}) \ldots (z - d_k x_0)}.\]
 Similarly for $q=(0,0)$ we have
 \[\text{Bl}_q \mathbb{A}^2 = \{((z,w),[T:S]): zS - x_0T\} \subset \mathbb{A}^2 \times \PP^1.\]
 Here, the second component of $\mu$ can be extended by sending $((z,x_0),[T:S])$ to
 \[ \frac{(z-a_1) \ldots (z-a_{d-k})(T - b_{1} S) \ldots (T - b_{k} S)}{(z-b_1) \ldots (z-b_{d-k})(T - d_{1} S) \ldots (T - d_{k} S)}.\]
 It is also clear\detex{ from the universal property of blowups} that the sections $\sigma_i$ factor through $\overline S$. As their tangent directions in every special point are distinct, they map to distinct points on the exceptional divisors. As $\pi_2 : \overline S \to \PP^1$ is dominant, it is flat and as blowups are projective, it is projective. The fibres of $\pi_2$ are isomorphic to $\PP^1$ except for the fibres over projections of special points, which are nodal genus $0$ curves with two branches. The exceptional divisors are one of the branches and as can be seen above, they map to $\PP^1 \times \PP^1$ with degree $(0,1)$ or $(0,k)$ at the base points and are contracted for the intersection points which are not base points, and thus map with degree $(0,0)$. The fact that the induced map $\PP^1 \to Y_{d,n}$ meets the boundary transverally follows because the total space $\overline S$ of the family is smooth\detex{ as the blowup of $\PP^1 \times \PP^1$ at finitely many points} (see \cite[Section 4.4]{MR1718648}).\detex{\todo{fix} can be seen from Proposition \ref{Pro:transverseboundary}. Furthermore, we can reach that the maps $\overline S_{x} \to \PP^1 \times \PP^1$ of fibres containing special points do not have automorphisms. Indeed, the component mapping with degree $(1,*)$ can never have automorphisms. The vertical components, that can occur, are either lying
 \begin{itemize}
  \item over the intersection of several sections $\sigma_i$, mapping with degree $0$ with one node and at least two markings,
  \item over one of the $2k(d-k)$ simple base points, hence mapping with degree $(0,1)$,
  \item over $([0:1],[0:1])$ or $([1:0],[1:0])$. Here, if $k \geq 1$ we can adjust the parameters $b_i, d_i$ such that the degree $k$ map from the vertical section to its image does not have automorphisms leaving the image of the node invariant. 
 \end{itemize}
 Hence all conditions of Proposition \ref{Pro:transverseboundary} are satisfied and the intersection with the boundary is transverse.}
\end{proof}
We will find that for a given subset $B \subset \{1, \ldots, n\}$, a particular choice of the sections $\sigma_i$ will be very useful.
\begin{Def} \label{Def:psiCBk}
 For $\alpha \in \mathbb{C}^*$ we define
 \[S_\alpha : \PP^1 \to \PP^1 \times \PP^1, [x_0:x_1] \mapsto ([\alpha x_0:x_1],[x_0:x_1]).\]
 For $B \subset \{1, \ldots, n\}$ and $0 \leq k \leq d$ we set 
 \begin{itemize}
  \item $\sigma_j = S_{\alpha_j}$ for general $\alpha_j \in \mathbb{C}\setminus\{0,1\}$ for $j \in B$,
  \item $\sigma_j = (p_j, \text{id})$ with general $p_j=[p_j:1] \in \mathbb{C} \subset \PP^1$ for $j \notin B$.
 \end{itemize}
 We will specify the required generality for the points $\alpha_i, p_j$ in the proof of Proposition \ref{Pro:CBk} below. We denote by $\psi_{B,k}: \PP^1 \to Y_{d,n}$ the corresponding curve from Lemma \ref{Lem:testcurves} and by $C_{B,k}=(\psi_{B,k})_*([\PP^1])$ its image cycle in $Y_{d,n}$. 
\end{Def}
For later use we compute the evaluation of the $i$th point along the map $\psi_{B,k}$. Here for brevity we identify points $[\lambda:1]\in \PP^1$ with $\lambda \in \mathbb{C}$.
 \begin{align}&(\text{ev}_i \circ \psi_{B,k})([x_0:x_1]) \nonumber\\
 =& \begin{cases} 
       (p_j,\frac{(p_j-a_1)\ldots (p_j-a_{d-k})(x_1 p_j - b_1 x_0) \ldots (x_1 p_j - b_k x_0)}{(p_j-b_1)\ldots (p_j-b_{d-k})(x_1 p_j - d_1 x_0) \ldots (x_1 p_j - d_k x_0)} ) \text{ for }i \notin B,\\
       ([\alpha_i x_0:x_1],\frac{(\alpha_i x_0-a_1 x_1)\ldots (\alpha_i x_0-a_{d-k}x_1) (x_1 \alpha_i x_0 - b_1 x_0 x_1) \ldots (x_1 \alpha_i x_0 - b_k x_0 x_1)}{(\alpha_i x_0-c_1 x_1)\ldots (\alpha_i x_0-c_{d-k}x_1) (x_1 \alpha_i x_0 - d_1 x_0 x_1) \ldots (x_1 \alpha_i x_0 - d_k x_0 x_1)} )\\
       =([\alpha_i x_0:x_1],\frac{(\alpha_i x_0-a_1 x_1)\ldots (\alpha_i x_0-a_{d-k}x_1) (\alpha_i - b_1) \ldots ( \alpha_i  - b_k )}{(\alpha_i x_0-c_1 x_1)\ldots (\alpha_i x_0-c_{d-k}x_1) ( \alpha_i  - d_1 ) \ldots ( \alpha_i  - d_k )} ) \text{ for }i \in B.
    \end{cases} \label{eqn:evipsi}
\end{align}
We will see now that the curves $C_{B,k}$ above were constructed to meet very specific divisors in $Y_{d,n}$.
\begin{Pro} \label{Pro:CBk}
 The nonzero intersection numbers of the curves $C_{B,k}$ with the generators of the rational Picard group of $Y_{d,n}$ are exactly
 \begin{itemize}
  \item $(C_{B,k},\mathcal{H}) = 1$ if $1 \in B$ and $0$ otherwise
  \item $(C_{B,k},D_{\{a,b\}, 0})=1$ for $a \in A=\{1, \ldots, n\} \setminus B, b \in B$
  \item $(C_{B,k},D_{B,k})=2$ 
  \item $(C_{B,k}, D_{\emptyset, 1})=2k(d-k)$ 
 \end{itemize}
 Note that for $B=\emptyset, k=1$ we have $(C_{B,k}, D_{\emptyset, 1})=2d = 2 + 2(d-1)$, so in this case the two different numbers from above are added.
\end{Pro}
\begin{proof}
All the base points in the construction of the test curves of Lemma \ref{Lem:testcurves} lie on the union of the images of the maps $S_{b_j},S_{d_j}$ for $j=1, \ldots, k$. But the images of $S_\alpha, S_\beta$ for $\alpha \neq \beta$ intersect exactly at $(0,0)$, $(\infty, \infty)$ and there they have distinct tangent directions. We choose the points $p_i$ such that the sections $(p_i, \text{id})$ miss all those base points and additionally their intersection points with the sections $\sigma_j$, $j \in B$ do not have the same second coordinate as one of the base points. Then it is ensured that the conditions of Lemma \ref{Lem:testcurves} are satisfied. 

Using the projection formula we obtain
 \begin{align*}(C_{B,k}.\mathcal{H}) &= \text{deg}((\psi_{B,k})_*[\PP^1].(\pi_1\circ \text{ev}_1)^* \mathcal{O}_{\PP^1}(1))\\ &= \text{deg}((\pi_1 \circ \text{ev}_1 \circ \psi_{B,k})^*\mathcal{O}_{\PP^1}(1)).\end{align*}
From equation (\ref{eqn:evipsi}) we see that $\pi_1 \circ \text{ev}_1 \circ \psi$ is an isomorphism for $1 \in B$ and constant for $1 \notin B$.
The sections $\sigma_i$, $i \in A$, do not meet among each other. For $j \in B$, they meet the section $\sigma_j$ once and all those sections meet exactly in the points $(0,0), (\infty, \infty)$, which are base points of degree $k$. This explains the remaining intersection numbers.
\end{proof}
We need one other test curve in case $d=0$.
\begin{Pro}
 For $d=0$ consider the identity map $\text{id}: S=\PP^1 \times \PP^1 \to S$ and constant sections $\sigma_i : \PP^1 \to S, q \mapsto (p_i, q)$. Then $\pi_2 : S \to \PP^1$ defines a family of stable maps over $\PP^1$ where the map over $[x_0:x_1]$ corresponds to the inclusion $\PP^1 \to \PP^1 \times \{[x_0:x_1]\}$ with the horizontal position of the marked points held fixed. This gives a curve $C_\mathcal{G}$ in $Y_{0,n}$ which of all generators in $\Gene{d}{n}$ intersects exactly the divisor $\mathcal{G}$ with multiplicity $1$.
\end{Pro}
For formulating results about the relations among the generators of the rational Picard group of $Y_{d,n}$, we group those generators into convenient subsets. For $0 \leq j \leq n$ we define
\begin{align*}
 \Delta_j = \{D_{B,k} : |B|=j, 0 \leq k \leq d \text{ and } 1 \leq k \text { for }j=0,1\}.
\end{align*}
In terms of our interpretation as pointed graphs of rational maps, these are exactly those graphs where the vertical section carries $j$ marked points. Additionally we let
\[\Delta = \bigcup_{j=0}^n \Delta_j\]
be the set of all boundary divisor generators and for $n \geq 4$ we define
\[\Delta' = \Delta \setminus (\Delta_0 \cup \Delta_1 \cup \Delta_{n-1} \cup \Delta_n),\]
where we set $\Delta'=\emptyset$ for $n\leq 4$.
\begin{Theo} \label{Theo:relations}
 Let $d \geq 0, n \geq 0$, then in $\text{Pic}(Y_{d,n}) \otimes \mathbb{Q}$
 \begin{itemize}
  \item the divisor $\mathcal{G}$ is linearly independent of all other generators for $d=0$,
  \item the divisor $\mathcal{H}$ is not contained in the span of $\Delta$ for $n=1,2$,
  \item the set $\Delta_0$ is linearly independent for $n=0$,
  \item the set $\Delta_0 \cup \Delta_1 \cup \Delta_{n-1} \cup \Delta_n$ is linearly independent modulo the span of $\Delta'$ for $n \geq 1$,
  \item the relations among the divisor classes in $\Delta'$ are exactly the pullback of 
  relations between the boundary divisors in $\overline M_{0,n}$ for $n \geq 4$. These are generated by the relations 
  \[\sum_{\substack{i,j \in A\\ k,l \in B}} D(A;B) = \sum_{\substack{i,k \in A\\ j,l \in B}} D(A;B)\]  
  for distinct $i,j,k,l \in \{1, \ldots, n\}$.
 \end{itemize}
\end{Theo}
\begin{proof}
 For $d=0$, as $(C_\mathcal{G}.\mathcal{G})=1$ and all intersections with other generators are zero, $\mathcal{G}$ is linearly independent of those.
 
 For $n=1$ we see that $C_{\{1\}, 0}$ intersects exactly the divisor $\mathcal{H}$ with multiplicity $1$.  Hence $\mathcal{H}$ is linearly independent from the span of $\Delta$.
 
 For $n=2$ we consider the curves $C_{\{i\}, 0}$ for $i=1,2$. They intersect $D_{\{1,2\},0}$ with multiplicity $1$ and $C_{\{1\},0}$ also intersects $\mathcal{H}$. Hence the $1$-cycle $C_{\{1\}, 0} - C_{\{2\}, 0}$ intersects only $\mathcal{H}$ nontrivially, which is therefore linearly independent.
 
 Now let $n \geq 0$ and consider the curves $C_{\emptyset, k}$, which only intersect $D_{\emptyset,k}$ and $D_{\emptyset,1}$ nontrivially. The case $k=1$ shows that $D_{\emptyset, 1}$ is linearly independent from all other boundary divisors. Using this in the case $2 \leq k \leq d$, we also obtain linear independence of the other divisors $D_{\emptyset,k}$ in $\Delta_0$.
 
 Let now $n \geq 1$. For $1 \leq i \leq n$ and $1 \leq k \leq d$ choose $\sigma_j = (p_j,\text{id})$ for $j \neq i$ missing the base points as in Definition \ref{Def:psiCBk}, but $\sigma_i = (0, \text{id})$. Then the only nonzero intersections with boundary divisors are
 \begin{align*}
  (C.D_{ \emptyset, 1}) = 2k(d-k), (C.D_{\emptyset, k}) = 1, (C.D_{\{i\}, k}) = 1
 \end{align*}
 As the first two types of divisors are in $\Delta_0$, which is already seen to be linearly independent, we obtain the independence of all divisors in $\Delta_1$.
 
 For $n \geq 2$, $0 \leq k \leq d$,  we first consider $C_{\{1, \ldots, n\},k}$, which only intersects $D_{\{1, \ldots, n\},k}$ and $D_{\emptyset, 1}$.
 Again we already know that $D_{\emptyset, 1}$ is independent of the other boundary divisors and thus we obtain the independence of all elements in $\Delta_n$.
 
 Now if for some $i \in \{1, \ldots, n\}$ we modify the family giving us $C_{\{1, \ldots, n\},k}$ by setting $\sigma_i = (0, \text{id})$, this section no longer meets $(\infty, \infty)$ and we obtain a new curve $C$ in $Y_{d,n}$. Note now that
 \[(C.D_{\emptyset, 1}) = 2k(d-k), (C.D_{\{1, \ldots, n\}, k}) = 1, (C.D_{\{1, \ldots, n\} \setminus \{i\}, k}) = 1.\]
 As we already know the independence of $\Delta_n$, this also gives us the independence of $\Delta_{n-1}$.
 
 Now for $n \geq 4$ we come to the relations among the divisors in $\Delta'$. Remember that there is a morphism
 \[F : Y_{d,n} = \overline M_{0,n}(\PP^1 \times \PP^1, (1,d)) \to \overline M_{0,n} \]
 by forgetting the map and only remembering the stabilization of the domain curve. The map $F$ is flat by Theorem \ref{Theo:flatforgetful}. 
 For $A,B \subset \{1, \ldots, n\}$ disjoint with $|A|, |B|\geq 2$ and $A \cup B = \{1, \ldots, n\}$, the boundary divisor $D(A;B)$ in $\overline M_{0,n}$ pulls back under $F$ to the multiplicity free sum
 \[D_{A \cup B} = \sum_{m=0}^d D_{A,m} + \sum_{m=0}^d D_{B,m} = F^*(D(A;B)).\]
 We claim that given a relation 
 \begin{equation}\sum_{D \in \Delta'} c_D D = 0 \label{eqn:bdryrelation} \end{equation}
 in $\text{Pic}(Y_{d,n}) \otimes \mathbb{Q}$, it is the pullback of a relation in $\overline M_{0,n}$. In a first step we will show, that the coefficient $c_{D_{B,k}}$ of $D_{B,k}$ in (\ref{eqn:bdryrelation}) only depends on $B$ and moreover, the coefficients for $B$ and $A=\{1, \ldots, n\} \setminus B$ coincide. We will denote them by $c_{A \cup B}$.
 
 Indeed for $k \geq 1$ we can take the intersection of the relation (\ref{eqn:bdryrelation}) with the $1$-cycle $C_{B,k} - C_{B,0}$. We see that all the intersections with the divisors $D_{\{i,j\}, 0}$ for $i \in A, j \in B$ cancel and what remains is
 \[0=(C_{B,k} - C_{B,0}, \sum_{D \in \Delta'} c_D D) = 2 c_{D_{B, k}} - 2 c_{D_{B, 0}},\]
 that is $c_{D_{B, k}} = c_{D_{B, 0}}$ does not depend on $k$. 
 
 For the claim that the coefficient also remains the same if we switch the roles of $A$ and $B$, we take the intersection of (\ref{eqn:bdryrelation}) with $C_{B,k} - C_{A,k}$. Again we see a cancellation and obtain
 \[0=(C_{B,k} - C_{A,k}, \sum_{D \in \Delta'} c_D D) = 2 c_{D_{B, k}} - 2 c_{D_{A,k}}\]
 which concludes the proof that $c_D$ only depends on the partition of the marked points in $D$. 
 
 Thus we know that the relation (\ref{eqn:bdryrelation}) is of the form 
 \begin{equation*}0=\sum_{A \cup B = \{1, \ldots, n\}} c_{A \cup B} D_{A \cup B} =F^*(\sum_{A \cup B = \{1, \ldots, n\}} c_{A \cup B} D(A;B)) \end{equation*}
 But as $F$ is surjective, proper and flat, by \cite[Corollary 2.3]{kerpictor} the kernel of the map map $F^* : \text{Pic}(\overline M_{0,n}) \to \text{Pic}(Y_{d,n})$ is torsion, so the induced map of rational Picard groups is injective. Hence as claimed, the relation (\ref{eqn:bdryrelation}) is the pullback under $F$ of the relation $\sum_{A \cup B = \{1, \ldots, n\}} c_{A \cup B} D(A;B)=0$ in $\text{Pic}(\overline M_{0,n}) \otimes \mathbb{Q}$.
 
 The form of the relations of boundary divisors in $\overline M_{0,n}$ was proved in \cite[Theorem 1]{keelintersect}.
\end{proof}
\begin{Cor}\label{Cor:PicrankY}
 For $d \geq 0, n \geq 0$, the rank of the Picard group of $Y_{d,n}$ is $2^n(d+1)-{n \choose 2} -1 +\delta_{n,1}+\delta_{n,2} + \delta_{d,0}$.
\end{Cor}
\begin{proof}
 Using Theorem \ref{Theo:Generators}, we first count generators. For the boundary divisors $D_{B,k}$, we see that there are $2^n$ choices for $B$ and $d+1$ choices for $k$. In the case $k =0$ we have to substract the $n+1$ choices for a set $B$ with at most $1$ element. We also have one additional generator $\mathcal{H}$ for $n=1,2$ and $\mathcal{G}$ for $d=0$.
 
 For the relations, we see as in \cite{pandhaintersect} (c.f. the discussion above Lemma 1.2.3) that there are ${n-1 \choose 2}-1$ independent relations among the boundary components in $\overline M_{0,n}$. 
 
 Thus the total dimension of $\text{Pic}(Y_{d,n}) \otimes \mathbb{Q}$ is
 \begin{align*} &2^n(d+1)-(n+1)+\delta_{n,1}+\delta_{n,2} + \delta_{d,0} - \left({n-1 \choose 2}-1 \right)\\  = &2^n(d+1)-{n \choose 2} -1+\delta_{n,1}+\delta_{n,2} + \delta_{d,0}\end{align*}
\end{proof}
For the sake of completeness, we also want to explicitly name one subset of the generators that forms a basis.
\begin{Cor} \label{Cor:PicYbasis}
 Consider the set $\Gene{d}{n}$ of generators of $W=\text{Pic}(Y_{d,n}) \otimes \mathbb{Q}$ from Theorem \ref{Theo:Generators}. 
 Then the set
 \begin{equation} \label{eqn:Picbasis}
 \Basi{d}{n}=\Gene{d}{n} \setminus \{D_{B,0}; B \subset \{2, \ldots, n\}, |B|=2 \text{ and } B\neq\{n-1,n\}\} 
 \end{equation} 
 forms a basis of $W$. Note that $\Basi{d}{n}=\Gene{d}{n}$ for $n \leq 3$.
\end{Cor}
\begin{proof}
 One sees easily from Corollary \ref{Cor:PicrankY} that for all $d,n$, the set $\Basi{d}{n}$ has $\text{dim}(W)$ elements. For $n \leq 3$ this finishes the proof, so we may assume that $n \geq 4$. Let 
 \[V= \bigoplus_{D \in \Gene{d}{n}} \mathbb{Q} D\]
 be the vector space with formal basis $\Gene{d}{n}$ together with the natural surjective map $\text{cl}: V \to W$ by taking the class in the rational Picard group. Let $U \subset V$ denote the kernel of this map. We have seen in the proof of Corollary \ref{Cor:PicrankY} that $\text{dim}(U) = {n-1 \choose 2}-1$. Moreover by Theorem \ref{Theo:relations}, it is spanned by relations among the generators in $\Delta'$ obtained as pullback from the relations 
  \begin{equation} \label{eqn:RelationsM0n}
   \sum_{\substack{i,j \in A\\ k,l \in B}} D(A;B) = \sum_{\substack{i,k \in A\\ j,l \in B}} D(A;B) 
  \end{equation}
 among boundary divisors $D(A;B)$ in $\overline M_{0,n}$.
 Set $D_B = \sum_{k} D_{B,k}$ where the sum is over $k=0, \ldots, d$ if $|B| \geq 2$ and $k=1, \ldots, d$ otherwise. Then the relation (\ref{eqn:RelationsM0n}) pulls back to
 \begin{equation} \label{eqn:RelationsM0n2}
  \sum_{\substack{i,j \in A\\ k,l \in B}} D_B + \sum_{\substack{i,j \in B\\ k,l \in A}} D_B = \sum_{\substack{i,k \in A\\ j,l \in B}} D_B + \sum_{\substack{i,k \in B\\ j,l \in A}} D_B. 
 \end{equation}
 Now consider the canonical projection $V \to V'$ on the subspace $V' \subset V$ generated by the divisors $D_{B,0}$ with $|B|=2$. From (\ref{eqn:RelationsM0n2}) we see that under this projection, the space $U$ maps to the space $U'$ generated by elements of the form
 \begin{align} \label{eqn:RelationsM0n3}
 D_{\{k,l\},0} + D_{\{i,j\},0}  - D_{\{j,l\},0} - D_{\{i,k\},0}.
 \end{align}
 We have another projection $V' \to V''$ where $V''$ is the span of 
 \[\{D_{B,0}; B \subset \{2, \ldots, n\}, |B|=2, \text{ and } B\neq\{n-1,n\}\}.\]
 One sees immediately that $\text{dim}(V'') = {n-1 \choose 2}-1 = \text{dim}(U)$ and we claim that the image $U''$ of the induced projection $U' \to V''$ is all of $V''$. This in turn would imply, that the map $U \to V''$ is an isomorphism and from this one immediately concludes that $\Basi{d}{n}$ forms a basis of $W$. 
 
 First note that under $V' \to V''$ the elements \[D_{\{1,2\},0}, D_{\{1,3\},0}, \ldots, D_{\{1,n\},0}, D_{\{n-1,n\},0}\] map to $0$ by definition. Taking $i=1, k=n, l=n-1$ in (\ref{eqn:RelationsM0n3}) we see that $- D_{\{j,n-1\},0} \in U''$. Switching $k$ and $l$ we also have $- D_{\{j,n\},0} \in U''$. But now take $2 \leq i,j \leq n-2$ arbitrary distinct and $k=n, l=n-1$ then we have $D_{\{i,j\},0} - D_{\{i,n\},0} - D_{\{j,n-1\},0} \in U''$, so also $D_{\{i,j\},0} \in U''$. But this finishes the proof.  
\end{proof}

\subsubsection{Identification of divisors} 
Now that we have a basis of the rational Picard group of $Y_{d,n}$, we can find an algorithm to explicitly represent a given divisor $D$ as a linear combination in this basis. We will see that all the information that is needed are the intersection numbers of $D$ with the test curves $C_{B,k}$ from Proposition \ref{Pro:CBk} together with the intersection $(C_\mathcal{G}.D)$ for $d=0$. We now give explicit formulas for the coefficients of the basis elements.
\begin{Pro} \label{Pro:iddivisors}
 Let $D$ be a rational divisor class on $Y_{d,n}$ and let $N_{B,k}=(C_{B,k}.D)$. Then $D$ has a unique representation
 \begin{equation} \label{eqn:Drep}
  D= \sum_{D_{B,k} \in \Basi{d}{n}} c_{D_{B,k}} D_{B,k} \underbrace{+ c_{\mathcal{H}} \mathcal{H}}_{\text{for }n=1,2} \underbrace{+ c_{\mathcal{G}} \mathcal{G}}_{\text{for }d=0}. 
 \end{equation}
 The coefficients are determined as follows
\begin{itemize}
 \item $c_\mathcal{G} = (C_\mathcal{G}.D)$ for $d=0$,
 \item $c_\mathcal{H} = N_{\{1\},0}$ for $n=1$, $c_\mathcal{H} = N_{\{1\},0} - N_{\{2\},0}$ for $n=2$ and $0$ otherwise,
 \item $c_{\{1,2\},0}=N_{\{2\},0}$ for $n=2$,
 \item $c_{\{1,j\},0}=N_{\{j\},0}$ for $2 \leq j \leq n-2$, $n\geq 3$,
 \item $c_{\{1,n-1\},0}=\frac{1}{2} (N_{\{1\},0} - N_{\{2\},0} - \ldots - N_{\{n-2\},0} + N_{\{n-1\},0} - N_{\{n\},0})$ for $n \geq 3$,
 \item $c_{\{1,n\},0}=\frac{1}{2} (N_{\{1\},0} - N_{\{2\},0} - \ldots - N_{\{n-2\},0} - N_{\{n-1\},0} + N_{\{n\},0})$ for  $n \geq 3$,
 \item $c_{\{n-1,n\},0}=\frac{1}{2} (-N_{\{1\},0} + N_{\{2\},0} + \ldots + N_{\{n-2\},0} + N_{\{n-1\},0} + N_{\{n\},0})$ for $n \geq 3$,
 \item $c_{\{k,l\},0}=0$ for $k,l>1$ if $\{k,l\} \neq \{n-1,n\}$ for $n \geq 3$,
 \item $c_{B,k} = \frac{1}{2}(N_{B,k} - \frac{k(d-k)}{d} N_{\emptyset,1} - \chi_B(1)c_\mathcal{H} - \sum_{\substack{a \notin B, b \in B}} c_{\{a,b\},0})$ for $B \subset \{1, \ldots, n\}$, $k \geq 0$ and $(|B|,k) \neq (2,0)$.
\end{itemize}
Here, $\chi_B$ is the characteristic function of the set $B$, so $\chi_B(m)=1$ if $m \in B$ and $\chi_B(m)=0$ otherwise.
\end{Pro}
\begin{proof}
 From Corollary \ref{Cor:PicYbasis} it is clear that a unique representation of $D$ in the form (\ref{eqn:Drep}) must exist. To arrive at the formulas above one takes the intersection of equation (\ref{eqn:Drep}) with the test curves $C_\mathcal{G}$, $C_{B,k}$ and checks that the resulting linear system uniquely determines the coefficients to be the numbers above. Here one should proceed in the order suggested above, except that one needs to determine $c_{\emptyset, 1}$ using $N_{\emptyset,1}$ before calculating the other numbers $c_{B,k}$. 
\end{proof}

\subsubsection{Geometric divisors} 
\label{Sect:Geodiv}
We now want to define several divisors on $Y_{d,n}$ and use Proposition \ref{Pro:iddivisors} to identify them in terms of our basis. 

First of all, for every $i \in \{1, \ldots, n\}$ we have the evaluation map $\text{ev}_i : Y_{d,n} \to \PP^1 \times \PP^1$. This gives us divisors
\begin{align*}
 \mathcal{H}_{i,1} &= (\pi_1 \circ \text{ev}_i)^* \mathcal{O}_{\PP^1}(1),\\
\mathcal{H}_{i,2} &= (\pi_2 \circ \text{ev}_i)^* \mathcal{O}_{\PP^1}(1).
\end{align*}
For $n=1,2$, the divisor $\mathcal{H}_{1,1} = \mathcal{H}$ is an element of or basis $\Basi{d}{n}$ of the rational Picard group. While it is possible to describe the elements $\mathcal{H}_{i,1}$ in other cases as well using the formulas in Proposition \ref{Pro:iddivisors}, the resulting representation is not very illuminating. However we will see that by substracting suitable multiples of the divisors $\mathcal{H}_{i,j}$ from other geometric divisors below, the representation of these divisors becomes much nicer. As the divisors $\mathcal{H}_{i,1}$ are easy to handle in any case, we will use them in our expression of other geometric divisors. 

As a first step we apply this to the divisors $\mathcal{H}_{i,2}$. 
\begin{Pro} \label{Pro:Hprime}
 The divisor class of $\mathcal{H}'_{i,2}= \mathcal{H}_{i,2}-d \mathcal{H}_{i,1}$ has the form
 \[\mathcal{H}'_{i,2}= \sum_{k=1}^d \left(\sum_{B \not\ni i} \frac{k^2}{2d} D_{B,k} + \sum_{B \ni i} (\frac{k^2}{2d} -k) D_{B,k} \right) \underbrace{+ \mathcal{G}}_{\text{for }d=0}. \]
\end{Pro}
\begin{proof}
 To compute the intersection numbers of the divisors $\mathcal{H}_{i,j}$ with $C_{B,k}$ consider again equation (\ref{eqn:evipsi}). Then it is clear that \[(C_{B,k}.\mathcal{H}_{i,j}) = \text{deg}(\pi_j \circ \text{ev}_i \circ \psi_{B,k}).\]
 Hence we conclude
 \begin{align*}
  (C_{B,k}.\mathcal{H}_{i,1})&=\chi_B(i),\\
  (C_{B,k}.\mathcal{H}_{i,2})&=d \chi_B(i) + k(1-2 \chi_B(i)).
 \end{align*}
 Thus
 \[(C_{B,k}.\mathcal{H}'_{i,2})=k(1-2 \chi_B(i)).\]
 Going through the recipe of Proposition \ref{Pro:iddivisors} we find the desired formula. 
\end{proof}
The reason for substracting $d \mathcal{H}_{i,2}$ was that it eliminates intersections with all the test curves $C_{B,0}$, which would make the formulas much more complicated.

For the next definitions, we will use that the evaluation maps, which were already considered above, are flat. 
\begin{Lem} \label{Lem:evflat}
 For $i=1, \ldots, n$, the evaluation maps $\text{ev}_i : Y_{d,n} \to \PP^1 \times \PP^1$ are flat and surjective.
\end{Lem}
\begin{proof}
 All that we will use, is that there exists a transitive action of an algebraic group $H$ on the target $X=\PP^1 \times \PP^1$, which leaves the curve class $\beta=(1,d)$ invariant. In our case, we can take the natural action of $H=\text{PGL}_2 \times \text{PGL}_2$. Using Lemma \ref{Lem:stackyaction}, we obtain an induced action of $H$ on $Y_{d,n}$ (by postcomposition) making $\text{ev}_i$ equivariant. Then surjectivity is immediate, as $H$ acted transitively on $X$. Moreover, by generic flatness, the map $\text{ev}_i$ is flat over some open subset $U\subset X$. But then it is flat over $g.U$ for all $g \in H$ and using again the transitivity of the action, it is flat everywhere.
\end{proof}
One divisor that will be important later is the subset of $Y_{d,n}$ where the $i$th marked point is a fixed point of the self-map. When looking at the graph $\Gamma$ of a map from $\PP^1$ to itself, the set of fixed points is exactly (the projection of) the intersection of $\Gamma$ with the diagonal $\Delta \subset \PP^1\times \PP^1$. Thus we make the following definition.
\begin{Def} \label{Def:fixdiv}
 For $d,n \geq 0$ we call
 \[D_{i=\text{fix}} = \text{ev}_i^{-1}(\Delta) \subset \overline M_{0,n}(\PP^1 \times \PP^1, (1,d))
 \]
 the $i$th fixed point divisor.
\end{Def}
This is an effective Cartier divisor. As $\mathcal{O}(\Delta) = \pi_1^*(\mathcal{O}(1)) \otimes \pi_2^*(\mathcal{O}(1))$ we have 
\[D_{i=\text{fix}} = \mathcal{H}_{i,1} + \mathcal{H}_{i,2}\]
as divisor classes in $\text{Pic}(Y_{d,n})$.

Next, for a point $p \in \PP^1 \times \PP^1$ we want to consider the locus $D_p \subset Y_{d,n}$ of stable maps $f : C \to \PP^1 \times \PP^1$, where $p$ lies on the graph $f(C)$. For this, consider the forgetful map $F: Y_{d,n+1} \to Y_{d,n}$ of the last marking $n+1$, which we interpret as the universal curve over $Y_{d,n}$. Here, we have the evaluation map $\text{ev}_{n+1} : Y_{d,n+1} \to \PP^1 \times \PP^1$, which is flat. Hence, we can pull back the cycle $\{p\} \subset \PP^1 \times \PP^1$ and then push it forward to $Y_{d,n}$ via $F$. Indeed, we define
\[D_p = F_* \text{ev}_{n+1}^* [\{p\}].\]
As $\{p\}$ is codimension $2$ and as $F$ has fibres of dimension $1$, this should indeed be a divisor in $Y_{d,n}$. We now compute its class in the rational Picard group.
\begin{Pro} \label{Pro:Dp}
 The divisor class of $D_p$ has the form
 \[D_p= \sum_{k=1}^d \frac{k^2}{2d} \sum_{B\subset \{1, \ldots, n\}} D_{B,k} \underbrace{+ \mathcal{G}}_{\text{for }d=0}. \]
\end{Pro}
\begin{proof}
 As the class of $\{p\}$ in the Chow group of $\PP^1 \times \PP^1$ is exactly 
 \[[\{p\}] = c_1(\mathcal{O}(1,0)) \cap c_1(\mathcal{O}(0,1)) \cap [\PP^1 \times \PP^1],\]
 we know that
 \begin{align*}\text{ev}_{n+1}^* [\{p\}] &= c_1(\mathcal{H}_{n+1,2}) \cap c_1(\mathcal{H}_{n+1,1}) \cap [Y_{d,n+1}]\\
  &= c_1(\mathcal{H}'_{n+1,2}) \cap c_1(\mathcal{H}_{n+1,1}) \cap [Y_{d,n+1}].
 \end{align*}
 Here we use that $c_1(\mathcal{H}_{n+1,1})^2 = \text{ev}_{n+1}^* c_1(\mathcal{O}(1,0))^2 = 0$. But now we can apply the formula for $\mathcal{H}'_{n+1,2}$ from Proposition \ref{Pro:Hprime}. Note in the following, that for $n+1 \in B$, $1 \leq k \leq n$  we have
 \[F_* \left(c_1(\mathcal{H}_{n+1,1}) \cap  D_{B,k} \right) = 0, \]
 because a general point of the cycle on the left has a vertical section with a fixed horizontal position, and this is already a codimension $2$ condition. Using this, we compute
 \begin{align*}
  &F_* \text{ev}_{n+1}^* [\{p\}]\\ = &F_* c_1(\mathcal{H}_{n+1,1}) \cap \left( \sum_{k, B \not \ni n+1} \frac{k^2}{2d} D_{B,k} + \sum_{k, B  \ni i} (\frac{k^2}{2d} -k) D_{B,k} \underbrace{+ \mathcal{G}}_{\text{for }d=0}\right)\\
  = &F_* c_1(\mathcal{H}_{n+1,1}) \cap \left( \sum_{k, B } \frac{k^2}{2d} D_{B,k} \underbrace{+ \mathcal{G}}_{\text{for }d=0} \right)\\
  = &F_* c_1(\mathcal{H}_{n+1,1}) \cap F^* \left( \sum_{k, B } \frac{k^2}{2d} D_{B,k}\underbrace{+ \mathcal{G}}_{\text{for }d=0}  \right),\\
 \end{align*}
 where in the last line, the sum of boundary divisors is on $Y_{d,n}$. But then, for using the projection formula to obtain the desired result, we only need that 
 \[F_* c_1(\mathcal{H}_{n+1,1}) \cap [Y_{d,n+1}] = [Y_{d,n}].\]
 But this is clear, since over the locus $Y_{d,n}^o$ with smooth domain curve, any subvariety $\text{ev}_{n+1}^{-1} ( \{p_1\} \times \PP^1)$ (for $p_1 \in \PP^1$) maps birationally onto its image via $F$ (for a formal proof see Lemma \ref{Lem:birationalsection}). 
\end{proof}

\subsection{The Picard group of \texorpdfstring{$M(d|d_1, \ldots, d_n)$}{M(d|d1, ..., dn)} finished}
\begin{Cor} \label{Cor:preBasis}
 Let $\textbf{d}=(d|d_1, \ldots, d_n)$ be admissible and set 
 \[d_T = d+1+\sum_{i=1}^n d_i.\] 
 Then the rational Picard group of $Y_{d,n}^{ss,\textbf{d}}$ is the quotient of $\text{Pic}(Y_{d,n})\otimes_{\mathbb{Z}} \mathbb{Q}$ by the linear span of the divisors $D_{B,k}$ with $k + \sum_{i \in B} d_i > \frac{d_T}{2}$ and the divisors $D_{i=\text{fix}}$ for all $i$ with $d_i>d_T/2-1$.
\end{Cor}
\begin{proof}
 By restricting to the open set $Y_{d,n}^{ss,\textbf{d}}$ we divide out the divisor classes of all codimension $1$ components of $Y_{d,n} \setminus Y_{d,n}^{ss,\textbf{d}}$. We will use Lemma \ref{Lem:semistableadvanced} to identify this locus.
 
 The first case that can make a closed point of $Y_{d,n}$ unstable is when a marked point $i$ with weight $d_i$ becomes a fixed point and $d_i + 1 > d_T/2$. This is exactly accounted for by dividing by $D_{i=\text{fix}}$. We note that this is the only cause of instability away from the boundary.
 
 Now a general point of the boundary divisor $D_{B,k}$ corresponds to a parametrized graph with exactly one vertical section  of multiplicity $k$ over a non-fixed point and with the points in $B$ on the vertical section. As we have already taken care of instabilities from marked points being fixed points, this is unstable iff $k + \sum_{i \in B} d_i > \frac{d_T}{2}$. 
\end{proof}

\begin{Cor} \label{Cor:finalBasis}
 For $\textbf{d}=(d|d_1, \ldots, d_n)$ admissible, the map 
 \[\phi^* : \text{Pic}(M(d|d_1, \ldots, d_n)) \otimes_{\mathbb{Z}} \mathbb{Q} \to \text{Pic}(Y_{d,n}^{ss,\textbf{d}}) \otimes_{\mathbb{Z}} \mathbb{Q}\] is an isomorphism. 
\end{Cor}
\begin{proof}
 For $(d,n) \neq (2,0), (1,1)$, this is exactly Corollary \ref{Cor:Picidentification}. For $d=2, n=0$ we have the isomorphism $j:Y_{d,n}^s \to Z_d^s$ from Lemma \ref{Lem:dn20} inducing an isomorphism $M(2,0) \cong \PP^2$. Then one checks that the generator $D_{\emptyset,1}$ of the rational Picard group corresponds exactly to the line at infinity $\PP^2 \setminus \mathbb{A}^2$ and hence forms a basis of the rational Picard group. On the other hand in the case $d=1,n=1$, we have seen that the only nonempty moduli spaces are all isomorphic to $M(1|1)$. 
 But we saw $M(1|1) \cong \PP^1$ in Lemma \ref{Lem:M(1,1)}. 
 On the other hand Corollary \ref{Cor:preBasis} gives generators $\mathcal{H}, D_{\emptyset, 1}$, but we must divide by the class
 \[D_{1=\text{fix}} = 2 \mathcal{H} + \frac{1}{2} D_{\emptyset, 1}.\]
 Hence the Picard group of $Y_{1,1}^{ss,(1|1)}$ is of rank $1$ and one checks that the pullback map $\varphi^*$ induces an isomorphism with $\text{Pic}(\PP^1)\otimes_\mathbb{Z} \mathbb{Q} = \mathbb{Q}$.
\end{proof}

\section{Iteration of rational functions} \label{Sect:Iterate}
One feature of self-maps $f:X \to X$ of a set $X$ is the possibility to iterate them. We define the $m$-fold self composition
\begin{align*}
 f^{\circ m} = \underbrace{f \circ f \circ \cdots \circ f}_{m \text{ times}},
\end{align*}
of $f$, where $f^{\circ 0} = \text{id}_X$. Now if $\psi: X \to X$ is an automorphism of $X$, we can look at the iteration of the conjugated map $f_\psi= \psi^{-1} \circ f \circ \psi$ and see
\[f_\psi^{\circ m} = \psi^{-1} \circ f \circ \psi \circ \psi^{-1} \cdots f \circ \psi = \psi^{-1} \circ f^{\circ m} \circ \psi.\]
Hence, $f_\psi^{\circ m}$ is conjugated to $f^{\circ m}$, again by the map $\psi$.

Returning to our setting of self-maps of $\mathbb{P}^1$ this allows us to define iteration on conjugacy classes $[f]$ of self-maps $f : \PP^1 \to \PP^1$. The map $f^{\circ m}$ has degree $d^m$. By the argument above, one sees that the induced map 
\[\mathfrak{sc}_m: \text{Rat}_d \to \text{Rat}_{d^m}, f \mapsto f^{\circ m}\]
descends to a map of the quotients
\[\mathfrak{sc}_m: M_d \to M_{d^m}, [f] \mapsto [f^{\circ m}].\]
In the following we want to see how to extend these maps to (parts of) the boundary of our compactifications $M(d,0)$ and also how to handle marked points. As above it will be more convenient to first work on the original varieties $Y_{d,n}=\overline M_{0,n}(\PP^1 \times \PP^1, (1,d))$, construct $\text{PGL}_2$-equivariant maps there and then descend them to the quotients.

\subsection{Composition maps}
Recall that the spaces $Y_{d,0}$ were natural compactifications of the spaces $\text{Rat}_d$ of degree $d$ maps from $\PP^1$ to $\PP^1$. But as long as we do not ``forget the coordinates'' on $\PP^1$ by taking the quotient under the action of $\text{PGL}_2$, we are even able to compose two different maps $f,g : \PP^1 \to \PP^1$ of possibly different degrees. One sees easily that the map 
\begin{equation} \label{eqn:circnaive}
 \mathfrak{c} : \text{Rat}_{d_1} \times \text{Rat}_{d_2} \to \text{Rat}_{d_1 d_2}, (f,g) \mapsto g \circ f
\end{equation}
is an algebraic morphism
The self-composition morphism $\mathfrak{sc}_2: \text{Rat}_d \to \text{Rat}_{d^2}$ is obtained from the map above (in case $d_1=d_2=d$) by precomposing with the diagonal $\text{Rat}_d \to \text{Rat}_d \times \text{Rat}_d$. Thus we can study the more general question of how to extend $\mathfrak{c}$ to parts of the compactification $Y_{d_1,0}\times Y_{d_2,0}$ of $\text{Rat}_{d_1} \times \text{Rat}_{d_2}$ and also how to handle marked points. We will see that this more general setup will allow us to recursively construct an extension of the map $\mathfrak{sc}_m$ using the fact that 
\[\mathfrak{sc}_m(f) = \underbrace{f \circ f \circ \cdots \circ f}_{m-1\text{ times}} \circ f = \mathfrak{sc}_{m-1}(f) \circ f.\]
We begin with an easy result that will allow us to handle marked points later.
\begin{Def}
 For $d,n \geq 0$ and $f=(f : C \to \PP^1 \times \PP^1;p_1, \ldots, p_n) \in Y_{d,n}$ let
 \[\text{Vert}(f) = \{q \in \PP^1: \text{some component of }C\text{ is contracted to }q\text{ by }\pi_1 \circ f\}\]
 be the set of positions of vertical sections for $f$.
\end{Def}
Obviously, $\text{Vert}(f)$ is finite and for $f \in Y_{d,n}^o$ it is empty.
\begin{Lem} \label{Lem:birationalsection}
 Let $F:Y_{d,n+1} \to Y_{d,n}$ be the map forgetting the $(n+1)$st marking and let $\text{ev}_{n+1} : Y_{d,n+1} \to \PP^1 \times \PP^1$ be the evaluation map for this marking. Then the map
 \[E=F \times (\pi_1 \circ \text{ev}_{n+1}) : Y_{d,n+1} \to Y_{d,n} \times \PP^1\]
 is birational. More precisely, the restriction of $E$ to
 \[V=Y_{d,n+1} \setminus \bigcup\left( D_{B,k}: n+1 \in B \text{ and }(|B|>2 \text{ or } k>0) \right) \]
 is an open embedding $V \xrightarrow{\sim} U \subset Y_{d,n} \times \PP^1$ with complement 
 \[\text{Vert}_{d,n} = Y_{d,n} \times \PP^1 \setminus U \]
 of codimension at least $2$. We have the description 
 \[U(\mathbb{C}) = \left\{ \left( f , q\right) : q \notin \text{Vert}(f) \right\}\]
 of the closed points of $U$.
 Over $U$ we have a section $s : U \to Y_{d,n+1}$ of $E$. 
\end{Lem}
\begin{proof}
 We are going to show that $E$ induces a bijection from the closed points of $V$ to the points in the set $U(\mathbb{C})$ above. Then by Zariski's main theorem, as $Y_{d,n} \times \PP^1$ is normal, $E|_V$ is an isomorphism of $V$ to an open set $U \subset Y_{d,n} \times \PP^1 $ with the set $U(\mathbb{C})$ of closed points given above. The section $s$ is then given by $(E|_V)^{-1}$.
 
 To prove the bijection between $V(\mathbb{C})$ and $U(\mathbb{C})$, note first that by construction the closed points $(f: C \to \PP^1 \times \PP^1;p_1, \ldots, p_{n+1}) \in V$ are exactly those elements of $Y_{d,n+1}$ such that the marking $p_{n+1}$ does not lie on a vertical component of $C$ remaining stable under forgetting $p_{n+1}$. Under the map $E$ we forget this marking and stabilize to obtain a map $\tilde f : \tilde C \to \PP^1 \times \PP^1$, but also remember the horizontal position $q$ of $f(p_{n+1})$. We have now verified that indeed $q \notin \text{Vert}(\tilde f)$, so there is exactly one point $\widehat q \in \tilde C$ with $f(\widehat q)$ above $q$ and it lies on the horizontal component of $\tilde C$. Hence we can recover the original curve $C$ by making $\widehat q$ the $(n+1)$st marking (possibly introducing a contracted component if $\widehat q$ is already a marked point of $\tilde C$). One checks that this is exactly a description of the inverse $s$ of $E|_V$ over $U$. 
 
 The only remaining point is to check that the complement $\text{Vert}_{d,n}$ of $U$ has codimension at least $2$. But looking at the map 
 \[\text{Vert}_{d,n} \to Y_{d,n}\]
 we see it only has nonempty fibres over the boundary of $Y_{d,n}$ and these fibres are finite, hence
 \[\dim (\text{Vert}_{d,n} ) \leq \dim(Y_{d,n})-1 = \dim(Y_{d,n} \times \PP^1)-2.\]
\end{proof}
The basic idea for extending the composition map $\mathfrak{c}$ from (\ref{eqn:circnaive}) to the boundary in $Y_{d_1,0} \times Y_{d_2,0}$ is to use the universal families 
\[\begin{CD}
 Y_{d_i,1}     @>\text{ev}_1>>  \PP^1\times \PP^1\\
@VVV        \\
Y_{d_i,0}   
\end{CD}\]
to construct a family in $Y_{d_1d_2,0}(W)$ for an open set $W \subset Y_{d_1,0} \times Y_{d_2,0}$, inducing a map $W \to Y_{d_1d_2,0}$. Let us illustrate the construction over points \[f_i = ((\pi_i, \phi_i) : C_i \to \PP^1 \times \PP^1) \in \text{Rat}_{d_i},\]
where we already know the desired result $f_2 \circ f_1$. Here the maps $\pi_i$ induce isomorphisms $C_i \cong \PP^1$ and the composition $f_2 \circ f_1 \in \text{Rat}_{d_1d_2}$ should be given by
\[f_2 \circ f_1 = (\text{id}_{\PP^1},\phi_2 \circ  \pi_2^{-1} \circ \phi_1 \circ \pi_1^{-1}) : \PP^1 \to \PP^1 \times \PP^1.\]
One way to obtain this map without having to explicitly invert the maps $\pi_i$ is to use as the domain of $f_2 \circ f_1$ the curve 
\[C = C_1 \times_{\phi_1, \PP^1,\pi_2} C_2 = \{(c_1, c_2) \in C_1 \times C_2: \phi_1(c_1) = \pi_2(c_2)\}. \]
As $\pi_2 : C_2 \to \PP^1$ is an isomorphism, the projection $ C \to C_1$ on the first factor is an isomorphism, so $C$ is isomorphic to $\PP^1$. However, it is now easy to write down $f_2 \circ f_1$ on C as
\begin{align*}
 f_2 \circ f_1 : C &\to \PP^1 \times \PP^1\\
  (c_1, c_2) &\mapsto  (\pi_1(c_1), \phi_2(c_2)).
\end{align*}
One sees that these two definitions of $f_2 \circ f_1$ are compatible via the isomorphism $C \to C_1 \xrightarrow{\pi_1} \PP^1$ of the source curves:
\begin{align*}
 (\phi_2 \circ  \pi_2^{-1} \circ \phi_1 \circ \pi_1^{-1})(\pi_1(c_1)) &= (\phi_2 \circ  \pi_2^{-1} )(\phi_1(c_1))\\
 &=(\phi_2 \circ  \pi_2^{-1} )(\pi_2(c_2)) = \phi_2(c_2).
\end{align*}
The advantage of the second construction is that it does not require the $\pi_i: C_i \to \PP^1$ to be isomorphisms, and this is exactly what fails on the boundary of $Y_{d_i,0}$. In the following Lemma, we will try to understand in which situations the fibre product $C_1 \times_{\phi_1, \PP^1,\pi_2} C_2$ is again a nice curve for $C_1, C_2$ not necessarily smooth. The reader might benefit from matching the descriptions in the Lemma below to the illustration in Figure \ref{Fig:composition} and the corresponding Remark \ref{Rmk:figcomposition}.
\begin{Lem} \label{Lem:compgeopoint}
 Let $f_i=((\pi_i, \phi_i) : C_i \to \PP^1 \times \PP^1;p_{i,1} , \ldots, p_{i,n_i}) \in Y_{d_i,n_i}$ be geometric points, $i=1,2$. Assume that they satisfy the following condition:
 \begin{equation} \label{eqn:condP1}
  \parbox{0.8\linewidth}{For all $q \in \text{Vert}(f_2)$, the preimage $\phi_1^{-1}(\{q\})$ is a union of smooth, reduced points of $C_1$.}
   \tag{P1}
 \end{equation}
 Then the fibre product 
 \[C = C_1 \times_{\phi_1, \PP^1,\pi_2} C_2 = \{(c_1, c_2) \in C_1 \times C_2: \phi_1(c_1) = \pi_2(c_2)\} \]
 is a projective, connected, reduced genus $0$ curve with at worst nodal singularities. It maps to $\PP^1\times \PP^1 \times \PP^1$ via
 \begin{align*} 
   \widehat{f_2 \circ f_1} : C &\to \PP^1 \times \PP^1 \times \PP^1\\
   (c_1, c_2) &\mapsto (\pi_1(c_1),\underbrace{\phi_1(c_1)}_{=\pi_2(c_2)}, \phi_2(c_2)).
 \end{align*}
 Let $G_i = ((V_{i,j})_{j=0}^{r_i}, E_i)$ be the dual graphs of the curves $C_i$. That is, the graphs with vertices $V_{i,0}, \ldots, V_{i,r_i}$ corresponding to the irreducible components of $C_i$ and edges $e \in E_i$ corresponding to nodes shared by the components connected by the edge. Let the component $V_{i,0}$ map with degree $(1,d_{i,0})$ and the components $V_{i,j}$ with degree $(0,d_{i,j})$ for $0<j\leq r_i$.
 
 Then the dual graph $G$ of $C$ is obtained from $G_1$ by glueing to all vertices $V_{1,j}$ in $G_1$ a number of $d_{1,j}$ copies of the graph $G_2$ along its root vertex $V_{2,0}$. The new degree (with respect to $(\PP^1)^3$) of this glueing vertex is $(\delta_{j,0} ,d_{j,0},d_{j,0}d_{2,0})$, the new degrees of the copies of $V_{2,j}$ for $0<j\leq r_2$ are $(0,0,d_{2,j})$. 
 
 Assume that in addition to condition (P1), the points $f_1, f_2$ satisfy
 \begin{equation} \label{eqn:condP2}
  \parbox{0.8\linewidth}{For all marked points $p_{1,j} \in C_1$ we have $\phi_1(p_{1,j}) \notin \text{Vert}(f_2) \cup \pi_2\left(  \{p_{2,j}: j=1, \ldots, n_2\}\right)$}
   \tag{P2}
 \end{equation}
 and
 \begin{equation} \label{eqn:condP3}
  \parbox{0.8\linewidth}{For all marked points $q \in C_2$ on the horizontal component (corresponding to $V_{2,0}$), the preimage $\phi_1^{-1}(\pi_2(q))$ is a union of smooth, reduced points of $C_1$.}
   \tag{P3}
 \end{equation}
 Then there exist $n_1 + d_1 n_2$ pairwise distinct markings 
 \[\widehat p_{1,1}, \ldots, \widehat p_{1,n_1}, \widehat p_{2,1}^1, \ldots, \widehat p_{2,1}^{d_1}, \widehat p_{2,2}^1, \ldots, \widehat p_{2,2}^{d_1}, \ldots,\widehat p_{2,n_1}^{d_1} \in C\]
 of smooth points in $C$ such that with the projections $\pi_{C_i} : C \to C_i$ on the two factors, we have
 \begin{itemize}
  \item $\pi_{C_1}(\widehat p_{1,j}) = p_{1,j}$ for $j=1, \ldots, n_1$,
  \item $\pi_{C_2}(\widehat p_{2,k}^l) = p_{2,k}$ for $k=1, \ldots, n_2, l=1, \ldots, d_1$.
 \end{itemize}
 Moreover, the data 
 \[\left( \widehat{f_2 \circ f_1}: C \to \PP^1 \times \PP^1 \times \PP^1; \widehat p_{1,1}, \ldots, \widehat p_{2,n_1}^{d_1} \right)\]
 gives a stable map of a genus $0$ curve with $n_1 + d_1 n_2$ markings of degree $(1,d_1, d_1d_2)$. It is uniquely defined up to the ordering of the markings $\widehat p_{2,k}^l$ for every $k=1, \ldots, n_2$. Thus by forgetting the second component of the map and stabilizing if necessary, we obtain a well-defined closed point 
 \[[f_2 \circ f_1] \in Y_{d_1 d_2,n_1 + d_1 n_2} / (S_{d_1})^{n_2},\] 
 where the $j$-th copy of $S_{d_1}$ acts by permuting the markings $\widehat p_{2,j}^1, \ldots, \widehat p_{2,j}^{d_1}$.
\end{Lem}
\newsavebox\picCone
\sbox{\picCone}{%
\begin{tikzpicture}[remember picture]
 \draw[domain=-3:3,smooth,variable=\x,blue] plot ({\x},{-0.09*\x * \x}) node[black, below]{$(1,2)$};
 \draw[domain=-2.9:-1.6,smooth,variable=\x,blue]  plot ({\x},{-0.5*(\x+5)*(\x+5)+4.5});
 \draw (-3,2.5) node[right]{$(0,1)$};
 \draw[domain=1.5:2.8,smooth,variable=\x,blue] plot ({\x},{-0.5*(\x-5)*(\x-5)+5}) node[black, right]{$(0,2)$};
 \filldraw[blue] (1,-0.09) circle (1pt) node[below]{$p_{1,1}$};
 \draw (0,-1.7) node{$(C_1, f_1)$};
\end{tikzpicture}}
\newsavebox\picCtwo
\sbox{\picCtwo}{%
\begin{tikzpicture}[remember picture]
 \draw[domain=-2:-0.2,smooth,variable=\x,red] plot ({\x},{-0.3*\x*(\x-5)}) node[black, right]{$(0,k)$};
 \draw[domain=-4:-1,smooth,variable=\x,red] plot ({\x},{-0.1*(\x+1.5)*(\x+7)-3.5}) node[black, right]{$(1,d-k)$};
 \filldraw[red] (-2.5,{-0.1*(-2.5+1.5)*(-2.5+7)-3.5}) circle (1pt) node[below]{$p_{2,1}$};
 \filldraw[red] (-0.7,{-0.3*(-0.7)*((-0.7)-5)}) circle (1pt) node[right]{$p_{2,2}$};
 \draw (-2,-4.7) node{$(C_2, f_2)$};
\end{tikzpicture}
}
\newsavebox\picConetwo
\sbox{\picConetwo}{%
\begin{tikzpicture}[scale=1.5,remember picture]
 \draw[domain=-3:3,smooth,variable=\x,blue] plot ({\x},{-0.09*\x * \x}) node[black, below]{$(1,2(d-k))$};
 \draw[domain=-2.9:-1.6,smooth,variable=\x,blue]  plot ({\x},{-0.5*(\x+5)*(\x+5)+4.5});
 \draw (-3,2.5) node[right]{$(0,d-k)$};
 \draw[domain=1.5:2.8,smooth,variable=\x,blue] plot ({\x},{-0.5*(\x-5)*(\x-5)+5}) node[black, right]{$(0,2(d-k))$};
 
 \filldraw[red] (-2.7,{-0.5*(-2.7+5)*(-2.7+5)+4.5}) circle (1pt) node[right]{$\widehat p_{2,1}^1$};
 \filldraw[red] (-1.7,{-0.09*(-1.7)*(-1.7)}) circle (1pt) node[below right]{$\widehat p_{2,1}^2$};
 \filldraw[red] (-0.2,{-0.09*(-0.2)*(-0.2)}) circle (1pt) node[below]{$\widehat p_{2,1}^3$};
 \filldraw[red] (1.9,{-0.5*(1.9-5)*(1.9-5)+5}) circle (1pt) node[right]{$\widehat p_{2,1}^4$};
 \filldraw[red] (2.16,{-0.5*(2.16-5)*(2.16-5)+5}) circle (1pt) node[right]{$\widehat p_{2,1}^5$};
 
 \draw[domain=0.1:0.4,smooth,variable=\x,red] plot ({\x},{-(\x-5)*(\x-5)+23}) node[black, above]{$(0,k)$};
 \filldraw[red] (0.3,{-(0.3-5)*(0.3-5)+23}) circle (1pt) node[right]{$\widehat p_{2,2}^3$};
 \draw[domain=-0.85:-0.7,smooth,variable=\x,red] plot ({\x},{-2*(\x-4)*(\x-4)+46}) node[black, above]{$(0,k)$};
 \filldraw[red] (-0.75,{-2*(-0.75-4)*(-0.75-4)+46}) circle (1pt) node[right]{$\widehat p_{2,2}^2$};
 
 \draw[domain=-2.8:-1.9,smooth,variable=\x,red] plot ({\x},{0.5*(\x+4)*(\x+4)-0.2}) node[black, above]{$(0,k)$};
 \filldraw[red] (-2.2,{0.5*(-2.2+4)*(-2.2+4)-0.2}) circle (1pt) node[right]{$\widehat p_{2,2}^1$};
 
 \draw[domain=1:2.5,smooth,variable=\x,red] plot ({\x},{0.25*(\x-4)*(\x-4)-0.2}) node[black, right]{$(0,k)$};
 \filldraw[red] (1.5,{0.25*(1.5-4)*(1.5-4)-0.2}) circle (1pt) node[right]{$\widehat p_{2,2}^4$};
 
 \draw[domain=1.3:2.6,smooth,variable=\x,red] plot ({\x},{0.26*(\x-4)*(\x-4)+0.6}) node[black, right]{$(0,k)$};
 \filldraw[red] (1.7,{0.26*(1.7-4)*(1.7-4)+0.6}) circle (1pt) node[right]{$\widehat p_{2,2}^5$};
 
 \filldraw[blue] (1,-0.09) circle (1pt) node[below]{$\widehat p_{1,1}$};
 
 \draw (0,-1.5) node{$(C_1 \times_{\PP^1} C_2, f_2 \circ f_1)$};
\end{tikzpicture}
}

\begin{figure}[htb]
\begin{tikzpicture}[remember picture]
 \draw (-3,3) node[inner sep=0pt] (picCone){\usebox{\picCone}};
\draw (3,3) node[inner sep=0pt] (picCtwo){\usebox{\picCtwo}};
\draw[->] (-1,0) -- (-1,-1) ;
 \draw (0,-4.5) node[inner sep=0pt] (picConetwo){\usebox{\picConetwo}};
\end{tikzpicture}
\caption{Illustration of composition of $(C_1,f_1)$ and $(C_2,f_2)$}
\label{Fig:composition}
\end{figure}
\begin{Rmk} \label{Rmk:figcomposition}
 In Figure \ref{Fig:composition} we illustrate the composition of $(C_1,f_1) \in Y_{5,1}$ and $(C_2,f_2) \in Y_{d,2}$, both of them contained in the boundary of their respective spaces. The components are labelled by the bidegree (with respect to $\PP^1 \times \PP^1$) of the maps $f_1, f_2, f_2 \circ f_1$ restricted to them, respectively. Note that for $j=1,2$, the ordering of the points $\widehat p_{2,j}^l$ only represents one of the possible choices and is not canonical. 
\end{Rmk}
\begin{proof}[Proof of Lemma \ref{Lem:compgeopoint}]
 As $C_1, C_2$ are projective, so is their product $C_1 \times C_2$ and $C$ is then also projective as a closed subscheme of $C_1 \times C_2$. 
 By property (P1), we have
 \[C \subset \left(C_1 \times_{\PP^1} C_{2,0}^{sm}\right) \cup \bigcup_{j=0}^{n_1} \left(C_{1,j}^{sm} \times_{\PP^1} C_2\right)\subset C_1 \times_{\PP^1} C_2,\]
 where $C_{1,j}^{sm}$ are the smooth points of irreducible components of $C_1$ and $C_{2,0}^{sm}$ are the smooth points of the horizontal component of $C_2$ (mapping with degree $1$ on the first component of $\PP^1 \times \PP^1$).
 Indeed, if $(c_1,c_2) \in C$ with $c_2 \notin C_{2,0}^{sm}$, then $q=\pi_2(c_2) \in \text{Vert}(f_2)$, but then by property (P1) we know $c_1 \in \phi_1^{-1}(\{q\})$ is a smooth point of $C_1$. 
 
 Let $U_j=C_{1,j}^{sm} \times_{\PP^1} C_2$ and $W=C_1 \times_{\PP^1} C_{2,0}^{sm}$ be the open sets covering $C$. We will see that the spaces $U_j, W$ are reduced, at worst nodal, connected genus $0$ curves. This shows the local properties of being a reduced and at worst nodal curve for $C$. For the global property of being connected of genus $0$ we will simultaneously keep track of which components of $C$ are connected by nodes and will verify the construction of the dual graph $G$ of $C$. As $G$ is obviously a tree, this will finish the first part of the proof.
 
 First consider the set $W$. Via the map $\pi_2$, the set $C_{2,0}^{sm} \subset C_2$ is embedded in $\PP^1$ as an open subset and $W$ is just the preimage of this subset in $C_1$ via $\phi_1$. This shows all the desired properties of $W$ except the connectedness. But as the preimage of $\PP^1 \setminus C_{2,0}^{sm} = \text{Vert}(f_2)$ consists of isolated, smooth points of $C_1$, the curve $W$ is also connected. In the dual graph $G$ of $C$ the open subcurve $W \subset C$ corresponds to the copy of the dual graph $G_1$ of $C_1$. The degrees with which the components map can easily be seen: the first two components of $\widehat{f_2 \circ f_1}$ agree with the map $f_1$ on $W \subset C_1$. The last component is the composition 
 \[W \xrightarrow{\phi_1} \PP^{1} \setminus \text{Vert}(f_2) \xrightarrow{\pi_2^{-1}} C_{2,0}^{sm} \xrightarrow{\phi_2} \PP^1.\]
 Thus, on each irreducible component of $W$, the degree of this composition is equal to the degree of $\phi_1$ on this component times the degree of $\phi_2$ on $C_{2,0}$, as claimed.
 
 Now consider the subsets $U_j$ coming from cartesian diagrams
 \[\begin{CD}
 U_j     @>>>   C_2\\
@VVV        @VV\pi_2V\\
C_{1,j}^{sm}    @>>\phi_1>  \PP^1 
\end{CD}\]
On the complement $C_{1,j}^{sm,nc} = C_{1,j}^{sm} \setminus \{p: d \phi_1|_p = 0\}$ of its critical points, the map $\phi_1$ is \'etale. Similarly, on the smooth points $C_{2,0}^{sm}$ of the horizontal component of $C_2$ the map $\pi_2$ is an open embedding, so also \'etale. By condition (P1) the preimages of these sets in $U_j$ cover $U_j$, because a critical point in $C_{1,j}^{sm}$ mapping by $\phi_1$ to an element in $q \in \text{Vert}(f_2)$, would give a non-reduced point in $\phi_1^{-1}(\{q\})$. Because the properties of being reduced and of dimension $1$ ascend along \'etale morphisms (see \cite[\href{http://stacks.math.columbia.edu/tag/034E}{Tag 034E},\href{http://stacks.math.columbia.edu/tag/04N4}{Tag 04N4} ]{stacks}), $U_j$ has these two properties. Note also that the property of having at worst nodal singularities can be checked in an analytic neighborhood. As $\phi_1$ is a local isomorphism around all points in $\text{Vert}(f_2)$, the curve $U_j$ inherits this property from $C_2$.

Now we come to the global geometry of $U_j$: over the points in $\PP^1 \setminus \text{Vert}(f_2)$, we obtain a copy of an open subset of $C_{1,j}^{sm}$, sitting inside of the curve $W$ above. Over each point $q \in \text{Vert}(f_2)$ there sits a unique tree of $\PP^1$s as the preimage under $\pi_2$ and a number of $d_{1,j}$ smooth, reduced points as the preimage under $\phi_1$ by condition (P1). Thus $d_{1,j}$ copies of this tree are glued to $C_{1,j}^{sm}$ at these preimages. This exactly corresponds to the operation on the dual graph $G$ described above and also the stated degrees of $\widehat{f_2 \circ f_1}$ are immediate from this description. As this subgraph of the dual graph is connected and as all components of $U_j$ have genus $0$, we are done with the first part of the proof.

Now assume that also (P2) and (P3) hold. By (P2), for every $j \in \{1, \ldots, n_1\}$ there is a unique point $c_{2,j} \in C_2$ lying on the horizontal component of $C_2$ such that $\widehat p_{1,j}=(p_{1,j},c_{2,j}) \in C$, that is $\phi_1(p_{1,j}) = \pi_2(c_{2,j})$. Similarly, (P3) implies that for every marking $p_{2,k}$ of $C_2$, the preimage $\phi_1^{-1}(\pi_2(p_{2,k}))$ consists of $d_1$ distinct smooth points $c_{1,k}^1, \ldots, c_{1,k}^{d_1}$. We set $\widehat p_{2,k}^l = (c_{1,k}^l, p_{2,k})$ for $l=1, \ldots, d_1$. By construction and by (P2), all these $n_1 + d_1 n_2$ markings are pairwise distinct.

Now consider the map $\widehat{f_2 \circ f_1}$. To see that it is stable, note that we can easily see on the dual graph $G$ of $C$ where the markings are placed: for a marking $p_{1,j}$ on the vertex $V_{1,j'}$ of $G_1$, the corresponding mark $\widehat p_{1,j}$ remains on the induced vertex $V_{1,j'}$ of the graph $G$. For any marking $p_{2,k}$ lying on the root vertex $V_{2,0}$ of $C_2$, every vertex $V_{1,j}$ of $G$ coming from $G_1$ receives a number of $d_{1,j}$ of the markings from the set $\{\widehat p_{2,k}^1, \ldots, \widehat p_{2,k}^{d_1}\}$. For all markings $p_{2,k}$ on a vertex $V_{2,j}$ with $j>0$, the induced markings $\widehat p_{2,k}^l$ lie on the $d_1$ copies of $V_{2,j}$ on the trees glued to the vertices coming from $G_1$. 

With this, we can check that $\widehat{f_2 \circ f_1}$ is indeed stable, looking at all the vertices of $G$. For those coming from $G_1$ it is clear that they are stable, because the first two components of their degree as well as any markings they had in $G_1$ are inherited from $G_1$. On the other hand, all remaining vertices come from vertices in $G_2$ and the second two components as well as the markings are inherited from this graph. Hence, all vertices of $G$ are stable.

For the total degree of the graph note that the first two components of the degree are simply inherited from $G_1$ on the corresponding subgraph of $G$ (and $0$ outside of this subgraph), so we get $(1,d_1)$ here. For the last component we have on the one hand the vertices $V_{1,j}$ from $G_1$, which contribute their old degree $d_{1,j}$ times $d_{2,0}$, so a total of $d_1 d_{2,0}$. On the other hand, for every vertex $V_{2,j}$ with $j>0$ we have $d_1$ copies of this vertex contributing $d_{2,j}$ each, so a total of 
\[\sum_{j=1}^{r_2} d_1 d_{2,j} = d_1 (d_2 - d_{2,0}).\]
Together we exactly have a degree of $d_1 d_2$.

The fact that the markings on $C$ only depended on the choices of ordering among the $p_{2,k}^l$ for every $k$ is clear. Hence, the rest of the statement follows.
\end{proof}
By the discussion preceding the Lemma above, we see that for $n_1=n_2 = 0$ and for $C_1, C_2$ irreducible (that is $f_1 \in \text{Rat}_{d_1}, f_2 \in \text{Rat}_{d_2}$), the map $f_2 \circ f_1$ is exactly the composition of $f_2$ with $f_1$ (seen as maps $\PP^1 \to \PP^1$). 

Now that we understand how to compose two elements of $Y_{d_1,n_1}$ and $Y_{d_2,n_2}$, we can define a composition morphism on a suitable open subset of $Y_{d_1,n_1} \times Y_{d_2,n_2}$. However, for simplicity we will restrict ourselves to the case $n_2=0$ below (see also Remark \ref{Rmk:n2zeroapology}).
\begin{Theo} \label{Theo:compomo}
 Let $d_1, d_2, n_1 \geq 0$ and let $U \subset Y_{d_1,n_1} \times Y_{d_2,0}$ be the set of pairs $(f_1, f_2)$ satisfying the conditions (P1), (P2) from Lemma \ref{Lem:compgeopoint}. Then $U$ is open with complement of codimension at least $2$ and the map 
 \begin{align*}
  \compomo: U & \to Y_{d_1 d_2,n_1}\\
  (f_1, f_2) & \mapsto f_2 \circ f_1
 \end{align*}
 is an algebraic morphism. It is an extension of the natural morphism 
 \begin{align*}
 Y_{d_1,n_1} \times Y_{d_2,0} \supset \text{Rat}_{d_1} \times \left( (\PP^1)^n \setminus \Delta \right) \times \text{Rat}_{d_2} &\to \text{Rat}_{d_1 d_2} \times \left( (\PP^1)^n \setminus \Delta \right)\\
 (\phi_1, p_1, \ldots, p_n, \phi_2) &\mapsto (\phi_2 \circ \phi_1, p_1, \ldots, p_n).
 \end{align*}
\end{Theo}
\begin{proof}
 To show that $U$ is open, we will prove that the points $(f_1,f_2)$ violating conditions (P1), (P2) form closed subsets of $Y_{d_1,n_1} \times Y_{d_2,0}$, respectively. Let 
 \[\begin{CD}
 \mathcal{C}_1    @>(\pi_1,\phi_1)>>  \PP^1 \times \PP^1\\
@VFVV        \\
Y_{d_1,n_1}    
\end{CD}\]
 be the universal curve\footnote{Here we mean $Y_{d_1,n_1+1}$ with $F$ the forgetful map of the $(n+1)$-st marking.} over $Y_{d_1, n_1}$ and consider the map 
 \[e: \mathcal{C}_1 \times Y_{d_2,0} \to Y_{d_2,0} \times \PP^1, (c,f_2) \mapsto (f_2, \phi_1(c)).\]
 Then the preimage $V$ of the closed set $\text{Vert}_{d,n} \subset Y_{d_2,0} \times \PP^1$ from Lemma \ref{Lem:birationalsection} under $e$ is closed and consists of those $(c,f_2)$ with $\phi_1(c) \in \text{Vert}(f_2)$. 
 
 Secondly, the set $N \subset \mathcal{C}_1$ of $c$ such that $c$ is a node of the curve $(\mathcal{C}_1)_{F(c)}$ is closed. Finally, away from $N$, the relative dualizing sheaf $\omega_{\mathcal{C}_1 / Y_{d_1, n_1}}$ is a line bundle and the set $R$ of $c \in \mathcal{C}_1 \setminus N$ such that $d \phi_1|_c =0$ is the vanishing set of a section of $\left( \phi_1^* \Omega_{\PP^1}^1 \right)^\vee \otimes \omega_{\mathcal{C}_1 / Y_{d_1, n_1}}$, so it is closed in $\mathcal{C}_1 \setminus N$.
 
 Now the set of all pairs $(c, f_2)$ such that $c \in \phi_1^{-1}(\text{Vert}(f_2))$ is not a reduced, smooth point of $(\mathcal{C}_1)_{F(c)}$  is exactly $V \cap (N \times Y_{d_2,0} \cup R \times Y_{d_2,0})$, and thus closed. The locus of points $(f_1, f_2) \in Y_{d_1,n_1} \times Y_{d_2,0}$ violating (P1) is the image of this set under the proper map $F \times \text{id}_{Y_{d_2,0}}$ and hence also closed. 
 
 The set of points in $Y_{d_1,n_1} \times Y_{d_2,0}$ violating (P2) is  the union of the preimages of $V$ under the morphisms
 \[\sigma_j \times \text{id} : Y_{d_1,n_1} \times Y_{d_2,0} \to \mathcal{C}_1 \times Y_{d_2,0}\]
 for $j=1, \ldots, n_1$, and hence also closed.
 
 For the codimension of the complement of $U$, note that for $(f_1,f_2) \in Y_{d_1,n_1} \times Y_{d_2,0} \setminus U$ to violate (P1), we need $f_2$ to be in the boundary of $Y_{d_2,0}$, giving codimension $1$ already. For a fixed such $f_2$ a generic $f_1=(\pi_1, \phi_1)$  will satisfy that all preimages of $q \in \text{Vert}(f_2)$ under $\phi_1$ are reduced, smooth points. Similarly, for condition (P2) to be violated, the map $f_2$ must be in the boundary and as $\phi_1^{-1}(\pi_2(\text{Vert}(f_2)))$ is a finite set for a generic $\phi_1 \in \text{Rat}_{d_1}$, a generic choice of the markings $p_{1,j}$ will satisfy (P2).
 
 To finish the proof, we need to show that $\compomo : U \to Y_{d_1 d_2, n_1}$ is an algebraic morphism. To do this, we will construct a family of stable maps in $Y_{d_1 d_2, n_1}(U)$ inducing $\compomo$. Let 
 \[\begin{CD}
 \mathcal{C}_i    @>(\pi_i,\phi_i)>>  \PP^1 \times \PP^1\\
@VF_iVV        \\
Y_{d_i,n_i}    
\end{CD}\]
 be the universal families\footnote{\label{foot:stackfamily}This is a bit unprecise, as of course $Y_{d_i,n_i}$ is not a fine moduli space. Here is what we actually do: over the smooth DM-stacks $\mathcal{Y}_{d_i,n_i}$ we do have a universal family. This we pull back under a smooth, surjective morphism $\tilde Y_{d_i,n_i} \to \mathcal{Y}_{d_i,n_i}$ from a smooth, finite type scheme $\tilde Y_{d_i,n_i}$. Afterwards we proceed as indicated in the proof, constructing a morphism from a suitable open subset of $\tilde Y_{d_1,n_1} \times \tilde Y_{d_0,0}$ to $Y_{d_1 d_2,n_1}$. By descent, we obtain a map defined on an open substack $\mathcal{U}$ of $\mathcal{Y}_{d_1,n_1} \times \mathcal{Y}_{d_0,0}$. This factors through the corresponding open subset $U$ of its coarse moduli space.} over $Y_{d_1,n_1}$, $Y_{d_2,0}$ (with $n_2 =0$) and let 
 \[\sigma_1, \ldots, \sigma_{n_1} : Y_{d_1, n_1 } \to \mathcal{C}_1\]
 be the sections of $F_1$ corresponding to the markings. We define
 \[\mathcal{C} = \mathcal{C}_1 \times_{\phi_1, \PP^1,\pi_2} \mathcal{C}_2 = \{(c_1, c_2) \in \mathcal{C}_1 \times \mathcal{C}_2: \phi_1(c_1) = \pi_2(c_2)\}.\]
 This maps to $(\PP^1)^3$ via
 \begin{align*} 
   \mu : \mathcal{C} &\to \PP^1 \times \PP^1 \times \PP^1\\
   (c_1, c_2) &\mapsto (\pi_1(c_1),\underbrace{\phi_1(c_1)}_{=\pi_2(c_2)}, \phi_2(c_2)).
 \end{align*}
 Moreover, there is the map $F=F_1 \times F_2 : \mathcal{C} \to Y_{d_1,n_1} \times Y_{d_2,0}$. Over $U$, we now want to define sections $\widehat \sigma_1, \ldots, \widehat \sigma_{n_1}$ of $F$ corresponding to markings of our family of stable maps. Let 
 \[s: Y_{d_2,0} \times \PP^1 \setminus \text{Vert}_{d_2,0} \to Y_{d_2,1}=\mathcal{C}_2\]
 be the map from Lemma \ref{Lem:birationalsection}. For $j=1, \ldots, n_1$ consider
 \begin{align*}
  \gamma_j : U & \to Y_{d_2,0} \times \PP^1 \setminus \text{Vert}_{d_2,0} \\
  (f_1,f_2) & \mapsto (f_2, (\phi_1 \circ \sigma_j)(f_1)).
 \end{align*}
 The fact that $\gamma_j$ indeed has image outside of $\text{Vert}_{d_2,0}$ is exactly guaranteed by condition (P2). Thus the composition $s \circ \gamma_j$ has values in $\mathcal{C}_2$ and we define
 \[\widehat \sigma_j = \sigma_j \times (s \circ \gamma_j) : U \to \mathcal{C}_1 \times \mathcal{C}_2.\]
 One checks that the map $\widehat \sigma_j$ is a section of $F_1 \times F_2$ over $U$ by using Lemma \ref{Lem:birationalsection}. It actually has image in $\mathcal{C} \subset \mathcal{C}_1 \times \mathcal{C}_2$ because
 \[\pi_2 \circ s \circ \gamma_j  = \phi_1 \circ \sigma_j.\]
 Here again we use that $s$ is a section of the map 
 \[E=F_2 \times \pi_2 : \mathcal{C}_2 \to Y_{d_2,0} \times \PP^1.\]
 We now show that the data 
 \[\mathcal{F}=\left(F : \mathcal{C}|_U \to U; \widehat \sigma_1, \ldots, \widehat \sigma_{n_1}; \mu \right)\]
 gives a well-defined element of $\overline{\mathcal{M}}_{0,n_1}(\PP^1 \times \PP^1 \times \PP^1, (1,d_1,d_1d_2))(U)$ which agrees with $\widehat{f_2 \circ f_1}$ over a geometric point $(f_1,f_2) \in U$. Then we are done, because by \cite[Theorem 3.6]{behrendmanin}, we can push forward this family via the projection $(\PP^1)^3 \to (\PP^1)^2$ on first and third factor, obtaining an element in $Y_{d_1d_2,n_1}(U)$ as desired. Fibrewise, this corresponds to composing $\mu$ with the projection above and then stabilizing. But this is exactly the procedure by which we obtained $f_2 \circ f_1$ from $\widehat{f_2 \circ f_1}$. Hence we obtain the desired family in $Y_{d_1d_2,n_1}(U)$ and the proof is finished.
 
 The map $F$ is projective, because $F_1, F_2$ were projective and we precompose $F_1 \times F_2$ with the closed embedding $\mathcal{C} \hookrightarrow \mathcal{C}_1 \times \mathcal{C}_2$. By Lemma \ref{Lem:compgeopoint}, over a geometric point $f=(f_1, f_2) \in U$, the fibre $\mathcal{C}_f$ of $F$ is an $n_1$-pointed, genus $0$ quasi-stable curve and the restriction of $\mu$ makes it a stable map of degree $(1,d_1, d_1d_2)$. This shows everything except that the family $\pi$ is flat. But observe that $U$, $\mathcal{C}_1 \times \mathcal{C}_2$ are smooth (see the footnote \textsuperscript{\ref{foot:stackfamily}} above) and $\mathcal{C} \subset \mathcal{C}_1 \times \mathcal{C}_2$ is an effective Cartier divisor, namely the pullback of the diagonal $\Delta \subset \PP^1 \times \PP^1$ under the surjective morphism 
 \[\mathcal{C}_1 \times \mathcal{C}_2 \xrightarrow{(\phi_1, \pi_2)} \PP^1 \times \PP^1.\]
 Hence $\mathcal{C}$ is Cohen-Macaulay and equidimensional. Then by the Miracle flatness theorem (\cite[Theorem 23.1]{miracleflatness}), the map $F$ is flat over $U$, as all fibres are of dimension $1$. 
\end{proof}
\begin{Rmk} \label{Rmk:n2zeroapology}
 In Lemma \ref{Lem:compgeopoint}, we clearly have laid the foundations for also defining a composition morphism on a subset of $Y_{d_1,n_1} \times Y_{d_2,n_2}$, with target $Y_{d_1d_2,n_1 + d_1 n_2}/(S_{d_1})^{n_2}$. However, handling this would force us to deal with the codimension $1$ condition that two preimages of a marked point on $f_2$ under $\phi_1$ come together. Furthermore, we would need to parametrize these preimages with sections, up to ordering. As this introduces more technical complications in an already quite technical proof, and as we are not going to use it in the future, we have omitted this. However, it could be of interest, as these $d_1 n_2$ new markings are essentially the preimage in $C_1$ of the $n_2$ markings on the curve $C_2$. In connection with the self-composition morphisms below, this could be used to study backward orbits of marked points.
\end{Rmk}
We have now an explicit description of the rational map $\compomo : Y_{d_1,n_1} \times Y_{d_2,0} \dashrightarrow Y_{d_1 d_2,n_1}$, defined away from a codimension $2$ locus in the domain. Note that because the spaces $Y_{d_i,n_i}$ have finite quotient singularities, every Weil-divisor is $\mathbb{Q}$-Cartier. Using also that they are rational (as they contain the rational open subvarieties $\text{Rat}_{d_i} \times  \left( (\PP^1)^{n_i}  \setminus \Delta \right)$), we have
\begin{align*}
 \text{Pic}(U)\otimes_{\mathbb{Z}} \mathbb{Q} &= \text{Cl}(U)\otimes_{\mathbb{Z}} \mathbb{Q} = \text{Cl}(Y_{d_1,n_1} \times Y_{d_2,0})\otimes_{\mathbb{Z}} \mathbb{Q}\\
 &= \text{Pic}(Y_{d_1,n_1} \times Y_{d_2,0})\otimes_{\mathbb{Z}} \mathbb{Q}\\
 &= \text{Pic}(Y_{d_1,n_1})\otimes_{\mathbb{Z}} \mathbb{Q} \oplus \text{Pic}(Y_{d_2,0})\otimes_{\mathbb{Z}} \mathbb{Q}.
\end{align*}
Thus it makes sense to study the pullback-morphism
\[\compomo^* : \text{Pic}(Y_{d_1 d_2,n_1})\otimes_{\mathbb{Z}} \mathbb{Q} \to \text{Pic}(Y_{d_1,n_1})\otimes_{\mathbb{Z}} \mathbb{Q} \oplus \text{Pic}(Y_{d_2,0})\otimes_{\mathbb{Z}} \mathbb{Q}\]
and to express it in terms of the generators found in Theorem \ref{Theo:Generators}. For handling the pullbacks of the boundary divisors, we will use the resultant.

Recall from Section \ref{Sect:degdmaps} the space 
\[Z_d = \PP(H^0(\PP^1 \times \PP^1, \mathcal{O}(d,1)))\]
parametrizing pairs $[(F,G)]$ of homogeneous polynomials of degree $d$ up to simultaneous scaling. The space $\text{Rat}_d \subset Z_d$ is the locus where $F,G$ do not have a common root. This condition can be expressed in terms of the resultant $\text{Res}(F,G)$. Going to affine coordinates, it has a very handy description in terms of the roots of $F,G$. Assume these polynomials are given by 
\begin{align*}
 F(x,1) = \lambda (x-a_1) \cdots (x-a_d), G(x,1)= \mu (x-b_1) \cdots (x-b_d).
\end{align*}
Then
\begin{align*}
 \text{Res}(F,G) = (\lambda \mu)^d \prod_{i,j=1}^d (a_i - b_j).
\end{align*}
One checks that this is a homogeneous polynomial of degree $2d$ in the coefficients of $F,G$. Hence it gives a section $\text{Res} \in H^0(Z_d, \mathcal{O}_{Z_d}(2d))$, which vanishes exactly on the locus where one of the $a_i$ equals one of the $b_j$, i.e. where $F,G$ have a common root. 

Now recall the map 
\[j : Y_{d,n} \to Z_d = \PP(H^0(\PP^1 \times \PP^1, \mathcal{O}(d,1)))\]
from Lemma \ref{Lem:jConstruction}. Pulling back $\text{Res}$ via $j$ gives us a section 
\[j^* \text{Res} \in H^0(Y_{d,n}, j^*\mathcal{O}_{Z_d}(2d)), \]
which vanishes exactly on the boundary divisors $D_{B,k}$ for $k \geq 1$. We will now see which multiplicities it has on these divisors.
\begin{Lem} \label{Lem:resultant}
 Let $d,n \geq 0$, then with notation as above, we have
 \begin{equation} \label{eqn:resmultiplicities}
  \text{div}(j^* \text{Res}) = \sum_{(B,k)} k^2 D_{B,k},
 \end{equation}
 where the sum runs over all pairs $(B,k)$ giving boundary divisors $D_{B,k}$ in $Y_{d,n}$.
\end{Lem}
\begin{proof}
 We know that $\text{div}(j^* \text{Res})$ is supported on the boundary divisors $D_{B,k}$ with $k \geq 1$, so it remains to check the multiplicity is $k^2$. 
 To determine this multiplicity, we can use our test curves $C_{B,k}$ from Proposition \ref{Pro:CBk}. Recall that these were the image cycles of maps $\psi_{B,m} : \PP^1 \to Y_{d,n}$ such that over $[x:1] \in \PP^1 \setminus \{0\}$ the induced rational map was given by
 \[\phi_{x}(z) = \frac{(z-a_1 )\ldots (z-a_{d-k})(z - b_1 x) \ldots (z - b_k x)}{(z-c_1)\ldots (z-c_{d-k})( z - d_1 x) \ldots ( z - d_k x)}  \]
 for $a_i, b_j, c_i, d_j \in \mathbb{C}$, $i=1, \ldots, d-k$, $j=1, \ldots, k$ sufficiently general. 
 We have then seen, that the corresponding map $\psi_{B,k}: \PP^1 \to Y_{d,0}$ intersects $D_{B, k}$ transversally over the point $[0:1]$. Hence the multiplicity of $D_{B, k}$ in $\text{div}(j^* \text{Res})$ is exactly the order of vanishing of $\text{Res} \circ j \circ \psi_{B,k}$ at $[0:1]$. This is
 \begin{align*}
  &\text{ord}_{x=0} \text{Res}(j(\psi_{B,k}(x)))\\
  =&\text{ord}_{x=0} \prod_{\substack{1 \leq i \leq d-k\\ 1 \leq j \leq d-k}} (a_i - c_j) \prod_{\substack{1 \leq i \leq d-k\\ 1 \leq j \leq k}} (a_i - d_jx) \prod_{\substack{1 \leq i \leq k\\ 1 \leq j \leq d-k}} (b_i x - c_j) \prod_{\substack{1 \leq i \leq k\\ 1 \leq j \leq k}} (b_i x - d_j x)\\
  =& k^2. 
 \end{align*}
 This finishes the proof.
\end{proof}
We see above that the resultant cuts out all the boundary divisors $D_{B,k}$ for $k \geq 1$. However, using the evaluation maps, we can also cut out the remaining divisors $D_{B,0}$. 
\begin{Lem} \label{Lem:ijdivisor}
 Let $d \geq 0, n \geq 2$ and let $1 \leq i,j \leq n$ with $i \neq j$. Then for the map 
 \[e_{i,j} = (\pi_1 \circ \text{ev}_i) \times (\pi_1 \circ \text{ev}_j) : Y_{d,n} \to \PP^1 \times \PP^1\]
 we have
 \[e_{i,j}^* \Delta_{\PP^1} = \sum_{k, B \ni i,j} D_{B,k}, \]
 where we mean pullback in the sense of pullback of pseudo-divisors (cf. \cite[Section 2.2]{inttheory}).
\end{Lem}
\begin{proof}
 It is clear that the preimage of the diagonal $\Delta_{\PP^1}$ is exactly the union of the $D_{B,k}$ with $i,j \in B$. To see that all multiplicities are $1$, note that the curves $C_{B,k}$ intersect $D_{B,k}$ transversally in exactly two points. On the other hand, the diagonal $\Delta_{\PP^1}$ is of class $\mathcal{O}(1,1)$ in $\text{Pic}(\PP^1 \times \PP^1)$. Thus its pullback by $e_{i,j}$ is of class $\mathcal{H}_{i,1} + \mathcal{H}_{j,1}$. But the intersection of $C_{B,k}$ with this is also exactly $2$. Hence the multiplicity of $D_{B,k}$ in $e_{i,j}^* \Delta_{\PP^1}$ must be exactly $1$.
\end{proof}

\begin{Pro} \label{Pro:compopullback}
 For $1 \leq i \leq n_1$ we have
 \begin{align*}
  \compomo^* \mathcal{H}_{i,1} &= ( \mathcal{H}_{i,1} , 0),\\
  \compomo^* \mathcal{H}_{i,2} &= ( d_2 \mathcal{H}_{i,2} , D_p),
 \end{align*}
 where $D_p$ is the divisor from Proposition \ref{Pro:Dp}. 
 
 On the other hand for a boundary divisor $D_{B,k}$ on $Y_{d_1 d_2,n_1}$ we have
 \[\compomo^* D_{B,k} = \underbrace{(D_{B,l},0)}_{\text{for }k=d_2 l} + \underbrace{(0,d_1 D_{\emptyset,k})}_{\text{for }0<k\leq d_2, B= \emptyset}.
\]
 Finally, for $d_1=0$ we have
 \[\compomo^* \mathcal{G} = (d_2 \mathcal{G}, D_p).\]
 For $d_2 = 0$ we obtain
 \[\compomo^* \mathcal{G} = (0, \mathcal{G}).\]
\end{Pro}
\begin{proof}
 Our general strategy will be to determine both components of the pullback divisors separately. This is done by restricting $\compomo$ to slices $\{f_1\} \times Y_{d_2,0}$ and $Y_{d_1, n_1} \times \{f_2\}$ for $f_1 \in Y_{d_1, n_1}$ and $f_2 \in Y_{d_2,0}$, and then pulling back.
 
 For the divisors $\mathcal{H}_{i,j} = \text{ev}_{i,j}^* \mathcal{O}(1)$ coming from the evaluation maps
 \[\text{ev}_i = (\text{ev}_{i,1},\text{ev}_{i,2})  : Y_{d_1 d_2,n_1} \to \PP^1 \times \PP^1\]
 we see that for 
 \begin{align*}
  f_1 &= ((\pi_1,\phi_1) : C_1 \to \PP^1 \times \PP^1; p_1, \ldots, p_n) \in Y_{d_1,n_1}\\
  f_2 &= ((\pi_2,\phi_2) : C_2 \to \PP^1 \times \PP^1) \in Y_{d_2,0}
 \end{align*}
 with $\phi_1(p_i) \notin \text{Vert}(f_2)$, we have
 \[\text{ev}_i(\compomo(f_1,f_2)) = (\pi_1(p_i), \phi_2( \pi_2^{-1} ( \phi_1 (p_i)))).\]
 Hence $(\text{ev}_{i,1} \circ \compomo)(f_1,f_2) = \text{ev}_{i,1}(f_1)$, so $\compomo^* \mathcal{H}_{i,1}$ has the claimed form.
 
 For $\compomo^* \mathcal{H}_{i,2}$ we apply the strategy described at the beginning of the proof: for $f_2 \in \text{Rat}_{d_2}$ a fixed degree $d_2$ map, we see $\text{ev}_{i,2}(\compomo(f_1,f_2)) = (f_2 \circ \text{ev}_{i,2})(f_1)$. As $f_2^* \mathcal{O}(1) = \mathcal{O}(d_2)$, we obtain $d_2 \mathcal{H}_{i,2}$ in the first component of $\compomo^* \mathcal{H}_{i,2}$. On the other hand, for $f_1$ fixed, the position $\phi_1(p_i)$ is fixed and for $q \in \PP^1$ fixed, the set 
 \[\left(\text{ev}_{i,2} \circ \compomo |_{\{f_1\} \times Y_{d_2,0}}\right)^{-1} (\{q\})\]
 is exactly the locus of $f_2 \in Y_{d_2,0}$ mapping $\phi_1(p_i)$ to $q$, and thus represents the divisor class $D_p$.
 
 The computation of the pullbacks of $\mathcal{G}$ is similar (after all, for $d_1=0, n_1>0$ we have $\mathcal{G} = \mathcal{H}_{1,2}$) and will be omitted.
 
 Now for the pullbacks of the boundary divisors $D_{B,k}$ note that if both $f_1, f_2$ have smooth source curves $C_1,C_2$, then the source curve of $f_2 \circ f_1$ is smooth too. Thus the pullback of $D_{B,k}$ is supported on the boundaries of $Y_{d_1,n_1}$ and $Y_{d_2,0}$ in $Y_{d_1,n_1} \times Y_{d_2,0}$.
 
 For $f_1 \in D_{B,l}$ and $f_2 \in \text{Rat}_{d_2}$, the composition $f_2 \circ f_1$ has exactly one vertical component of degree $d_2l$ with the markings in $B$ on it. If $f_1$ has smooth source curve and $f_2 \in D_{\emptyset, l}$ then $f_2 \circ f_1$ has $d_1$ vertical sections of degree $l$. It remains to show that in the first case the multiplicity of $D_{B,l}$ in $\compomo^* D_{B,d_2 l}$ is $1$ and in the second case the multiplicity of $D_{\emptyset,k}$ in $\compomo^* D_{\emptyset,k}$ is $d_1$.
 
 To show this, we are going to use our test curves $C_{B,m}$ (see also the proof of Lemma \ref{Lem:resultant}). 
 
 For the first case fix $f_2=(z \mapsto z^{d_2})$. Assume first that $d_2l>0$, then by Lemma \ref{Lem:resultant} the divisor $D_{B,d_2l}$ appears with a coefficient of $(d_2 l)^2$ in $\text{div}(j^* \text{Res})$. Now let $f_1$ vary along the test curve $\psi_{B,l} : \PP^1 \to Y_{d_1,n_1}$. Then the rational map induced by $f_2 \circ f_1$ over $x \in \mathbb{C}$ is
 \[z \mapsto \left( \frac{(z-a_1 )\ldots (z-a_{d-l})(z - b_1 x) \ldots (z - b_l x)}{(z-c_1)\ldots (z-c_{d-l})( z - d_1 x) \ldots ( z - d_l x)}  \right)^{d_2}.\]
 Over $x=0$ the zeroes $b_1x, \ldots, b_l x$ of the numerator converge to the zeroes $d_1x, \ldots, d_l x$ of the denominator, and all zeroes now have multiplicity $d_2$. Thus the resultant vanishes to order 
 \[\text{ord}_{x=0} \prod_{1 \leq i,j \leq l} (b_ix - c_jx)^{d_2^2} = d_2^2 l^2.\]
 As this is the same as the multiplicity of $D_{B,d_2l}$ in $\text{div}(j^* \text{Res})$, the multiplicity of $D_{B,l}$ in $\compomo^* D_{B,d_2l}$ is $1$.
 
 If $d_2 l=0$ we make a similar argument, but now use Lemma \ref{Lem:ijdivisor}. For $i,j \in B$ we have that $D_{B,0}$ appears in $(e_{i,j}^{Y_{0,n_1}})^* \Delta_{\PP^1}$ with multiplicity $1$ and that $D_{B,l}$ also has multiplicity $1$ in $(e_{i,j}^{Y_{d_1,n_1}})^* \Delta_{\PP^1}$. But $e_{i,j}^{Y_{d_1 d_2,n_1}} \circ \compomo = e_{i,j}^{Y_{d_1,n_1}}$, so $D_{B,l}$ has multiplicity $1$ in $\compomo^* D_{B,0}$.
 
 Now fix $f_1 \in Y_{d_1,n_1}$ as the rational map $z \mapsto z^{d_1}$ with markings in general position and vary $f_2$ along $\psi_{\emptyset,k}$, then the rational map induced by $f_2 \circ f_1$ is 
 \[\phi_{x}(z) = \frac{(z^{d_1}-a_1 )\ldots (z^{d_1}-a_{d-k})(z^{d_1} - b_1 x) \ldots (z^{d_1} - b_k x)}{(z^{d_1}-c_1)\ldots (z^{d_1}-c_{d-k})( z^{d_1} - d_1 x) \ldots ( z^{d_1} - d_k x)}.\]
 At $x=0$ the $d_1 l$ zeroes $(b_i x)^{1/d_1} \rho^j$, $i=1, \ldots, l$, $\rho$ a $d_1$-th root of unity, $j=0, \ldots, d_1 -1$, come together with the $d_1 l$ zeroes $(d_k x)^{1/d_1} \rho^m$ of the denominator. Thus over $x=0$ the resultant vanishes to order
 \[\text{ord}_{x=0} \prod_{\substack{1 \leq i,k \leq l\\ 1 \leq j,m \leq d_1}} ((b_i x)^{1/d_1} \rho^j - (d_k x)^{1/d_1} \rho^m) = \frac{d_1^2 l^2}{d_1} = d_1 l^2.\]
 As on $Y_{d_1 d_2,n_1}$ the resultant vanished to order $l^2$ along $D_{\emptyset,l}$, the multiplicity of $D_{\emptyset,k}$ in $\compomo^* D_{\emptyset,k}$ is $d_1$.
\end{proof}
\subsection{Self-composition and multiplier loci} \label{Sect:selfcompo}
We can now apply the results from the last section to our moduli spaces of self-maps. Indeed, the rational map 
\[\mathfrak{sc}_m: Y_{d,n} \xrightarrow{\text{id} \times \prod_{i=0}^{m-1} F} Y_{d,n} \times \prod_{i=1}^{m-1} Y_{d,0} \overset{\mathfrak{c}\times \text{id}}{\dashrightarrow} Y_{d^2,n}\times \prod_{i=1}^{m-2} Y_{d,0} \cdots \overset{\mathfrak{c}\times \text{id}}{\dashrightarrow} Y_{d^m,n}\]
is $G$-equivariant. Here $F$ is the map $Y_{d,n} \to Y_{d,0}$ forgetting all markings. For $\textbf{d}=(d|d_1, \ldots, d_n)$ admissible, it thus induces a rational map of quotients 
\[\mathfrak{sc}_m : M(d|d_1, \ldots, d_n) \dashrightarrow M(d^m|d_1, \ldots, d_n).\]
On the open locus in $M(d|d_1, \ldots, d_n)$ of maps from a smooth curve to itself, this is given by
\[(\phi: C \to C; p_1, \ldots, p_n) \mapsto (\phi^{\circ m} : C \to C; p_1, \ldots, p_n).\]
As a rational map from a normal to a projective variety is defined away from a set of codimension at least $2$, it makes sense to pull back divisor classes on $M(d^m|d_1, \ldots, d_n)$. In the following Proposition, $\mathfrak{sc}_m$ can either denote the rational map $Y_{d,n} \dashrightarrow Y_{d^m,n}$ or the induced map of the quotients. Because of the usual isomorphisms of rational Picard groups, it will suffice to prove it for the former.
\begin{Pro} \label{Pro:selfcompopullback}
 Let $m \geq 1$. For $1 \leq i \leq n_1$ we have
 \begin{align*}
  \mathfrak{sc}_m^* \mathcal{H}_{i,1} &= \mathcal{H}_{i,1},\\
  \mathfrak{sc}_m^* \mathcal{H}_{i,2} &= d^{m-1} \mathcal{H}_{i,2}+ \left(\sum_{j=0}^{m-2}d^j \right)D_p,
 \end{align*}
 where $D_p$ is the divisor from Proposition \ref{Pro:Dp}. 
 
 On the other hand for a boundary divisor $D_{B,k}$ on $Y_{d_1 d_2,n_1}$ we have
 \[\mathfrak{sc}_m^* D_{B,k} = D_{B,l} \text{, for }B \neq \emptyset, k=d^{m-1} l
\]
 and $\mathfrak{sc}_m^* D_{\emptyset,k}$ can inductively be computed by
 \begin{align*}\mathfrak{sc}_m^* D_{\emptyset,k} &= \underbrace{d^{m-1} \sum_B D_{B,k}}_{\text{for }k=1, \ldots, d} + \underbrace{\mathfrak{sc}_{m-1}^* D_{\emptyset,l}}_{\text{for }k=dl, l\geq 1},\\
 \mathfrak{sc}_1^* D_{\emptyset,k}&=D_{\emptyset,k}.\end{align*}
 Finally, for $d=0$ we have
 \[\mathfrak{sc}_m^* \mathcal{G} = \mathcal{G}.\]
\end{Pro}
\begin{proof}
 Note that the map $\mathfrak{sc}_m$ factors as
 \[Y_{d,n} \xrightarrow{\mathfrak{sc}_{m-1} \times F}  Y_{d^{m-1},n} \times Y_{d,0} \xrightarrow{\mathfrak{c}} Y_{d^m,n}.\]
 We can thus show the formulas above in a simultaneous induction on $m$. We first recall the pullbacks of all divisors above under $\mathfrak{c}$:
 \begin{align*}
  Y_{d^{m-1},n} \times Y_{d,0} &\xrightarrow{\mathfrak{c}} Y_{d^m,n}\\
  (\mathcal{H}_{i,1},0) &\mapsfrom \mathcal{H}_{i,1}\\
  (d \mathcal{H}_{i,2},D_p) &\mapsfrom \mathcal{H}_{i,2}\\
  \underbrace{(D_{B,l},0)}_{\text{for }k=dl}+\underbrace{(0, d^{m-1} D_{\emptyset,k})}_{\text{for }B=\emptyset,k \leq d} &\mapsfrom D_{B,k}\\
  (0,\mathcal{G}) &\mapsfrom \mathcal{G}.
 \end{align*}
 Note also $F^*(D_p)=D_p$. Then the formulas for $m=1$ are immediate as $\mathfrak{sc}_1 = \text{id}$ and the induction step follows from the factorization above.
\end{proof}
We are going to apply these results by defining another divisorial locus in $M(d|d_1, \ldots, d_n)$ and computing its class in the rational Picard group. This locus will be induced from a $G$-invariant locus on $Y_{d,n}$, so we will work on this space instead.

For $\varphi: \PP^1 \to \PP^1$, a point $p \in \PP^1$ is called of \emph{period $m$} if $\varphi^{\circ m}(p)=p$. We say that $p$ is of \emph{strict period $m$} if $\varphi^{\circ k}(p) \neq p$ for all $1 \leq k <m$. For a point $p$ of period $m$, it makes sense to interpret the differential of $\varphi^{\circ m}$ at $p$ as a complex number $\lambda$  (because $d \varphi^{\circ m} : T_p \PP^1 \to T_p \PP^1$ is given by multiplication with some $\lambda \in \mathbb{C}$). We call $\lambda$  the \emph{multiplier} of $\varphi$ at $p$.

Given a fixed $\lambda \in \mathbb{C}$, we denote by
\[\text{Per}_m^\sim(\lambda) = \{\varphi \in \text{Rat}_d : \varphi \text{ has a point }p\text{ of period }m\text{ with multiplier }\lambda \}.\]
It is easy to see that this locus is a closed subset of codimension $1$. In the following, we are going to decompose it into several components, corresponding to different orbit types, and associate natural multiplicities to these components.

Given a degree $d$ map $\varphi \in \text{Rat}_d$, we associate to it the normed polynomial
\begin{equation}
 \Psi_d(\varphi,\mu) = \prod_{j=1}^{d+1} (\mu - \xi_j) = \mu^{d+1} + \sigma_1(\varphi) \mu^{d} + \ldots + \sigma_{d+1}(\varphi),
\end{equation}
whose zeroes $\xi_1, \ldots, \xi_{d+1}$ are exactly the multipliers at the fixed points of $\varphi$ (fixed points counted with multiplicities). Note that the $\xi_j$ are only defined up to reordering, but the $\sigma_i$ are (up to a sign) the elementary symmetric polynomials in the $\xi_j$ and thus are well-defined. Even more, in \cite[Proposition 4.2]{silverman} it is shown that the $\sigma_i$ are algebraic on $\text{Rat}_d$. This is done by interpreting the function $\Psi_d(-,\mu)$ as the determinant of an endomorphism on a vector bundle on $\text{Rat}_d$, a construction we will recall in the proof of Proposition \ref{Pro:Permcompatibility}.

From the definition, it is immediate that 
\[\text{Per}_1^\sim(\lambda) = V(\Psi_d(-,\lambda)),\]
which shows that $\text{Per}_1(\lambda)$ is in a natural way a closed subscheme of $\text{Rat}_d$. Using the self-composition morphism $\mathfrak{sc}_m : \text{Rat}_d \to \text{Rat}_{d^m}$ one sees that
\[\text{Per}_m^\sim(\lambda) = V( \mathfrak{sc}_m^* \Psi_{d^m}(-,\lambda)) = \mathfrak{sc}_m^{-1}(\text{Per}_1^\sim(\lambda)).\]
This induces a scheme structure on $\text{Per}_m^\sim(\lambda)$, which however will not be reduced in general. For $d=2$, Milnor defines in the paper \cite{milnorquaddyn} closely related subvarieties of $M_2 = \text{Rat}_2 \sslash \text{PGL}_2$. His construction immediately generalizes to the case of general $d$ and can be carried out on $\text{Rat}_d$. In the following, we want to relate the two constructions. We will denote by $P_m(\lambda) \subset \text{Rat}_d$ the loci of points of strict period $m$ defined by Milnor and now recall their construction.

First note that a fixed point of $\varphi^{\circ m}$ is exactly a periodic point for $\varphi$ with a strict period $n$ dividing $m$. Hence the function $\mathfrak{sc}_m^* \Psi_{d^m}(-,\lambda)$ on $\text{Rat}_d$ naturally decomposes into a product 
\[\Psi_{d^m}(\mathfrak{sc}_m(\varphi),\lambda) = \prod_{n|m} G_{n}(\varphi,\lambda),\]
where in $G_n$ we collect all the factors $\lambda - \xi$ for $\xi$ a multiplier at a point of strict period $n$. One sees immediately, that on an orbit $p_0=p, p_1=\varphi(p), \ldots, p_{m-1}=\varphi^{\circ m-1}(p)$ of a point of strict period $m$, all the $p_i$ are also of period $m$ and have the same multiplier as $p$. Hence for $\varphi$ fixed, all roots of $G_{m}(\varphi,\lambda)$ occur with multiplicity divisible by $m$ and we can write $G_m(\varphi,\lambda) = \tilde G_m(\varphi,\lambda)^m$. One can see that for fixed $\lambda$, the function $\tilde G_m(\varphi,\lambda)$ depends algebraically on $\varphi \in \text{Rat}_d$ and Milnor defines
\[P_m(\lambda) = V(\tilde G_m(- , \lambda)).\]
In the following Proposition, denote by $[V]$ the algebraic cycle associated to a subscheme $V \subset X$ of a scheme $X$.
\begin{Pro}
 For $d,m$ fixed integers and $\lambda \in \mathbb{C} \setminus \{0\}$ we have
 \[[\text{Per}_m^\sim(\lambda)] = \sum_{n | m} \sum_{\mu^{m/n}=\lambda} n [P_n(\mu)] \in Z_{\codim=1}(\text{Rat}_d).\]
\end{Pro}
\begin{proof}
 By the discussion above, it is clear that $[\text{Per}_m^\sim(\lambda)]$ is an algebraic cycle supported on the union of codimension $1$ subvarieties of $\text{Rat}_d$ on the right side. It remains to compare the multiplicities. Note that the case $m=1$ is trivial, as here the definition of $\text{Per}_1^\sim(\lambda)$ and $P_1(\lambda)$ agree. For general $m$, we will first argue that it suffices to check the multiplicity at the terms $n=1$ and $n=m$.
 
 Indeed, assume we know that in $[\text{Per}_m^\sim(\lambda)]$ the coefficient of $[P_1(\mu)]$ is $1$ (for $\mu^m=\lambda$) and the coefficient of $[P_m(\lambda)]$ is $m$. Then for any $n|m$, the map $\mathfrak{sc}_m$ factors as $\mathfrak{sc}_{m/n} \circ \mathfrak{sc}_n$. The elements of $P_n(\mu)$ for $\mu^{m/n}=\lambda$ map to $P_1(\mu) = \text{Per}_1^\sim(\mu)$ under $\mathfrak{sc}_n$ and then by $\mathfrak{sc}_{m/n}$ to $\text{Per}_1(\lambda)$. But by assumption, $[P_1(\mu)]$ appears with a coefficient $1$ in $\mathfrak{sc}_{m/n}^{-1}(\text{Per}_1^\sim(\lambda))$ and in turn $P_n(\mu)$ appears with a coefficient of $n$ in $\mathfrak{sc}_{n}^{-1}(\text{Per}_1^\sim(\mu))$, as desired.
 
 Now for the $P_m(\lambda)$ we note that they were defined by the equation $\tilde G_m(-,\lambda)$, whereas in the defining equation $\Psi_{d^m}(\mathfrak{sc}_m(-),\lambda)$ of $\text{Per}_m^\sim(\lambda)$, the factor vanishing on the support of $P_m(\lambda)$ is $G_m(-,\lambda) = \tilde G_m(-,\lambda)^m$. This explains the coefficient $m$.
 
 As for the coefficient of $P_1(\mu)$ with $\mu^m=\lambda$, the corresponding factor of the function $\Psi_{d^m}(\mathfrak{sc}_m(-),\lambda)$ vanishing on it is $G_1(-,\lambda)$. Given $\varphi \in \text{Rat}_d$ having fixed points $p_1, \ldots, p_{d+1}$ with multipliers $\xi_1, \ldots, \xi_{d+1}$, the multipliers of $\varphi^{\circ m}$ at the $p_i$ are $\xi_1^m, \ldots, \xi_{d+1}^m$ and those are exactly the multipliers contributing to $G_1(\varphi,\lambda)$. For some primitive $m$-th root of unity $\rho$ we thus have
 \begin{align*}
  G_1(\varphi,\lambda) &= \prod_{j=1}^{d+1} \lambda - \xi_j^m = \prod_{j=1}^{d+1} \mu^m - \xi_j^m\\
  &= \prod_{i=0}^{m-1} \prod_{j=1}^{d+1} \mu \rho^i - \xi_j\\
  &= \prod_{i=0}^{m-1} \Psi_d(\varphi,\mu \rho^i).
 \end{align*}
 Hence, all the loci $\text{Per}_1^\sim(\mu \rho^i) = V(\Psi_d(-,\mu \rho^i))$ appear with multiplicity $1$ in $[\text{Per}_m^\sim(\lambda)]$ as claimed.
 \end{proof}
Now we will extend the definition of $\text{Per}_m^\sim(\lambda)$ to all of $Y_{d,0}$ (and $Y_{d,n}$). Consider the diagram
\[\begin{CD}
   Y_{d,n+1} @>\text{ev}_{n+1}=(\mu, \varphi) >> \PP^1 \times \PP^1\\
   @V\pi VV  @.\\
   Y_{d,n} @.
  \end{CD}
\]
where we interpret $Y_{d,n+1}$ as the universal curve over $Y_{d,n}$. Note that on the locus $Y_{d,n}^o$ this actually is the universal curve, as the elements $(f: \PP^1 \to \PP^1 \times \PP^1; p_1, \ldots, p_n) \in Y_{d,n}^o$ do not have automorphisms. The differential of $\text{ev}_{n+1}$ induces two sections
\begin{align*}
 d \mu &\in \text{Hom}(\mu^* \Omega_{\PP^1},\omega_\pi),\\
 d \varphi &\in \text{Hom}(\varphi^* \Omega_{\PP^1},\omega_\pi).
\end{align*}
On the preimage $\text{Fix} = \text{ev}_{n+1}^{-1}(\Delta_{\PP^1})$ of the diagonal $\Delta_{\PP^1} \subset \PP^1 \times \PP^1$, the bundles $\mu^* \Omega_{\PP^1}, \varphi^* \Omega_{\PP^1}$ coincide. Hence it makes sense to take the vanishing set $V=V((d \varphi - \lambda d \mu)|_{\text{Fix}}) \subset \text{Fix}$. Denote by
\[\text{Per}_1(\lambda) = \pi_* [V] \in Z_{\codim=1}(Y_{d,n})\]
the pushforward under $\pi$ of the algebraic cycle associated to $V$. By construction, the support $|\text{Per}_1(\lambda)|$ of this effective divisor is the locus of $((f_1, f_2) : C \to \PP^1 \times \PP^1; p_1, \ldots, p_n) \in Y_{d,n}$ such that there exists $p \in C$ with $f_1(p)=f_2(p)$ and $d f_2|_p = \lambda d f_1|_p$. Thus clearly in the case $n=0$ we have
\[|\text{Per}_1(\lambda)| \cap \text{Rat}_d = \text{Per}_1^\sim(\lambda)\]
as sets.

We can generalize this construction for $m$-periodic points. Consider the pullback $\mathcal{C}$ of the forgetful map $Y_{d^m,n+1} \to Y_{d^m,n}$ via the rational map $\mathfrak{sc}_m$.
\begin{center}
\begin{tikzcd}
\mathcal{C} \arrow{r} \arrow{d}{\pi'} & Y_{d^m,n+1} \arrow{d}{\pi} \arrow{r}{\text{ev}_{n+1}} &\PP^1 \times \PP^1\\
Y_{d,n} \arrow{r}[dashed]{\mathfrak{sc}_m} & Y_{d,n} &
\end{tikzcd}  
\end{center}
As above, denote by $\mu, \varphi$ the components of the map $\mathcal{C} \to \PP^1 \times \PP^1$, then on $\text{Per}_m = (\mu, \varphi)^{-1} (\Delta_{\PP^1})$ we can consider the vanishing set $V=V((d \varphi - \lambda d \mu)|_{\text{Per}_m})$ and push it forward via $\pi'$ to obtain
\[\text{Per}_m(\lambda) = \pi'_* [V] \in Z_{\codim=1}(Y_{d,n}).\]
Note that here we use that $\mathfrak{sc}_m$ is defined away from a codimension $2$ locus in $Y_{d,n}$. Again for $n=0$ we have
\[|\text{Per}_m(\lambda)| \cap \text{Rat}_d = \text{Per}_m^\sim(\lambda).\]
In the following Proposition, we check various compatibility conditions between these cycles and compute their class in the rational Picard group of $Y_{d,n}$. Note that as the spaces $Y_{d,n}$ have finite quotient singularities, all Weyl-divisors $D$ are $\mathbb{Q}$-Cartier and thus it makes sense to speak of pullback of$D$ under maps $X \to Y_{d,n}$ not having image inside $D$.
\begin{Pro} \label{Pro:Permcompatibility}
 Let $\lambda \in \mathbb{C} \setminus \{0\}$. For the forgetful map $F : Y_{d,n} \to Y_{d,0}$ we have
 \begin{equation} \label{eqn:Per1} F^* \text{Per}_m(\lambda)^{Y_{d,0}} = \text{Per}_m(\lambda)^{Y_{d,n}}.\end{equation}
 For the rational self-composition map $\mathfrak{sc}_m : Y_{d,n} \dashrightarrow Y_{d^m,n}$ we have 
 \begin{equation} \label{eqn:Per2} \mathfrak{sc}_m^* \text{Per}_1(\lambda)^{Y_{d^m,n}} = \text{Per}_m(\lambda)^{Y_{d,n}}.\end{equation}
 In the case $n=0$ we have
 \begin{equation} \label{eqn:Per3} \text{Per}_m(\lambda) |_{\text{Rat}_d} = [\text{Per}_m^\sim(\lambda)] \in Z_{\codim=1}(\text{Rat}_d).\end{equation}
 For the rational divisor class of $\text{Per}_m(\lambda)$ we have
 \[\text{Per}_m(\lambda) = m d^{m-1} \sum_{k,B} k D_{B,k} \in A_{\codim=1}(Y_{d,n}) \otimes_\mathbb{Z} \mathbb{Q} = \text{Pic}(Y_{d,n}) \otimes_\mathbb{Z} \mathbb{Q},\]
 where the sum runs over all boundary divisors $D_{B,k}$ of $Y_{d,n}$.
\end{Pro}
\begin{proof}
 In each of the equations (\ref{eqn:Per1} - \ref{eqn:Per3}), the divisors on the left and right side are supported on the same closed sets. Therefore, it suffices to check that for each irreducible component of these sets, the multiplicities of the divisors agree. As also none of the divisors have components supported in the boundary of $Y_{d,n}$, we can compute these multiplicities on the complement $Y_{d,n}^o$ of the boundary. Note that once we show equations (\ref{eqn:Per1}) and (\ref{eqn:Per3}), the equation (\ref{eqn:Per2}) follows as by definition $\text{Per}_m^\sim(\lambda) = \mathfrak{sc}_m^{-1}(\text{Per}_1(\lambda))$. 
 
 Let $P=((\PP^1)^n \setminus \Delta)$, then $Y_{d,n}^o \cong \text{Rat}_d \times P$. Then restricted to $Y_{d^m,n}^o$, the diagram for defining $\text{Per}_m(\lambda)$ looks like
 \[\begin{CD}
   \text{Rat}_d \times \PP^1 \times P @>>> \text{Rat}_{d^m} \times \PP^1 \times P @>\text{ev} >> \PP^1 \times \PP^1\\
   @V \pi' VV @V\pi VV  @.\\
   \text{Rat}_d \times P @> \mathfrak{sc}_m >> \text{Rat}_{d^m} \times P @.
  \end{CD}\]
 We then take the preimage $\text{Per}_m$ of $\Delta_{\PP^1}$ under the upper map 
 \[(\mu, \varphi^{\circ m}) : (\varphi,p,p_1, \ldots, p_n) \mapsto (p, \varphi^{\circ m}(p)),\] 
 take the vanishing set of the restricted section $(d \varphi^{\circ m} - \lambda d \mu)|_{\text{Per}_m}$ and push forward via $\pi'$. First we note that the factor $P$ does not play any role in these operations, so we may as well leave it out. This shows equation (\ref{eqn:Per1}).
 
 Note that $\text{Per}_m$ is a relative Cartier divisor for the projection map $p: \text{Rat}_d \times \PP^1 \to \text{Rat}_d$. Hence the map $\pi'$ is flat and clearly also finite. Instead of considering $s=d \varphi^{\circ m} - \lambda d \mu|_{\text{Per}_m}$ as a section of $\text{Hom}(\mu^* \Omega_{\PP^1}, \omega_p)$ we precompose with the inverse of the isomorphism $d \mu : \mu^* \Omega_{\PP^1} \to \omega_p$ and obtain a section 
 \[s \circ (d \mu)^{-1} = d \varphi^{\circ m} \circ(d \mu)^{-1} - \lambda : \omega_p \to \omega_p.\] Then by Lemma \ref{Lem:detpushforward}, we have $\pi'_*([V(d \varphi^{\circ m} \circ(d \mu)^{-1} - \lambda)]) = [V(\text{det}(d \varphi^{\circ m} \circ(d \mu)^{-1} - \lambda))]$. But 
 \[\text{det}(d \varphi^{\circ m} \circ(d \mu)^{-1} - \lambda) = \Psi_{d^m}(\varphi^{\circ m}, \lambda),\]
 so the vanishing scheme of this section is exactly $\text{Per}_m^\sim(\lambda)$ (see also \cite[Theorem 4.5]{silverman}). This shows (\ref{eqn:Per3}).
 
 To compute the class of $\text{Per}_m(\lambda)$ , by (\ref{eqn:Per1}) it suffices to show the formula above for $n=0$, as it pulls back correctly under forgetful maps (all boundary divisors $D_{B,k}$ with nonzero coefficient have $k \geq 1$). Assume we have showed the formula for $m=1$. Then we can use (\ref{eqn:Per2}) to derive it for general $m$. Indeed, using Proposition \ref{Pro:compopullback}, we find by induction
 \begin{align*}
  \text{Per}_m(\lambda) &= \mathfrak{sc}_m^* \text{Per}_1(\lambda) = \mathfrak{sc}_m^* \sum_{k=1}^{d^m} k D_{\emptyset,k}\\
  &=\sum_{k=1}^d k d^{m-1} D_{\emptyset,k} + \mathfrak{sc}_{m-1}^* \sum_{l=1}^{d^{m-1}} dl  D_{\emptyset,l}\\
  &=\sum_{k=1}^d k d^{m-1} D_{\emptyset,k} + d \cdot \mathfrak{sc}_{m-1}^* \text{Per}_1(\lambda)\\
  &=\left(\sum_{k=1}^d k  D_{\emptyset,k}\right) \cdot(d^{m-1}  + d (m-1) d^{m-2}) \\
  &=\left(\sum_{k=1}^d k  D_{\emptyset,k}\right) \cdot m d^{m-1}.
 \end{align*}
 It remains to show the case $m=1, n=0$. Here recall that for the diagram
 \[\begin{CD}
   Y_{d,1} @>\text{ev}=(\mu, \varphi) >> \PP^1 \times \PP^1\\
   @V\pi VV  @.\\
   Y_{d,0} @.
  \end{CD}
\]
 we defined 
 \[\text{Per}_1(\lambda) = \pi_* [V(d \varphi - \lambda d \mu|_{\text{Fix}})].\]
 Here $d \mu \in \text{Hom}(\mu^* \Omega_{\PP^1},\omega_\pi) = \omega_\pi \otimes \mu^* \Omega_{\PP^1}^\vee$, so clearly the class of $\text{Per}_1(\lambda)$ is given by
 \[\text{Per}_1(\lambda) = \pi_* \left( c_1(\omega_\pi \otimes \mu^* \Omega_{\PP^1}^\vee) \cap  [\text{Fix}] \right).\]
 Now it is well-known that $c_1(\omega_\pi) = \psi_1$ and in Proposition \ref{Pro:psiformula}, we are going to show that 
 \[\psi_1 =  -2 \mathcal{H}_{1,1} + \sum_{k=1}^d D_{\{1\},k}.\]
 But by definition $c_1(\mu^* \Omega_{\PP^1}^\vee) = c_1(\mu^* \mathcal{O}_{\PP^1}(2)) = 2 \mathcal{H}_{1,1}$. Thus
 \[\text{Per}_1(\lambda) = \pi_* \left(\sum_{k=1}^d c_1(D_{\{1\},k} ) \cap  [\text{Fix}] \right).\]
 Note that in the interpretation as a universal curve, the divisor $D_{\{1\},k}$ is exactly the union of vertical sections over the divisor $D_{\emptyset,k} \subset Y_{d,0}$. For $k=1, \ldots, d$, a general point $f \in D_{\emptyset,k}$ has exactly $k$ fixed points on its vertical component. Hence, as we have claimed, the coefficient of $D_{\emptyset,k}$ in $\pi_* \left( c_1(D_{\{1\},k} ) \cap  [\text{Fix}] \right)$ is $k$, which shows the desired formula.
\end{proof}
\todoOld{Do we need to take the reduced subscheme structure on $\text{Per}_m$? $\text{Per}_m(1)$ may have issues not being reduced, formulas still ok? }
Finally, as the sets $\text{Per}_m(\lambda)$ are invariant under $\text{PGL}_2$, the Weyl divisors defined above induce Weyl-divisors on the quotient spaces $M(d|d_1, \ldots, d_n)$ and their classes in the rational Picard group are given by the formulas above (where unstable divisors $D_{B,k}$ are omitted). 

\section{Intersection theory on \texorpdfstring{$M(d|d_1, \ldots, d_n)$}{M(d|d1, ..., dn)}} \label{Sect:IntTh}
In this section, we treat the intersection theory of $M(d|d_1, \ldots, d_n)$. We first describe $\psi$-classes and show that their intersection numbers satisfy the usual recursions. Note that this discussion is largely independent of the following results. Then we start dealing with general top-intersections of divisor classes. Here we first show how to compute an intersection number involving a boundary divisor to an equivariant intersection on a product of lower-dimensional moduli spaces. As an excursion, we treat the equivariant intersection theory of such products in a separate subsection. Finally, we show how to use these results to obtain a recursive algorithm, computing all intersection numbers of divisors.
\subsection{Definitions and properties of \texorpdfstring{$\psi$}{psi}-classes} \label{Sect:psi}
Let $\textbf{d}=(d|d_1, \ldots, d_n)$ be admissible. We  construct $\psi$-classes on the spaces $\mathcal{M}(\textbf{d}) = M(d|d_1, \ldots, d_n)$ and prove analogues of the String, Dilaton and Divisor equation. In order to have a universal family, we first work with the smooth DM-stacks $\mathcal{M}(d|d_1, \ldots, d_n)$ and then descend to $M(d|d_1, \ldots, d_n)$. Indeed, for a smooth DM-stack $\mathcal{X}$ with coarse moduli space $X$, using \cite[Proposition 6.1]{vistoliintersection} we have natural isomorphisms
\[\text{Pic}(\mathcal{X}) \otimes_{\mathbb{Z}} \mathbb{Q} = A_{\codim = 1}(\mathcal{X})_\mathbb{Q}  = A_{\codim = 1}(X)_\mathbb{Q}=  \text{Pic}(X) \otimes_{\mathbb{Z}} \mathbb{Q}.\]
We apply these identifications for $\mathcal{X}=\mathcal{M}(d|d_1, \ldots, d_n)$ and $\mathcal{X}=\overline{\mathcal{M}}_{0,n}(\PP^1 \times \PP^1, (1,d))$. 
\begin{Def}
 Let $\textbf{d}=(d|d_1, \ldots, d_n)$ be admissible and denote by $(\textbf{d},0)=(d|d_1, \ldots, d_n,0)$, which is also admissible. Then consider the forgetful map $\pi: \mathcal{M}(\textbf{d},0) \to \mathcal{M}(\textbf{d})$ of the additional marking, with sections $\sigma_i: \mathcal{M}(\textbf{d}) \to \mathcal{M}(\textbf{d},0)$ corresponding to the marked points. We define
 \[\mathbb{L}_i = \sigma_i^*(\omega_\pi), \psi_i = c_1(\mathbb{L}_i),\]
 where $\omega_\pi$ is the relative dualizing sheaf of $\pi$.
 
 This induces a well-defined classes $\mathbb{L}_i \in \text{Pic}(M(d|d_1, \ldots, d_n)) \otimes_{\mathbb{Z}}\mathbb{Q}$ and  $\psi_i \in A^1(M(d|d_1, \ldots, d_n))_{\mathbb{Q}}$.
\end{Def}
In the entire section, the additional mark of $\mathcal{M}(\textbf{d},0)$ will be $0$ instead of $n+1$ to make the formulas more pleasant. In the following we want to relate these classes to the corresponding (and well-studied) $\psi$-classes on the prequotient $\overline{\mathcal{M}}_{0,n}(\PP^1 \times \PP^1, (1,d))$. Our main tool is the following result.
\begin{Pro} \label{Pro:uniquotdiag}
 The following diagram is cartesian (in the obvious sense that $\tilde \pi, \tilde \sigma_i$ are the pullbacks of $\pi, \sigma$).
 \begin{equation} \label{eqn:comdiagunifam}
  \begin{tikzpicture}
   \draw (0,0) node {$\overline{\mathcal{M}}_{0,n}(\PP^1 \times \PP^1, (1,d))^{ss,\textbf{d}}$};
   \draw (0,1.5) node {$\overline{\mathcal{M}}_{0,n+1}(\PP^1 \times \PP^1, (1,d))^{ss,(\textbf{d},0)}$};
   \draw (5,0) node {$\mathcal{M}(\textbf{d})$};
   \draw (5,1.5) node {$\mathcal{M}(\textbf{d},0)$};
   \draw[->] (0,1.2) -- (0,0.75) node[left] {$\tilde \pi$} -- (0,0.3);
   \draw[->] (4.9,1.2) -- (4.9,0.75) node[left] {$\pi$} -- (4.9,0.3);
   \draw[<-] (0.3,1.2) -- (0.3,0.75) node[right] {$\tilde \sigma_i$} -- (0.3,0.3);
   \draw[<-] (5.2,1.2) -- (5.2,0.75) node[right] {$\sigma_i$} -- (5.2,0.3);
   \draw[->] (2.4,0) -- (3.5,0) node[above] {$\phi$} -- (4.3,0);
   \draw[->] (2.4,1.5) -- (3.5,1.5) node[above] {$\tilde \phi$} -- (4.3,1.5);
  \end{tikzpicture}
 \end{equation} 
\end{Pro}
\begin{proof}
 The diagram is cartesian by Lemma \ref{Lem:quotfibreprod} as the maps $\phi, \tilde \phi$ are quotients under the actions of $G=\text{PGL}_2$ and the map $\pi$ is induced by the $G$-equivariant map $\tilde \pi$. 
\end{proof}
Let
\[\mathbb{\tilde L}_i = \tilde \sigma_i^*(\omega_{\tilde \pi}), \tilde \psi_i = c_1(\mathbb{\tilde L}_i)\]
be the $i$-th cotangent line bundle and $\psi$-class on $\mathcal{Y}_{d,n}$.
\begin{Cor}
 In the situation of Proposition \ref{Pro:uniquotdiag} we have
 \begin{align*}\mathbb{\tilde  L}_i &= \phi^* \mathbb{L}_i\text{ in }\text{Pic}(\mathcal{Y}_{d,n}^{ss,\textbf{d}}) \otimes_{\mathbb{Z}} \mathbb{Q},\\ \tilde \psi_i &= \phi^* \psi_i \text{ in }A^1(\mathcal{Y}_{d,n}^{ss,\textbf{d}})_{\mathbb{Q}}.\end{align*}
\end{Cor}
\begin{proof}
 As this diagram is cartesian, we have that $(\tilde \phi)^* \omega_{\pi} = \omega_{\tilde \pi}$ by \cite[Theorem 2.11]{nironidualizing}. Pulling back along $\tilde \sigma_i$ gives the desired result.
\end{proof}
As mentioned before, cotangent line bundles and $\psi$-classes have been extensively studied (for a good survey see \cite{kockpsi}).
As the pullback of line bundles (with rational coefficients) under the map $Y_{d,n}^{ss,\textbf{d}} \to M(d|d_1, \ldots, d_n)$ induced by $\phi$ gives an isomorphism of rational Picard groups by Corollary \ref{Cor:Picidentification}, we see that \[\phi^*: A^1(\mathcal{M}(\textbf{d}))_{\mathbb{Q}} \to A^1(\mathcal{Y}_{d,n}^{ss,\textbf{d}})_{\mathbb{Q}}\] is an isomorphism. As a first step we explicitly identify the line bundles $\mathbb{\tilde L}_i$ in terms of our generators from Theorem \ref{Theo:Generators}.
\begin{Pro} \label{Pro:psiformula}
 On $\overline M_{0,n}(\PP^1 \times \PP^1, (1,d))$ we have
 \begin{equation*}
 \mathbb{\tilde L}_i  = \sum_{\substack{(B,k)\\i \in B, a,b \notin B\text{ or}\\i \notin B, a,b \in B}}D_{B,k}
 \end{equation*}
 for $n\geq 3, a,b \in \{1, \ldots, n\} \setminus \{i\}, a \neq b$. On the other hand, for $n=1,2$ we have
 \begin{equation*}
  \mathbb{\tilde L}_i = -2 \mathcal{H}_{i,1} + \sum_{\substack{(B,k)\\i \in B}}D_{B,k}.
 \end{equation*}
 In all cases the sum only includes pairs $B \subset \{1, \ldots, n\}$, $0 \leq k \leq d$, such that $|B| \geq 2$ for $k=0$.
\end{Pro}
\begin{proof}
 As explained in \cite[Proposition 5.1.8]{kockpsi}, for  $n \geq 3$ the class $\psi_i$ is the sum $(i|a,b)$ of boundary classes where marks $i$ and $a,b$ are on different components (for two other markings $a,b$ such that $i,a,b$ are pairwise distinct). This readily translates to the expression above.
 
 In the cases $n=1,2$ we need a different argument. Note here that for the map $\tilde \pi$ as in Proposition \ref{Pro:uniquotdiag} we have 
 \[\tilde \psi_i = \tilde \pi^* (\psi_i) + D_{\{i,0\},0}\]
 by \cite[(5.1.7.2)]{kockpsi}. Now the map $\tilde \pi^*$ is injective on Picard groups as $\tilde \sigma_i^* \circ \tilde \pi^* = (\tilde \pi \circ \tilde \sigma_i)^* = \text{id}$. Hence all we need to show is that the divisor given by the formula above for $n=2,1$ pulls back to the correct class $\psi_i'=\tilde \psi_i - D_{\{i,0\},0}$ on the space with $3$ and $2$ marks, respectively.
 
 For $n=2$ we have
 \[\tilde \pi^* \left(-2 \mathcal{H}_{i,1} + \sum_{\substack{k, B \subset \{1,2\}\\i \in B}}D_{B,k}\right) = -2 \mathcal{H}_{i,1} + \sum_{\substack{k, B \subset \{0,1,2\}\\i \in B}}D_{B,k} - D_{\{i,0\},0}=:\psi_i''.\]
 By adding $D_{\{i,0\},0}$ to both sides of the equation $\psi_i''=\psi_i'$ we are left to show
 \[\psi_i = -2 \mathcal{H}_{i,1} + \sum_{\substack{k, B \subset \{0,1,2\}\\i \in B}}D_{B,k} \text{ on }Y_{d,3}.\]
 For this we intersect both sides with our test curves $C_{B',k'}$ (and $C_\mathcal{G}$ in case $d=0$) and show they give the same numbers (which suffices by Proposition \ref{Pro:iddivisors}). For calculating the intersection with the left side we use the explicit formula from the case $n=3$ that was already shown. The intersection with $C_\mathcal{G}$ is always zero.  It turns out that both intersections only depend on $B'$. Let the symbol $(r|s)$ be $1$ for $r \in B', s \notin B'$ or $r \notin B', s \in B'$ and $0$ otherwise. Similarly let $(r|s,t)$ be $1$ for $B'=\{r\}$ or $B'=\{s,t\}$ and zero otherwise. Then the desired equation of intersection numbers)(for $i=1$ for simplicity) is exactly
 \[(1|2) + (1|0) = (2|0) + 2(1|20).\]
 Substracting $(2|0)$ this is a combinatorial inclusion-exclusion-type identity (or it can simply be explicitly checked on all $8$ possibilities for $B' \subset \{0,1,2\}$).\\
 For $n=1$ the equality
  \[\psi_1=-2 \mathcal{H}_{1,1} + \sum_{\substack{(B,k)\\1 \in B}}D_{B,k} = \tilde \pi^* \left(-2 \mathcal{H}_{1,1} + \sum_{\substack{(B,k)\\1 \in B}}D_{B,k}\right) + D_{\{1,0\},0}\]
 is obvious.
\end{proof}
Now we define descendant Gromov-Witten invariants on the space $M(\textbf{d})$. Because the markings can possibly carry different weights $d_i$, they are not necessarily interchangeable and it will be necessary to record them in the notation. For $i\in \{1, \ldots, n\}$, $k \geq 0$ and 
\[\gamma = r \cdot 1 + c_1(\mathcal{O}(s,t)) + u \cdot [pt] \in A^*(\PP^1 \times \PP^1)\]
let 
\begin{align*}
e_i(\gamma) &= r + s c_1(\mathcal{H}_{i,1}) + t c_1(\mathcal{H}_{i,2}) + u c_1(\mathcal{H}_{i,1}) c_1(\mathcal{H}_{i,2})\in A^*(M(\textbf{d})),\\
\taumo{i}{k}(\gamma) &= \psi_i^ke_i(\gamma) \in A^*(M(\textbf{d})). 
\end{align*}
This definition is of course constructed in such a way that the pullback of $e_i(\gamma)$ under the quotient map $\phi: Y_{d,n}^{ss,\textbf{d}} \to M(\textbf{d})$ is exactly $\text{ev}_i^*(\gamma)$ and the pullback of $\taumo{i}{k}(\gamma)$ is $\tilde \psi_i^k \text{ev}_i^*(\gamma)$. Then for $k_1, \ldots, k_n \geq 0$ and homogeneous classes $\gamma_1, \ldots, \gamma_n \in A^*(\PP^1 \times \PP^1)$ with
\begin{equation} \label{eqn:degreecond}
\sum_{i=1}^n k_i + \text{deg}(\gamma_i) = 2(d-1)+n = \text{dim}(M(\textbf{d})) 
\end{equation}
we define
\[\langle \taumo{1}{k_1}(\gamma_1) \cdots \taumo{n}{k_n}(\gamma_n)\rangle_{\textbf{d}} = \int_{[M(\textbf{d})]} \taumo{1}{k_1}(\gamma_1) \cdots \taumo{n}{k_n}(\gamma_n).\]
We will now verify that the familiar String, Dilaton and Divisor equation hold, establishing recursions that allow to compute some of the invariants above. For this, the following Lemma summarises some key properties of $\psi$-classes and related divisors that carry over from $\overline M_{0,n}(\PP^1 \times \PP^1, (1,d))$ to $M(\textbf{d})$ via Proposition \ref{Pro:uniquotdiag}.
\begin{Lem} \label{Lem:auxipushpullpsi}
 In $\text{Pic}(\mathcal{M}(\textbf{d},0)) \otimes_{\mathbb{Z}} \mathbb{Q}$ we have
 \[\mathbb{L}_i = \pi^*(\mathbb{L}_i) + D_{\{i,0\},0}.\]
 Using that $\psi_i D_{\{i,0\},0}=0$ one shows inductively
 \[\psi_i^k = \pi^*(\psi_i)^k + \pi^*(\psi_i)^{k-1} D_{\{i,0\},0}.\]
 As the divisors $D_{\{i,0\},0}$ are exactly the image of the (disjoint) sections $\sigma_i$, we have
 \begin{align*}
  D_{\{i,0\},0} D_{\{j,0\},0} = 0 \text{, for }i \neq j.
 \end{align*}
 The following formulas hold for pushforwards of codimension $1$ cycles on $\mathcal{M}(\textbf{d},0)$ under $\pi$:
 \begin{align*}
  \pi_*(D_{\{i,0\},0}) &= [\mathcal{M}(\textbf{d})],\\
  \pi_*(\psi_i) &= (n-2)[\mathcal{M}(\textbf{d})],\\
  \pi_*(\mathcal{H}_{i,j}) &= \begin{cases}
                              [\mathcal{M}(\textbf{d})] &,\text{ for } i=0, j=1\\
                              d [\mathcal{M}(\textbf{d})] &,\text{ for } i=0, j=2\\
                              0 &,\text{ otherwise}.
                             \end{cases}
 \end{align*}
 Finally we have
 \[\mathcal{H}_{i,j} D_{\{i,0\},0} =  \mathcal{H}_{0,j} D_{\{i,0\},0}\]
\end{Lem}
\begin{proof}
 The main strategy of our proof will be to transfer the claimed results, which live on the right hand side of the commutative diagram in Proposition \ref{Pro:uniquotdiag}, to the left side. There, all of them are either proved in \cite{kockpsi} or are immediate from the definitions.
 
 The first formula can be seen using this strategy, as the pullback on rational Picard groups under $\phi$ is injective.
 To prove $\psi_i D_{\{i,0\},0}=0$ and also $\mathcal{H}_{i,j} D_{\{i,0\},0} =  \mathcal{H}_{0,j} D_{\{i,0\},0}$, one uses that $D_{\{i,0\},0}$ is the isomorphic image of $\mathcal{M}(\textbf{d})$ under $\sigma_i$. Thus we only need to pull back $\mathbb{L}_i$ and $\mathcal{H}_{0,j} - \mathcal{H}_{i,j}$ via $\sigma_i$ and check that they are trivial bundles on $\mathcal{M}(\textbf{d})$. But again this can be checked after pulling back via $\phi$. On the left side of the diagram in Proposition \ref{Pro:uniquotdiag}, these are standard identities. The formula for $\psi_i^k$ follows by induction.
 
 It remains to show the formulas for the pushforwards of codimension $1$ cycles. Here, we can use the formula $\phi^*\pi_*=\tilde \pi_* \tilde \phi^*$ (see \cite[Lemma 3.9]{vistoliintersection}) to transfer these identities to the morphism $\mathcal{Y}_{d,n+1} \to \mathcal{Y}_{d,n}$. There, all of the formulas follow very easily from known results. 
\end{proof}

In the following, the nonnegative integers $k_1, \ldots, k_n$ and homogeneous classes $\gamma_1, \ldots, \gamma_n \in A^*(\PP^1 \times \PP^1)$ are chosen such that (\ref{eqn:degreecond}) is satisfied.
\begin{Pro}[String equation] \label{Pro:string}
 We have
 \begin{equation} \label{eqn:string}
  \langle \taumo{0}{0}(1) \prod_{i=1}^n \taumo{i}{k_i}(\gamma_i) \rangle_{(\textbf{d},0)} = \sum_{j=1}^n \langle \taumo{j}{k_j-1}(\gamma_j) \prod_{i\neq j} \taumo{i}{k_i}(\gamma_i) \rangle_{\textbf{d}}.
 \end{equation}
\end{Pro}
\begin{Pro}[Dilaton equation] \label{Pro:dilaton}
 We have
 \begin{equation} \label{eqn:dilaton}
  \langle \taumo{0}{1}(1) \prod_{i=1}^n \taumo{i}{k_i}(\gamma_i) \rangle_{(\textbf{d},0)} = (n-2) \langle \prod_{i=1}^n \taumo{i}{k_i}(\gamma_i) \rangle_{\textbf{d}}.
 \end{equation}
\end{Pro}
\begin{Pro}[Divisor equation] \label{Pro:divisor}
 For a divisor $D \in A^1(\PP^1 \times \PP^1)$ we have
 \begin{align*} \label{eqn:divisor}
  &\langle \taumo{0}{0}(D) \prod_{i=1}^n \taumo{i}{k_i}(\gamma_i) \rangle_{(\textbf{d},0)} \\= &\langle D,(1,d)\rangle \langle \prod_{i=1}^n \taumo{i}{k_i}(\gamma_i) \rangle_{\textbf{d}} +  \sum_{j=1}^n \langle \taumo{j}{k_j-1}(\gamma_j \cup D) \prod_{i\neq j} \taumo{i}{k_i}(\gamma_i) \rangle_{\textbf{d}}.
 \end{align*}
\end{Pro}
The proofs of all three identities are exactly as for the corresponding descendant invariants on $\overline M_{g,n}(X,\beta)$. The analogous identities necessary in these proofs are exactly the statements in Lemma \ref{Lem:auxipushpullpsi}.
\subsection{Intersection on boundary divisors}
The idea of our algorithm for the top-intersections of divisors $D_1, \ldots, D_{2d-2+n}$ on $M(d|d_1, \ldots, d_n)$, following a similar algorithm presented in \cite{pandhaintersect}, will be to first reduce to the case $D_1=D_{B,k}$, a boundary divisor.  Then we restrict all divisors $D_2, \ldots, D_{2d-2+n}$ to $D_1$ and perform the intersection there. In the following, we look at this last step.

Let us outline again the strategy used in \cite{pandhaintersect}, but already adapted to our setting. There, the product structure of the boundary divisor $D_1$ was used to reduce the intersection of the remaining divisors to an intersection on a moduli space of lower degree or with fewer marked points. Indeed, consider a boundary divisor $D= D_{B,k}$ on $Y_{d,n}=\overline M_{0,n}(\PP^1 \times \PP^1, (1,d))$. Let $A=\{1, \ldots, n\} \setminus B$.
Then we have a gluing map 
\begin{align*}
 \tau :  \overline M_{0,A \cup \{p\}}(\PP^1 \times \PP^1, (1,d-k)) \times_{\PP^1 \times \PP^1} \overline M_{0,B \cup \{p'\}}(\PP^1 \times \PP^1, (0,k)) \to D_{B,k}.
\end{align*}
Informally a pair $(f_A: C_A \to \PP^1 \times \PP^1, f_B: C_B \to \PP^1 \times \PP^1)$ of stable maps is sent to the induced map $f_A \coprod f_B : C_A \coprod_p C_B \to \PP^1 \times \PP^1$, where the curves are glued along the points $p,p'$, which by definition map to the same point $f_A(p)=f_B(p')$. For a formal definition of $\tau$ see \cite[Section 6.2]{fultonpandha}. 
As we have 
\[M_{0,B \cup \{p'\}}(\PP^1 \times \PP^1, (0,k)) = \PP^1 \times M_{0,B \cup \{p'\}}(\PP^1, k),\]
we can perform a partial fibre product above. Let 
\begin{align*}
 \overline M_A &= \overline M_{0,A \cup \{p\}}(\PP^1 \times \PP^1, (1,d-k)),\\
 \overline M_B &= \overline M_{0,B \cup \{p'\}}(\PP^1, k),
\end{align*}
then the map $\tau$ above has the form
\begin{align*}
 \tau :  \overline M_A \times_{\PP^1} \overline M_B \to D_{B,k}.
\end{align*}
By \cite[Lemma 12]{fultonpandha}, the map $\tau$ is birational, as the degrees $(1,d-k)$ and $(0,k)$ are always distinct. Thus any top-dimensional intersection number on $D$ can be computed on the product $\overline M_A \times_{\PP^1 } \overline M_B$. We can further simplify this problem by observing that the diagram 
\[ \begin{CD}
 \overline M_A \times_{\PP^1 } \overline M_B     @>>>  \overline M_A \times \overline M_B\\
@VVV        @VV\text{ev}=(\pi_2 \circ \text{ev}_p) \times \text{ev}_p'V\\
\Delta_{\PP^1 }    @>>>  \PP^1 \times \PP^1
\end{CD}\]
is cartesian. As we will see immediately, all the restrictions of divisor classes from $Y_{d,n}$ via $\tau$ come from classes on $M_A \times M_B$. Hence we can perform the intersection there if we add to the intersection the class $\text{ev}^*([\Delta_{\PP^1} ])$. But (the Poincare dual of) the class $[\Delta_{\PP^1}]$ is given by $c_1(\mathcal{O}(1,1))$, hence after the pullback via $\text{ev}$ we have
\[\text{ev}^* \mathcal{O}(1,1) = \mathcal{H}_{p,2} \boxtimes \text{ev}_{p'}^* \mathcal{O}_{\PP^1}(1).\]
Here, for $\mathcal{L}$ a line bundle on $\overline M_A$ and $\mathcal{M}$ a line bundle on $\overline M_B$, we denote by $\mathcal{L} \boxtimes \mathcal{M}$ the bundle $\pi_{\overline M_A}^*(\mathcal{L}) \otimes \pi_{\overline M_B}^* (\mathcal{M})$ on $\overline M_A \times \overline M_B$. Below, we denote in the same way also the restriction of $\mathcal{L} \boxtimes \mathcal{M}$ to $\overline M_A \times_{\PP^1 } \overline M_B$.
Let us compute the pullbacks of divisors on $Y_{d,n}$ via $\tau$. 
\begin{Pro} \label{Pro:DivRestr}
 Let $D$ be a rational divisor class on $Y_{d,n}$. Then
 \begin{itemize}
  \item for $D= \mathcal{H}_{i,j}$ with $i \in A$ we have $\tau^* D = \mathcal{H}_{i,j} \boxtimes \mathcal{O}$,
  \item for $D= \mathcal{H}_{i,1}$ with $i \in B$ we have $\tau^* D = \mathcal{H}_{p,1} \boxtimes \mathcal{O}$,
  \item for $D= \mathcal{H}_{i,2}$ with $i \in B$ we have $\tau^* D = \mathcal{O} \boxtimes \text{ev}_{i}^* \mathcal{O}_{\PP^1}(1)$,
  \item for $D=\mathcal{G}$ in the case $d=0$ we have $\tau^* D = \mathcal{G} \boxtimes \mathcal{O}$,  
  \item for $D=D_{B',k'}$ with $B' \neq B$ or $k' \neq k$ we have that $D,D_{B,k}$ intersect transversally in the loci
  \begin{itemize}
   \item $D_{B',k'} \times_{\PP^1} \overline M_{B}$ for $B' \subset A$, $d-k \geq k'$,
   \item $D_{(B' \setminus B) \cup \{p\},k'-k}\times_{\PP^1} \overline M_{B}$ for $B \subset B'$ and $k \leq k'$,
   \item $\overline M_A \times_{\PP^1} D((B \setminus B') \cup \{p'\}, k-k'; B', k')$ for $B' \subset B$ and $k' \leq k$ and ($k'<k$ or $B \setminus B' \neq \emptyset$),
  \end{itemize}
  and thus $\tau^* D$ is given by the sum of the corresponding divisors (with multiplicity $1$),
  \item for $D=D_{B,k}$ with $B \neq \emptyset$ or $2k > d$ we have $\tau^*D = \psi_{p}^\vee \boxtimes \psi_{p'}^\vee$. If $B = \emptyset$ and $2k \leq d$ we have an additional summand $D_{\emptyset,k} \boxtimes \mathcal{O}$. Here $\psi_{q}$ denotes the cotangent line bundle corresponding to the marking $q$.
 \end{itemize}
\end{Pro}
\begin{proof}
 The formulas for the pullbacks of divisors $\mathcal{H}_{i,j}$ are immediate from the fact that, depending on whether $i \in A$ or $i \in B$ , the composition of the $i$th evaluation map on $Y_{d,n}$ with $\tau$ only depends on one of the factors $\overline M_A$, $\overline M_B$. Moreover, the horizontal position of the points in $B$ agrees with the horizontal position of the gluing point $p$ coming from $\overline M_A$. The pullback of $\mathcal{G}$ in case $d=0$ is also clear.
 
 It remains to consider the boundary divisors. Two different boundary divisors intersect transversally by \cite[Theorem 3]{fultonpandha} and above we have just distinguished all cases that are combinatorially possible. They correspond to decorated dual graphs of curves with three vertices that are specializations of the dual graphs coming from $D_{B,k}$ and $D_{B',k'}$.
 
 We are left to consider the self-intersection of the boundary divisor $D_{B,k}$. It is known that at a generic point $(f: C \coprod_{p,p'} C' \to \PP^1 \times \PP^1) \in D_{B,k}$, the normal space of $D_{B,k}$ in $Y_{d,n}$ is exactly given by $T_p C \otimes T_{p'} C'$, the fibre of the bundle $\psi_{p}^\vee \boxtimes \psi_{p'}^\vee$. If $B \neq \emptyset$ or $2k>d$, the map $\tau$ is an embedding and thus indeed the restriction of $D_{B,k}$ via $\tau$ is the normal bundle $\psi_{p}^\vee \boxtimes \psi_{p'}^\vee$.
 
 However, if $B= \emptyset$ and $2k \leq d$, then $\tau$ is still birational, but it is not injective exactly on the divisor $D_{\emptyset,k}\times_{\PP^1} \overline M_B$, where it is generically $2:1$. Thus we have to add this cycle with multiplicity $1$. \todo{Please proofread} \todoOld{Rigorous argument? ... use that two critical components of $C_W B \times B$ intersect in codimension $2$ set, which can be ignored. Still missing: deformation theoretic argument ...}
\end{proof}
For $\textbf{d}=(d|d_1, \ldots, d_n)$ admissible, we use a similar strategy as above on our quotient spaces $M(\textbf{d})$. Let $A \cup B = \{1, \ldots, n\}$ be a partition of our markings and $0 \leq k \leq d$ such that $D=D_{B,k}$ is a divisor on $M(\textbf{d})$, i.e. $D^{ss,\textbf{d}}= D_{B,k} \cap  Y_{d,n}^{ss,\textbf{d}} \neq \emptyset$. We now want to restrict the map $\tau$ above to the preimage of $D^{ss,\textbf{d}}$. From the criterion for semistability in Lemma \ref{Lem:semistableadvanced}, one sees that the question whether $\tau(p,q)$ is semistable for $p \in \overline M_A, q \in \overline M_B$, depends solely on $p$. Let $A=\{a_1, \ldots, a_m\}$ and $\textbf{d}_A = (d-k|d_{a_1}, \ldots, d_{a_m}, k+ \sum_{b \in B} d_b)$. Then we define 
\[\overline M_A^{ss} = \overline M_{0,A \cup \{p\}}(\PP^1 \times \PP^1, (1,d-k))^{ss,\textbf{d}_A}.\]
We have 
\[\tau^{-1}(Y_{d,n}^{ss,\textbf{d}}) = \overline M_A^{ss} \times_{\PP^1} \overline M_B\]
and $\tau$ restricted to this set will still be birational. 
Now we can apply Lemma \ref{Lem:preimstable} to the morphism $\tau$ and see that for a suitable $G$-linearized ample line bundle on $\overline M_A \times_{\PP^1} \overline M_B$, the (semi)stable points are exactly given by $\tau^{-1}(Y_{d,n}^{ss,\textbf{d}})$. Thus we obtain a geometric quotient $\overline M_A^{ss} \times_{\PP^1} \overline M_B / G$ of $\overline M_A^{ss} \times_{\PP^1} \overline M_B$ and an induced map
\[\tilde \tau : \overline M_A^{ss} \times_{\PP^1 } \overline M_B / G \to D^{ss,\textbf{d}}/G = D_{B,k}.\]
As $\tau$ was an isomorphism over a $G$-invariant subset of $D^{ss, \textbf{d}}$, the map $\tilde \tau$ is still proper and birational. Hence we can reduce top-dimensional intersections on $D_{B,k}$ to intersections on $\overline M_A^{ss} \times_{\PP^1 } \overline M_B / G$. As before we have that
\[\overline M_A^{ss} \times_{\PP^1} \overline M_B / G \subset \overline M_A^{ss} \times \overline M_B / G\]
is a divisor and it is cut out exactly by the Chern class of the line bundle corresponding to \[\text{ev}^*\mathcal{O}(1,1) = \mathcal{H}_{p,2} \boxtimes \text{ev}_{p'}^* \mathcal{O}_{\PP^1}(1).\] 
As the quotient maps induce isomorphisms of the rational Picard groups, Proposition \ref{Pro:DivRestr} tells us how to restrict divisor classes from $M(\textbf{d})$ via the map $\tilde \tau$. Hence we have reduced the original question to the intersection theory of $\overline M_A^{ss} \times \overline M_B / G$. In the following subsection we are going to see how to reduce this to (equivariant) intersections on the factors $\overline M_A^{ss}$ and $\overline M_B$.

\subsection{Equivariant intersection theory of products} \label{Sect:EquivarProduct}
In the following let $X,Y$ be varieties over $\mathbb{C}$ and let $G$ be a reductive group acting on $X,Y$ with finite stabilizers. The example to keep in mind is $X=\overline M_A^{ss}$, $Y=\overline M_B$. Assume that for the induced diagonal action of $G$ on $X \times Y$ we have a geometric quotient $Z=X \times Y / G$. We want to study the intersection theory of $Z$. To compute the full intersection ring is, in general, a very hard problem. Even for $G=\{e\}$ trivial, $Z=X \times Y$, this ring can be quite complicated, as only very special schemes $X,Y$ satisfy a K\"unneth-type formula $A^*(X \times Y) = A^*(X) \otimes A^*(Y)$. However, if we restrict ourselves to intersections of classes pulled back from the two factors $X,Y$, we can compute the intersections on $X,Y$ and then take the exterior product of the resulting cycles. We will try a similar approach in the presence of a nontrivial $G$.

Here, it is natural to use equivariant intersection theory. This was introduced by Edidin and Graham in \cite{EqIntTh}. Let $V$ be a  representation of $G$ of dimension $l$ containing an open set $U$, with complement of sufficiently large codimension, on which $G$ acts freely. Denote by $X_G= (X \times U)/G$ the quotient in the category of algebraic spaces. Then one defines the equivariant Chow-group $A_i^G(X)$ of $i$-cycles to be $A_{i+l-g}(X_G)$. For $X' \subset X$ an invariant closed subscheme, there is a natural fundamental class $[X']_G \in A_i^G(X)$, induced from the class of the subscheme $X' \times U$ in $(X \times U)/G$. These groups share many functorialities (proper pushforward, flat pullback, exterior product, ...) with ordinary Chow groups. This allows to define equivariant operational Chow groups $A_G^i(X)$ as operations $c(W \to X) : A_*^G(W) \to A_{*-i}^G(W)$ for every morphism $W \to X$, satisfying the usual compatibility relations. Below we are going to collect several key properties of these groups from \cite{EqIntTh}, that we will use in the following:
\begin{itemize}
 \item If $X$ is smooth of dimension $n$ then $A^{i}_G(X) \xrightarrow{ \cap [X]_G} A_{n-i}^G(X)$ is an isomorphism. (\cite[Proposition 4]{EqIntTh})
 \item As $X$ is assumed separated and as $\text{char}(\mathbb{C})=0$, we have a canonical isomorphism $A_G^k(X) = A^k(X_G)$, assuming that $V \setminus U$ has codimension more than $k$. (\cite[Corollary 2]{EqIntTh})
 \item For $G$ a connected reductive group with split maximal torus $T$ and Weyl group $W$, we have $A^G_*(X) \otimes \mathbb{Q} = A^T_*(X)^W \otimes \mathbb{Q}$. (\cite[Proposition 6]{EqIntTh})
 \item For $H \subset G$ a subgroup, there is a pull-back map $A^G_*(X) \to A^H_*(X)$ induced by the flat map $X_H = (X \times U)/H \to X_G = (X \times U)/G$. (Section after \cite[Proposition 6]{EqIntTh})
 \item If the action of $G$ on $X$ is (locally) proper and we have a quotient $\pi: X \to Y$, then there is an isomorphism 
 \[\pi^* : A^*(Y)_\mathbb{Q} \to A_G^*(X)_\mathbb{Q}\]
 of operational Chow rings. (\cite[Theorem 3]{EqIntTh})
\end{itemize}
We will need the following variant of the fourth point for operational Chow rings.
\begin{Lem}
 Let $H \subset G$ be a subgroup, then there is a pull-back map $\varphi_{G,H}: A_G^*(X) \to A_H^*(X)$ induced by the flat map $X_H = (X \times U)/H \to X_G = (X \times U)/G$.
\end{Lem}
By the last property above, if the action of $G$ on $X \times Y$ is proper, we have an isomorphism $A^*((X  \times Y) /G)_{\mathbb{Q}} \cong A^*_G(X \times Y)_{\mathbb{Q}}$. As we will heavily rely on this result, we make the following convention:
\begin{center}
 \textbf{From now on, all Chow groups and operational Chow groups are taken with $\mathbb{Q}$-coefficients.}
\end{center}
By pullback and product we can define a map 
\begin{equation}\Phi: A_G^*(X) \otimes A_G^*(Y) \to A^*_G(X \times Y). \label{eqn:EqInt1}\end{equation}
Our goal in this section will be to understand the map above in top-degree, that is in degree $\dim(X)+\dim(Y)-\dim(G)$. Note that in contrast to the usual intersection theory, the equivariant operational Chow groups $A_G^*(pt) = A_G^*(\text{Spec}\ \mathbb{C})$ can be highly nontrivial. Both $A_G^*(X), A_G^*(Y)$ are modules over this ring and by functoriality we see that the map in (\ref{eqn:EqInt1}) factors through their tensor product over this ring:
\begin{equation}\Phi: A_G^*(X) \otimes_{A_G^*(pt)} A_G^*(Y) \to A^*_G(X \times Y). \label{eqn:EqInt2}\end{equation}
Given several elements in $A_G^*(X)$ and $A_G^*(Y)$, the basic idea will be to first compute their product in their respective rings, then use the functoriality above to shift factors coming from $A_G^*(pt)$ between $A_G^*(X)$ and $A_G^*(Y)$ and only then compute their image in $A^*_G(X \times Y)$. Let us first make this last step precise, to see why we may want to shift classes between the factors.
\begin{Pro} \label{Pro:EqProdRecipe}
 Assume that for the action of $G$ on $X$  and on $X \times Y$ there exist $G$-linearized line bundles with $X=X^{ss}=X^{s}$ and $X \times Y = (X \times Y)^{ss}=(X \times Y)^s$. Let $\alpha \in A^*(X \sslash G)=A_G^*(X)$ such that $\alpha \cap [X \sslash G] = [\{[x]\}]$ for a point $x \in X$ (which has finite stabilizer $G_x$ by our global assumptions). Then the preimage of $[x]$ under the projection $\pi_X: (X \times Y) \sslash G \to X \sslash G$ is the image of $Y$ under the map 
 \[i : Y \to (X \times Y) \sslash G, y \mapsto [(x,y)].\]
 We have $i_* [Y] = |G_x| \cdot [\pi_X^{-1}( [x])]$.
 For any class $\beta \in A_G^*(Y)$ we then have
 \[\Phi(\alpha \otimes \beta) \cap [(X \times Y) \sslash G] = \frac{1}{|G_x|} i_* \left( \varphi_{G,\{e\}}(\beta) \cap [Y]\right).\]
 Here $\varphi_{G,\{e\}} : A_G^*(Y) \to A^*(Y)$ is the map coming from the trivial subgroup $\{e\}$ of $G$.
\end{Pro}
\begin{proof}
 One sees easily that $\pi_X^{-1}([x]) = Y  \sslash  G_x$. Hence it follows immediately that $i_* [Y] = |G_x|  \cdot [\pi_X^{-1}( [x])]$. 
 Now in the calculation below, we will want to use flat pullback under the map $\pi_X$. Unfortunately, this map does not need to be flat. However, the map $[(X \times Y) / G] \to [X/G]$ of quotient stacks is flat. As these quotient stacks are also Deligne-Mumford (because $G$ acted with finite stabilizers), we may apply the results about intersection theory from Vistoli's paper \cite{vistoliintersection}. They tell us that on the stack level, we have our accustomed functorialities (proper pushforward, flat pullback, etc.) and that the Chow groups of the stacks are isomorphic to the Chow groups of their moduli spaces (assuming we work with $\mathbb{Q}$-coefficients, see \cite[Proposition 6.1]{vistoliintersection}). And indeed, the quotients $(X \times Y) \sslash G$ and $X \sslash G$ are the coarse moduli spaces of the quotient stacks $[(X \times Y) / G]$, $[X/G]$ by Lemma \ref{Lem:quotcoarsemod} and Remark \ref{Rmk:quotcoarsemod}. Using these isomorphisms, we can indeed apply the formal rules for flat pullbacks under $\pi_X$ in the following. We compute
 \begin{align*}
 &\Phi(\alpha \otimes \beta) \cap [(X \times Y) \sslash G] \\
 = &\Phi(1 \otimes \beta) \cap \pi_X^*(\alpha) \cap \pi_X^* [X \sslash G]\\
  = &\Phi(1 \otimes \beta) \cap \pi_X^*(\alpha \cap [X \sslash G]) = \Phi(1 \otimes \beta) \cap \pi_X^*([\{[x]\}])\\
  = &\Phi(1 \otimes \beta) \cap [\pi_X^{-1}([x])]\\
  = & \frac{1}{|G_x|} \Phi(1 \otimes \beta) \cap i_* [Y].
 \end{align*}
 It remains to see how the class $\Phi(1 \otimes \beta)$ restricts under the map $i$. For this let $V$ be a representation of $G$ such that $G$ acts freely on an open set $U \subset V$ with complement of sufficiently high codimension. Then we have the following diagram of maps
 \[\begin{CD}
Y @>i>> (X \times Y) \sslash G @<\psi << (X \times Y \times U) /G\\
	    @.     @.  @V\eta VV\\
  @.   @. (Y \times U)/G\\
	    @.     @.  @AAA\\
	    @.   @.  Y \times U
\end{CD}\]
Now we want to compute $i^* \Phi(1 \otimes \beta)$. By the properties above, the map $\psi^*$ induces an isomorphism of operational Chow groups in sufficiently small degree. Hence, the class $\Phi(1 \otimes \beta)$ is the unique class such that $\psi^* \Phi(1 \otimes \beta) = \eta^* \beta$, where we regard $\beta$ as an element of $A^*(Y_G) = A^*((Y \times U)/G)$. But $i$ lifts along $\psi$ to a map $i' : Y \to (X \times Y \times U) /G$ by $i'(y) = [(x,y,u)]$ for a fixed $u \in U$. Hence $\Phi(1 \otimes \beta) = (\eta \circ i')^* \beta$. But $(\eta \circ i')(y)=[(y,u)]$, so this map lifts to a map $Y \to Y \times U, y \mapsto (y,u)$. But by definition, the pullback under the map $Y \times U \to (Y \times U)/G$ is exactly $\varphi_{G,\{e\}}$ and the pullback under $y \mapsto (y,u)$ gives the identification $A^*(Y \times U) \cong A^*(Y)$. Together with the projection formula for $i$, this finishes the proof. 
\end{proof}
Now we have to look more carefully at the group $A^*_G(Y)$. For our application, we can assume that $Y$ is the coarse moduli space of a smooth, projective DM-stack. We will include these assumptions when necessary.

We will apply torus localization to compute intersections on $Y$. As we have seen above, we can express $A^G_*(X)$ as the Weyl-invariant part of $A^T_*(X)$. To go from operational Chow groups to Chow groups of cycles, we use the following Lemma, extending the first property of equivariant cohomology mentioned above.
\begin{Lem} \label{Lem:stackeqPoincareduality}
 Assume $Y$ is the coarse moduli space of a smooth DM-stack $\mathcal{Y}$ of finite type over $\mathbb{C}$ with generically trivial stabilizer and of dimension $n$. Assume also that the action of $G$ comes from an action of $G$ on the stack $\mathcal{Y}$. Then $A^{i}_G(Y) \xrightarrow{ \cap [Y]_G} A_{n-i}^G(Y)$ is an isomorphism (still with $\mathbb{Q}$-coefficients). 
\end{Lem}
\begin{proof}
 For a representation $V$ of $G$ of dimension $l$ with an open set $U$ on which $G$ acts freely and whose complement has sufficient codimension, we can relate the equivariant Chow groups of $Y$ to the usual Chow groups for $(Y \times U)/G$. Then we have the diagram
 \[\begin{CD}
A_G^i(Y)     @=  A^i((Y \times U)/G)\\
@VV\cap [Y]_GV        @VV\cap [(Y \times U)/G]V\\
A_{n-i}^G(Y)     @=  A_{n+l-g-i}((Y \times U)/G)
\end{CD}\]
 The space $(Y \times U)/G$ (which exists as an algebraic space as the action of $G$ on $U$ is free) is the coarse moduli space of the quotient $(\mathcal{Y} \times U)/G$. This follows from the proof of Lemma \ref{Lem:quotcoarsemod}. As $\mathcal{Y} \times U$ is a smooth DM-stack, Proposition \ref{Pro:smoothDMquotient} implies that $(\mathcal{Y} \times U)/G$ is still a smooth DM stack. Hence by \cite[Proposition 6.1]{vistoliintersection}, the map on the right side of the diagram above is an isomorphism after tensoring with $\mathbb{Q}$, which finishes the proof. \todo{Smooth stacks satisfy this by Proposition 5.6, but does this carry over to our definition of coarse moduli space?}
\end{proof}
We conclude that for $G$ reductive and connected with maximal torus $T$ and Weyl group $W$ and $Y$ as in the Lemma, we have the diagram
\[\begin{CD}
A_G^i(Y)_\mathbb{Q}     @> \phi_{G,T} >>  A_T^i(Y)^W_\mathbb{Q}\\
@VV\cap [Y]_GV        @VV\cap [Y]_TV\\
A_{n-i}^G(Y)_\mathbb{Q}     @=  A_{n-i}^T(Y)^W_\mathbb{Q}
\end{CD}\]
whose vertical maps are isomorphisms by Lemma \ref{Lem:stackeqPoincareduality}. Thus $\phi_{G,T}$ is an isomorphism (to the Weyl-invariant part of $A_T^*(Y)_{\mathbb{Q}}$ and we can carry out computations in the torus equivariant operational Chow groups.

We now want to argue that for applying Proposition \ref{Pro:EqProdRecipe} to compute top-dimensional intersections on $(X \times Y)/G$, it will suffice to intersect $\beta \in A^*_T(Y)$ with the fundamental class $[Y]_T$ and find the image of $\beta \cap [Y]_T$ in $A_*(Y)$ under $\phi_{T,\{e\}}$. 

For this we will use results of Vistoli from \cite{ChowQuotientVariety} for the Chow group of quotient varieties. We still have that all (operational) Chow groups are taken with $\mathbb{Q}$-coefficients. Let us summarize the results from the above paper in our current language.
\begin{Pro}
 Let $G=\text{GL}_n$ or $G=\text{SL}_n$ act on a scheme $Y$ such that $Y$ is covered by open, invariant, affine subsets. Then the map $\phi_{G,\{e\}} : A^G_*(Y) \to A_*(Y)$ is surjective and its kernel are those elements $\beta \in A^G_*(Y)$ that can be written as a product $\beta = \sigma \tilde \beta$ of an element $\sigma \in A_G^*(pt)$ and an element $\tilde \beta \in A^G_*(Y)$. So we have an exact sequence 
 \begin{equation} \label{eqn:exseqequivgps}
  A_G^*(pt) \otimes A^G_*(Y) \to A^G_*(Y) \xrightarrow{\phi_{G,\{e\}}} A_*(Y) \to 0.
 \end{equation}
\end{Pro}
\begin{proof}
 This follows by applying Theorem 1 and Theorem 2 from \cite{ChowQuotientVariety} to the quotient map $Y \times U \to (Y \times U)/G$ for $U \subset V$ an open subset of a representation $V$ of $G$, on which $G$ acts freely and whose complement has sufficient codimension. Indeed, we still have that $Y \times U$ is covered by open, invariant affine subsets (as we can take $V=\text{Mat}(n,p)$ for $p$ large and the corresponding $U$ is covered by invariant affines). The action of $G$ on $Y \times U$ is even free, so the assumptions of the Theorems above are satisfied. The classes $c_i$ operating on $A_*((Y \times U)/G)$ are exactly the generators of $A_*^G(pt)$ from \cite{EqIntTh} with their operation on $A_*^G(Y)$. 
\end{proof}
The following Corollary summarizes the conclusions from this Proposition, that we will use later. Here we identify the ring $A_{\text{SL}_2 }^*(pt) \subset A_{\mathbb{G}_m}^*(pt)$ as  $Z[t^2] \subset Z[t]$.
\begin{Cor} \label{Cor:betafactorize}
 Let $G=\text{SL}_2$ act on a scheme $Y$. Assume that the map $A^*_G(Y) \xrightarrow{\cap [Y]_G} A_*^G(Y)$ is an isomorphism (for instance for $Y$ smooth or satisfying the conditions of \ref{Lem:stackeqPoincareduality}).
 
 Then for $i \geq 0$, $j=\lceil i/2\rceil$ any element $\beta \in A_T^{\dim(Y)+i}(Y)$ can be written as $\beta = t^{2j} \tilde \beta$ for $\tilde \beta \in A_T^{\dim(Y)+i-2j}(Y)$.
\end{Cor}
\begin{proof}
 Using the isomorphism $A^*_G(Y) \cong A_*^G(Y)$, the exact sequence (\ref{eqn:exseqequivgps}) becomes
 \[t^2 A_G^*(Y) \to A_G^*(Y) \xrightarrow{\phi_{G,\{e\}}} A^*(Y) \to 0.\]
 We do an induction on $i$, reducing it in every step by $2$. The case $i=0$ is trivial and the case $i=1$ will be obvious from the following induction step. For $i\geq 1$ we have for degree reasons that $\phi_{G,\{e\}}(\beta)=0$, as $A^*(Y)$ vanishes above degree $\dim(Y)$. Hence $\beta$ has the form $t^2 \beta_1$ with $\beta_1$ of degree $\dim(Y)+i-2$. We can then apply induction on $\beta_1$ and obtain the desired result.
\end{proof}

Now we will describe how to explicitly compute the map $\Phi$ from (\ref{eqn:EqInt2}) in our situation $X=\overline M_A, Y=\overline M_B$ with the action of $G=\text{SL}_2$. All Chow groups below will be with $\mathbb{Q}$-coefficients. Let $\alpha \in A_G^i(X), \beta \in A_G^j(Y)$ with $i+j=\dim(X \times Y/G)$. If $i > \dim(X /G)$, then $\alpha$ vanishes as $A_G^i(X) = A_G^{i}(X/G)$. Hence we know $j \geq \dim(Y)$. But then we can use Corollary \ref{Cor:betafactorize} to decompose $\beta = t^{2\lceil (j-\dim(Y))/2 \rceil} \tilde \beta$ and we can put the power of $t$ to the factor $A_G^*(X)$ in the tensor product $A_G^*(X) \otimes_{A_G^*(pt)} A_G^*(Y)$. If $j-\dim(Y)$ is odd, we see that in this case the product of $\alpha$ with this power of $t$ vanishes as it lies in degree $\dim(X/G)+1$ of $A_G^*(X) = A^*(X/G)$. 

%
%
%
In the remaining case that $j-\dim(Y)$ is even, we have arranged $\alpha \otimes \beta = (\alpha t^{j-\dim(Y)}) \otimes \tilde \beta$ and $(\alpha t^{j-\dim(Y)}) \cap [X/G]$ is a zero-cycle. We can represent this cycle as a sum of points, all with finite stabilizers. For each of these points, we are in the situation of Proposition \ref{Pro:EqProdRecipe} and we are reduced to computing the intersection of $\tilde \beta$ with $[Y]$. As we have explained above, we can also perform this intersection in $A_T^*(Y)$. Here we can apply standard localization techniques. Even more is true: we do not need to perform the decomposition $\beta = t^{j-\dim(Y)} \tilde \beta$ explicitly. We can just compute the intersection $\beta \cap [Y]_T$ in $A_*^T(Y)$ and then push forward to $A_*^T(pt)=\mathbb{Q}[t^2]$. Then $\tilde \beta \cap [Y]$ is the coefficient of $t^{j-\dim(Y)}$. This finishes the description of the algorithm.
\subsection{A recursive algorithm for intersection numbers of divisors} \label{Sect:algo}
We now have all the necessary ingredients to compute intersection numbers of divisors on our moduli spaces $M(d|d_1, \ldots, d_n)$. As input we have a collection $D_1, \ldots, D_{2d-2+n}$ of divisors on $M(d|d_1, \ldots, d_n)$. By using Corollary \ref{Cor:finalBasis} about the structure of the Picard group, we may assume that all divisors are either boundary divisors or one of $\mathcal{H} =  \mathcal{H}_{1,1}$ for $n=1,2$ or $\mathcal{G}$ for $d=0$. 

If one of the divisors (say $D_1$) is a boundary divisor $D_{B,k}$, we can use Proposition \ref{Pro:DivRestr} to restrict the other divisors to it and reduce to an intersection on $\overline{M}_A^{ss} \times \overline{M}_B / G$. As we have seen in Section \ref{Sect:EquivarProduct}, we can reduce this to
first performing an equivariant intersection on $\overline M_B$ and then computing an intersection on $\overline{M}_A/G$, which is a moduli space $M(d-k| k, (d_a)_{a \in A})$ with lower degree $d-k$ or fewer markings. 
Recall that $\overline M_B = \overline M_{0,B \cup \{p'\}}(\PP^1, k)$
so the intersection can be computed using Atiyah-Bott localization (see for instance \cite[Chapter 27]{mirrorpandha}).

Note, however, that the intersection on $M(d-k| k, (d_a)_{a \in A})$ may also involve a purely equivariant factor $t^{2r} \in A_G^*(pt)$, which arose from the intersection on $\overline M_B$. We need a way to identify this class in terms of the known divisor classes in order to conclude the recursion. This can be done using the following Lemma.
\begin{Lem}
 In $A^*(M(d|d_1, \ldots, d_n))_{\mathbb{Q}}$, the class coming from $t^2 \in A_G^*(pt)_{\mathbb{Q}}$ agrees with $c_1(\mathcal{H}_{i,j})^2$ for $i=1, \ldots, n$, $j=1,2$ and also with $c_1(\mathcal{G})^2$ for $d=0$. 
\end{Lem}
\begin{proof}
 Note that from \cite[Section 3.3]{EqIntTh}, it follows that for the usual action of $G=\text{SL}_2$ on $\PP^1$ we have an equivariant operational Chow ring \[A_G^*(\PP^1_{\mathbb{Q}}=\mathbb{Q}[h,t^2]/(h^2-t^2),\] where $t^2$ is the generator of the equivariant Chow ring of a point and $h$ comes from the canonically $G$-linearized line bundle $\mathcal{O}(1)$ on $\PP^1$. To see this, use that $A_G^*(\PP^1)_{\mathbb{Q}}$ is the Weyl-invariant part of $A_T^*(\PP^1)_{\mathbb{Q}}=\mathbb{Q}[h,t]/(h^2-t^2)$. Here $T \subset G$ is the diagonal torus acting with weights $1,-1$ on $\PP^1$ and the generator $\sigma$ of the Weyl-group $W=\mathbb{Z}/2\mathbb{Z}$ acts by $\sigma h = h, \sigma t = -t$.
 
 Now for $i=1, \ldots, n$, $j=1,2$, consider the $G$-equivariant evaluation map 
 \[\text{ev}_{i,j} = \pi_j \circ \text{ev}_i : Y_{d,n}^{ss,\textbf{d}} \to \PP^1\]
 giving the $j$th component of the evaluation of the $i$th marking in $\PP^1 \times \PP^1$. The line bundle $\mathcal{H}_{i,j}$ is exactly $\text{ev}_{i,j}^* \mathcal{O}(1)$. But then the class \[c_1(\mathcal{H}_{i,j}) \in A^*(M(d|d_1, \ldots, d_n))_{\mathbb{Q}} = A_G^*(Y_{d,n}^{ss,\textbf{d}})_{\mathbb{Q}}\] of the induced line bundle is nothing but the pullback of the class $h \in A_G^*(\PP^1)_{\mathbb{Q}}$ by
 \[\text{ev}_{i,j}^* : A_G^*(\PP^1)_{\mathbb{Q}} \to A_G^*(Y_{d,n}^{ss,\textbf{d}})_{\mathbb{Q}}.\]
 But as $h^2 = t^2$ in $A_G^*(\PP^1)_{\mathbb{Q}}$, we have $ c_1(\mathcal{H}_{i,j})^2 = t^2$.
 
 For $d=0$, we must have at least one marking to have a nonempty moduli space, so $n \geq 1$. Then the claim follows because $\mathcal{G} = \mathcal{H}_{1,2}$. 
\end{proof}
By the above result, we can replace factors $t^2$ by $\mathcal{H}^2$ in the presence of markings. But note that the space $\overline M_A$ always has at least the marking coming from the gluing point $p$.

The only thing missing to finish the recursion is a strategy to deal with intersections which only feature the divisors $\mathcal{H}, \mathcal{G}$ (in cases $n=1,2$ or $d=0$, respectively), so we intersect $c_1(\mathcal{H})^a c_1(\mathcal{G})^b$. Here, we have one tool at our disposal in case $a>1$, because we know $c_1(\mathcal{H})^2 = c_1(\mathcal{H}_{1,1})^2 = c_1(\mathcal{H}_{1,2})^2$ by the Lemma above. However, using Proposition \ref{Pro:Hprime}, we know that 
\[\mathcal{H}_{1,2} = d \mathcal{H}_{1,1} + R,\]
where $R$ is a combination of boundary divisors and possibly $\mathcal{G}$. 
Thus we know
\[c_1(\mathcal{H}_{1,1})^2 = c_1(\mathcal{H}_{1,2})^2=d^2 c_1(\mathcal{H}_{1,1})^2 + c_1(R) \left(2d c_1(\mathcal{H}_{1,1}) + c_1(R) \right).\]
In the case $d \neq 1$, we can solve this equation for $c_1(\mathcal{H}_{1,1})$ and express it as an intersection of boundary divisors, $\mathcal{G}$ and only one copy of $c_1(\mathcal{H}_{1,1})$. With this we can reduce $a$ sucessively until we can assume $a \leq 1$.

However, for $d=1$ we need a different strategy. But note that we have already restricted to the cases $n=1,2$. The case $d=1,n=1$ is a base-case below. For $d=1,n=2$ one checks using Lemma \ref{Lem:semistableadvanced} that either $D_{1=\text{fix}}=0$ or $D_{2=\text{fix}}=0$, as one of these loci in $Y_{1,2}$ is unstable. But we have
\[D_{1=\text{fix}} = 2 \mathcal{H}_{1,1} + R, D_{2=\text{fix}}=-2 \mathcal{H}_{1,1} + R',\]
where $R,R'$ are sums of boundary divisors. Hence we can express $\mathcal{H}_{1,1}$ as such a sum and use the recursion above. 

What we have proven is that we can restrict to the case $a \leq 1$.
For $a=1$ we have $n=1,2$. For $d \neq 0$, we would have $b=0$, so our moduli space is one-dimensional, so $2d-2+n=1$, which implies $d=1,n=1$. This is one of the base cases covered below. On the other hand, for $d=0$ we need $n=2$ to have $2d-2+n \geq 0$, which will also be treated below as a base case for the recursion.

On the other hand for $a=0$ we again distinguish cases: if $b=0$, our moduli space must be zero-dimensional, so $2d-2+n=0$, which implies again $d=0, n=2$. For $b=1$ we have $d=0$ and so we are in the case of $M(0|c,d,e)$. This is covered in the following Lemma.
\begin{Lem}
 To compute the intersection number $(\mathcal{G})$ on the moduli spaces $M(0|c,d,e)$ we have to distinguish two cases
 \begin{itemize}
  \item If the sum of any two of $c,d,e$ is strictly larger than the third plus $1$, we have $M(0|c,d,e) \cong \PP^1$ and $\mathcal{G}=\mathcal{O}(1)$, hence $(\mathcal{G})=1$.
  \item If, say, $c+d \leq 1+e$ then we have
  \[(\mathcal{G}) = \frac{1}{2}\left(1 \underbrace{-1}_{\text{if }c+e<1+d} \underbrace{-1}_{\text{if }d+e<1+c}\right).\]
 \end{itemize}
\end{Lem}
\begin{proof}
 In the first case, we see that by Lemma \ref{Lem:semistableadvanced} no two of the horizontal positions of the three markings can collide, so we can use our $\text{PGL}_2$-action to arrange these positions to be $0,1, \infty$, in a unique way. It also follows that all three markings can still be fixed points of the degree $0$ map (as $2+c<1+d+e$ and similar inequalities hold). Hence we see that the map 
 \begin{align*}
  \PP^1 &\to M_{0,3}(\PP^1 \times \PP^1, (1,0))^{ss,(c,d,e)}\\
  q &\mapsto (\text{id} \times q : \PP^1 \to \PP^1 \times \PP^1; 0,1, \infty)
 \end{align*}
 composed with the quotient map to $M(0|c,d,e)$ is an isomorphism. But from the definition it is clear that $\mathcal{G}=\mathcal{O}(1)$.
 
 For $c + d \leq 1+ e$, it follows $1+c+d \leq 2 + e$. Again by Lemma \ref{Lem:semistableadvanced} this implies that $\text{Fix}_3 = 0$. Using Proposition \ref{Pro:iddivisors}, one finds that on $Y_{0,3}$ we have
 \[\text{Fix}_3 =  \mathcal{G}-\frac{1}{2} D_{\{1,2\},0} + \frac{1}{2} D_{\{1,3\},0}+ \frac{1}{2} D_{\{2,3\},0} + \frac{1}{2} D_{\{1,2,3\},0}.
 \]
 A generic point of $D_{\{1,2\},0}$ is always stable and on the other hand all points in $D_{\{1,2,3\},0}$ are unstable (we need $a+b+c\geq 1$ because otherwise there are no semistable points at all). On the other hand, a generic point of $D_{\{1,3\},0}$ (or $D_{\{2,3\},0}$) is stable iff $c+e<1+d$ (or $d+e<1+c$).
 
 Then setting $\text{Fix}_3=0$, solving for $\mathcal{G}$ and observing that $D_{\{1,2\},0}$, $D_{\{1,3\},0}$ and $D_{\{2,3\},0}$ are reduced points in the quotient (if they are stable at all), we find the claimed formula.
\end{proof}
Finally, for $b \geq 2$ we have $n=b+2 \geq 4$ and we replace two of the factors $c_1(\mathcal{G})$ by $c_1(\mathcal{H}_{1,1})$. But $\mathcal{H}_{1,1}$ can be expanded in the basis of the Picard group of our moduli space, which by Corollary \ref{Cor:finalBasis} only contains boundary divisors and $\mathcal{G}$. But by Proposition \ref{Pro:iddivisors}, the coefficient $c_\mathcal{G}$ of $\mathcal{G}$ in this expansion is equal to $(C_\mathcal{G}, \mathcal{H}_{1,1})=0$. Thus only boundary divisors appear and we are able to enter the recursion described above.

We are left to do the base cases: for $d=0, n=2$ the only nonempty spaces are isomorphic to a point by Lemma \ref{Lem:M(0;1,1)}, so the empty intersection gives result $1$. For $d=1, n=1$ we have seen in Section \ref{Sect:d1n0} that all nonempty moduli spaces are isomorphic to $M(1|1) \cong \PP^1$. On this space we can use 
\[0 = \text{Fix}_1 = \mathcal{H}_{1,1} + \mathcal{H}_{1,2} = 2 \mathcal{H} + \frac{1}{2} D_{\emptyset,1}\]
by Proposition \ref{Pro:Hprime}. As $D_{\emptyset,1}$ consists of a single point with no stabilizer, it gives intersection number $1$ and $\mathcal{H}$ therefore gives intersection number $-\frac{1}{4}$.

This finishes the description of the algorithm. It has been implemented using \verb+SAGE+ (see \cite{sage}) and is available from the author upon request. In this implementation, we used a SAGE program to compute top-intersections on spaces $\overline{M}_{0,n}$ written by Drew Johnson. This program can be found at \cite{drewj}.

\begin{appendix}
\section{Notations} \label{App:notations}
\noindent\begin{tabular}{@{}p{0.22 \linewidth} p{0.75 \linewidth}@{}}
$\text{Rat}_d$ & the space $\text{Rat}_d$ of degree $d$ maps $\PP^1 \to \PP^1$\\
$Z_d$ & the compactification $Z_d=\PP(H^0(\PP^1 \times \PP^1, \mathcal{O}(d,1)))$ of $\text{Rat}_d$\\
$\mathcal{Y}_{d,n}$ & the moduli stack $\overline{ \mathcal{M}}_{0,n}(\PP^1 \times \PP^1, (1,d))$ of stable maps to $\PP^1 \times \PP^1$ of degree $(1,d)$ with $n$ marked points\\
$Y_{d,n}$ & the coarse moduli space $\overline M_{0,n}(\PP^1 \times \PP^1, (1,d))$ of $\mathcal{Y}_{d,n}$\\
$M(d,n)$ & the moduli space of degree $d$ self-maps with $n$ marked points (Corollary \ref{Cor:quotexist1})\\
$M(d|d_1, \ldots, d_n)$ & the moduli space of degree $d$ self-maps with $n$ marked points and weights $d_1, \ldots, d_n$ (Corollary \ref{Cor:quotexist2})\\
$\mathcal{M}(d|d_1, \ldots, d_n)$ & the quotient stack $\mathcal{Y}_{d,n}^{ss,\textbf{d}} / \text{PGL}_2$\\
$A_{d,n}$ & the locus of points in $Y_{d,n}$ with $\text{PGL}_2$-isotropy\\
$D_{B,k}$ & the boundary divisor $D(\{1, \ldots, n\} \setminus B, (1,d-k); B,(0,k))$ inside $Y_{d,n}$ or $M(d|d_1, \ldots, d_n)$\\
$\mathcal{H}_{i,j}$ & the divisor $(\pi_j \circ \text{ev}_i)^* \mathcal{O}_{\PP^1}(1)$ on $Y_{d,n}$\\
$\mathcal{H}_{i,2}'$ & the divisor $\mathcal{H}_{i,2} - d \mathcal{H}_{i,1}$ on $Y_{d,n}$\\
$D_{i=\text{fix}}$ & the divisor $\text{ev}_i^{-1}(\Delta) \subset Y_{d,n}$ for the evaluation map $\text{ev}_i : Y_{d,n} \to \PP^1 \times \PP^1$\\
$\mathcal{G}$ & the divisor $\pi_{\PP^1}^* \mathcal{O}_{\PP^1}(1)$ on $Y_{0,n} = \overline M_{0,n}(\PP^1 \times \PP^1,(1,0)) \cong \overline M_{0,n}(\PP^1 ,1)\times \PP^1$\\
$\Gene{d}{n}$ & the generators of the rational Picard group of $Y_{d,n}$ from Theorem \ref{Theo:Generators}\\
$\Basi{d}{n}$ & the basis of the rational Picard group of $Y_{d,n}$ from Corollary \ref{Cor:PicYbasis} \\
$\psi_{B,k}$ & the test curve $\PP^1 \to Y_{d,n}$ constructed in Definition \ref{Def:psiCBk}\\
$C_{B,k}$ & the image cycle $(\psi_{B,k})_* [\PP^1]$ in $Y_{d,n}$\\
$\compomo$ & the composition map $Y_{d_1,n} \times Y_{d_2,0} \dashrightarrow Y_{d_1 d_2, n}$ (Theorem \ref{Theo:compomo})\\
$\mathfrak{sc}_m$ & the $m$-fold self-composition map $Y_{d,n} \dashrightarrow Y_{d^m,n}$ (Section \ref{Sect:selfcompo})\\
$\text{Per}_m(\lambda)$ & the Weyl-divisors in $Y_{d,n}$ of maps with $m$-periodic orbits of multiplier $\lambda$
\end{tabular}

\section{Generalities} \label{App:generalities}
\begin{Theo} \label{Theo:flatforgetful}
 Let $X=G/P$ be a homogeneous space, where $G$ is an algebraic group and $P$ is a parabolic subgroup. Let $\beta \in A_1(X)$ be a curve class, $n \geq 3$ and consider the forgetful morphism
 \[F : S= \overline M_{0,n}(X,\beta) \to \overline M_{0,n}.\]
 Then $F$ is flat of relative dimension $\delta = \text{dim}(X) + \int_\beta c_1(T_X)$.
\end{Theo}
\begin{proof}
 By \cite[Proposition 7.4]{behrendmanin} applied to $(V,\tau) = (X,g, n, \beta)$, we know that the morphism of Deligne-Mumford stacks $\overline M_{g,n}(X,\beta) \to \overline M_{g,n}$ is flat of dimension $\delta$. In general this will not imply that the corresponding morphism of the coarse moduli spaces is also flat.
 
 However in the case $g=0$, we can use the Miracle flatness theorem (see for instance \cite[Theorem 23.1]{miracleflatness}). It tells us that a map from a Cohen-Macaulay variety to a smooth variety with constant fibre dimension is flat. Now indeed by \cite{pandhaconnected} we have that $F$ is a map between irreducible varieties and $\overline M_{0,n}$ is smooth by \cite{KnudsenII}. By \cite[Theorem 2]{fultonpandha} the variety $\overline M_{0,n}(X,\beta)$ is locally the quotient $V/H$ of a smooth variety $V$ by a finite group $H$. Smooth varieties are Cohen-Macaulay and we now use the Hochster-Roberts theorem from \cite{Hochster} to show that then also $V/H$ is Cohen-Macaulay. The theorem says that if an affine linearly reductive group (like $H$) acts rationally on a Noetherian $k$-algebra, then the ring of invariants is Cohen-Macaulay. But the map $V \to V/H$ satisfies the conditions of \cite[1.§2, Theorem 1.1]{git} and thus $V/H$ is covered by the spectra of rings of invariants as above, hence it is Cohen-Macaulay.
 
 We are thus left to show that the fibres of $F$ are of dimension $\delta$. But this property is preserved when going from the stacks to the coarse moduli spaces, hence we are done by \cite[Proposition 7.4]{behrendmanin}. 
\end{proof}

\begin{Pro} \label{Pro:bijectiso}
 Let $k$ be an algebraically closed field of characteristic zero and $f:X \to Y$ a morphism between integral $k$-schemes of finite type over $k$. Then if $f$ is bijective on the closed points of $X$ and $Y$ and $Y$ is normal, the morphism $f$ is an isomorphism.
\end{Pro}
\begin{proof}
 This can be proved using Zariski's Main Theorem together with the fact that for a dominant, separable morphism of varieties over an algebraically closed field $k$ which has degree $d$, the generic fibre consists of $d$ points. \detex{For a proof see \cite{bijectiso}.}
\end{proof}

\begin{Pro} \label{Pro:SectionEmbedding}
 Let $\pi : X \to Y$ be a locally finitely presented, flat, separated morphism of relative dimension $n$. Let $s:Y \to X$ be a section of $\pi$ such that $s(p)$ is a smooth point of $X_p$ for all geometric points $p \in Y$. Then $s$ is an effective Cartier divisor in $X$.
\end{Pro}
\begin{proof}
 As $s$ is a section of a separated morphism, it is a closed embedding (see \cite[\href{http://stacks.math.columbia.edu/tag/024T}{Tag 024T}]{stacks}). On the complement of the singular locus $S$ of the fibres $X_p$ we have that $\pi: X \setminus S \to Y$ is still locally finitely presented and flat and now also smooth, as the geometric fibres are smooth. Then by \cite[B7.3]{inttheory}, $s$ is a regular embedding into $X \setminus S$. But as $X \setminus S$ and $X \setminus s(Y)$ define an open cover of $X$, $s$ is also a regular embedding into $X$, that is a Cartier divisor. \todo{Union of singular loci of fibres is closed?}
\end{proof}

\begin{Lem} \label{Lem:detpushforward}
 Let $\pi : X \to Y$ be a finite, flat morphism from a scheme $X$ to a normal, irreducible variety $Y$. Let $\mathcal{L}$ be a line bundle on $X$ and let $\varphi : \mathcal{L} \to \mathcal{L}$ be a morphisms such that $\pi(V(\varphi)) \subset Y$ is a proper subset.
 
 Then $\pi_* \mathcal{L}$ is a vector bundle on $Y$ and we have an equality of cycles
 \[\pi_*([V(\varphi)]) = [V(\text{det}(\pi_* \varphi: \pi_* \mathcal{L} \to \pi_* \mathcal{L}))].\]
\end{Lem}
\begin{proof}
 For any $V \subset Y$ irreducible of codimension $1$, we have to check that the multiplicities of $[V]$ in both sides of the claimed equation agree. Hence we may take a base change by the spectrum of $A=\mathcal{O}_{V,Y}$ and we obtain a map $\pi_V : \text{Spec}(B) \to \text{Spec}(A)$, where $B$ is a finite $A$-algebra.
 
 As the fibres of $\pi$ are finite, by general theorems about base-change and cohomology the sheaf $\pi_* \mathcal{L}$ is locally free. Moreover, we have
 \[(\pi_* \mathcal{L})|_{\text{Spec}(A)} = (\pi_V)_* \mathcal{L}|_{\text{Spec}(B)}.\]
 The restriction of $\mathcal{L}$ to $\text{Spec}(B)$ corresponds to a $B$-module $M$. The pushforward $(\pi_V)_* \mathcal{L}|_{\text{Spec}(B)}$ still corresponds to $M$, which is now a locally free $A$-module. As $A$ is a discrete valuation ring, $M$ is thus even free as an $A$-module. The endomorphism $\varphi$ of $\mathcal{L}$ is given by multiplication with some function, so on $\text{Spec}(B)$ it corresponds to multiplication with some $b \in B$. Then by \cite[Example A.2.3]{inttheory} we have
 \[l_A(M/(b)) = l_A(A/(\text{det}(b)).\]
 The left side gives the multiplicity of $[V]$ in $\pi_*([V(\varphi)])$ and the right side gives its multiplicity in $[V(\text{det}(\pi_* \varphi: \pi_* \mathcal{L} \to \pi_* \mathcal{L}))]$. 
\end{proof}


\section{Group actions on stacks} \label{App:stackaction}
The following results use definitions and techniques for group actions on stacks introduced by Romagny in \cite{romagny2005}. 
\begin{Lem} \label{Lem:quotfibreprod}
 Let $S$ be a scheme, $G$ a flat, separated group scheme of finite presentation over S. Let $\mathcal{M},\mathcal{N}$ be $G$-algebraic stacks over $S$ and let $f: \mathcal{M} \to \mathcal{N}$ be a morphism of $G$-stacks. Then there exists a canonical commutative diagram
 \begin{equation}
  \begin{CD}
 \mathcal{M}     @>f>>  \mathcal{N}\\
@VVV        @VVV\\
\mathcal{M}/G     @>\tilde f>>  \mathcal{N}/G
\end{CD}, \label{eqn:quotfibreprodgen}
 \end{equation}
 and this diagram is a fibre product.
\end{Lem}
\begin{proof}
 Our assumptions on $G$ are chosen in such a way, that the quotient stacks $\mathcal{M}/G, \mathcal{N}/G$ are isomorphic to the stacks of $G$-torsors $(\mathcal{M}/G)^*$, $(\mathcal{N}/G)^*$  by \cite[Theorem 4.1]{romagny2005}. 
 
 Recall that for a $G$-stack $(\mathcal{M}, \mu_\mathcal{M}:G \times \mathcal{M} \to \mathcal{M})$, the stack $(\mathcal{M}/G)^*$ has as objects over a scheme $T/S$ triples 
 \[t=(p:E \to T,h:E \to \mathcal{M},\sigma).\]
 Here $E$ is an algebraic space with a strict action $\nu: G \times E \to E$, $p$ is $G$-invariant such that fppf-locally on $T$ it is isomorphic to the projection $G \times T \to T$ and the map $(h, \sigma) : E \to \mathcal{M}$ is a morphism of $G$-stacks.
 
 Isomorphisms between $t,t'$ are pairs $(u,\alpha)$ with a $G$-morphism $u: E \to E'$ and a $2$-isomorphism $\alpha: (h,\sigma) \implies (h',\sigma') \circ u$. 
 
 What will be important for us is that $(\mathcal{M}/G)^*, (\mathcal{N}/G)^*$ are algebraic stacks, i.e. they have surjective, fppf atlases $V \to (\mathcal{M}/G)^*$, $U \to (\mathcal{N}/G)^*$ from schemes $U,V$. Moreover, the natural maps $\mathcal{M} \to (\mathcal{M}/G)^*$ and $\mathcal{N} \to (\mathcal{N}/G)^*$ are actually the universal $G$-torsors.
 
 As a first step, we want to construct the map $\tilde f: (\mathcal{M}/G)^* \to (\mathcal{N}/G)^*$. On objects over $T/S$ it is given by
 \[\left(\begin{CD}
 E     @>(h,\sigma)>>  \mathcal{M}\\
@VVV        \\
T     
\end{CD}\right) \mapsto \left(\begin{CD}
 E     @>f \circ (h,\sigma)>>  \mathcal{N}\\
@VVV        \\
T     
\end{CD}\right).\]
 One checks that the resulting diagram (\ref{eqn:quotfibreprodgen}) is commutative. Let $\mathcal{Z} = \mathcal{N} \times_{(\mathcal{N}/G)^*} (\mathcal{M}/G)^*$, then there is a natural map $\phi: \mathcal{M} \to \mathcal{Z}$. We want to show that it is an isomorphism by taking suitable base-changes with smooth maps from schemes. First take the base change with the atlas $V \to (\mathcal{M}/G)^*$ and we obtain
 \[\begin{CD}
\mathcal{M}_V @>>> \mathcal{Z}_V @>>> V\\
	    @VVV     @VVV  @VVV\\
 \mathcal{M} @>>> \mathcal{Z} @>>> (\mathcal{M}/G)^*\\
	    @.     @VVV  @VVV\\
	    @.   \mathcal{N} @>>>  (\mathcal{N}/G)^*
\end{CD}\]
It suffices to show $\mathcal{M}_V \to \mathcal{Z}_V$ is an isomorphisms. Because $\mathcal{M}_V \to V$ is a $G$-torsor we have $V \cong \mathcal{M}_V / G$ and $\mathcal{M}_V$ is an algebraic space. We have reduced the problem to the case where $\mathcal{M}$ is an algebraic space and $(\mathcal{M}/G)^*$ is a scheme. Similarly, by taking the base change of the entire diagram by the atlas $U \to (\mathcal{N}/G)^*$ we reduce to the case $\mathcal{N}$ an algebraic space. But here we see that the morphism $(\mathcal{M}/G)^* \to (\mathcal{N}/G)^*$ is induced by the $G$-equivariant map $\mathcal{M} \to \mathcal{N}$ of the $G$-torsor $\mathcal{M}$ over $(\mathcal{M}/G)^*$. Hence the fact that the corresponding diagram is cartesian is simply the definition of the universal property of the $G$-torsor $\mathcal{N} \to (\mathcal{N}/G)^*$. 
 \todo{Check this with someone competent!!!}
\end{proof}
\begin{Rmk} \label{Rmk:quotinherit}
 As the map $\mathcal{N} \to \mathcal{N}/G$ above is fppf (check on any atlas $U \to \mathcal{N}/G$ pulling back to a locally trivial $G$-bundle on $U$) any property of $f: \mathcal{M} \to \mathcal{N}$ that is fppf-local on the target is inherited by the induced map $\tilde f: \mathcal{M}/G \to \mathcal{N}/G$. Also, arguing as above, for any cartesian diagram of $G$-algebraic stacks where all morphisms are $G$-morphisms the induced diagram of the quotients is also cartesian.
\end{Rmk}

\begin{Pro} \label{Pro:smoothDMquotient}
 Let $\mathcal{M}$ be an orbifold, i.e. a smooth separated Deligne-Mumford stack with connected coarse moduli space and containing a non-empty open substack which is a scheme. Let the smooth group scheme $G$ act on $\mathcal{M}$ with finite, reduced stabilizers at geometric points. Then the quotient $\mathcal{M}/G$ is again a smooth DM stack.
\end{Pro}
\begin{proof}
 We want to use the frame bundle $\mathcal{F} = \text{Fr}\left( T_\mathcal{M} \right)$ of the tangent bundle $T_\mathcal{M}$ of $\mathcal{M}$. By \cite[Exercise 1.183]{brambila2014moduli}, $\mathcal{F}$ is an algebraic space. For $n=\dim(\mathcal{M})$, the group $\text{GL}_n$ acts on $\mathcal{F}$ on the right by
 \begin{equation} \label{eqn:GLnaction}
   (v_1, \ldots, v_n) . (a_{ij})_{i,j=1}^n = (\sum_{i=1}^n a_{i1} v_i, \ldots, \sum_{i=1}^n a_{in} v_i)
 \end{equation}
 and we have that $\mathcal{M} = [\mathcal{F}/\text{GL}_n]$. On the other hand, the action of $G$ on $\mathcal{M}$ induces an action of $G$ on $\mathcal{F}$ by
 \begin{equation} \label{eqn:Gaction}
   g . (v_1, \ldots, v_n) = (g_* v_1, \ldots, g_* v_n),
 \end{equation}
 where $g_*$ denotes the pushforward under the map $p \mapsto gp$ on $\mathcal{M}$. 
 
 Note that the actions of $G$ and $\text{GL}_n$ commute, because $g_* : T_\mathcal{M} \to T_\mathcal{M}$ is linear in the fibres. Indeed, we have
 \begin{align*}
  g. ((v_1, \ldots, v_n) . (a_{ij})) &= g.(\sum_i a_{i1} v_i, \ldots, \sum_i a_{in} v_i) \\
  &=(\sum_i a_{i1} g_* v_i, \ldots, \sum_i a_{in} g_*v_i)\\
  &=(g. (v_1, \ldots, v_n)) . (a_{ij}).
 \end{align*} 
 This means that we can combine these two actions to an action of $G \times \text{GL}_n$ on $\mathcal{F}$. As $\mathcal{F}$ is an algebraic space, every action of an algebraic group on $\mathcal{F}$ (in the usual sense) is automatically a strict action in the sense of \cite{romagny2005}.

 Note that a geometric point $\sigma=(v_1, \ldots, v_n) \in \mathcal{F}$ in the fibre of $p \in \mathcal{M}$ has finite, reduced stabilizer in $G \times \text{GL}_n$. Indeed, for a pair $(g,A)$ to stabilize $\sigma$, the element $g \in G$ must stabilize $p$. By assumption, the stabilizer of $p$ in $G$ is finite and reduced. Hence the claim follows, as the action of $\text{GL}_n$ on the fibre of $\mathcal{F}$ over $p$ is simply transitive, so the stabilizer of $\sigma$ is isomorphic to the stabilizer of $p$ in $G$.
 
 With these preparations done, we simply note that
 \[\mathcal{M}/G =  [\mathcal{F}/\text{GL}_n]/G = (\mathcal{F}/\text{GL}_n)/G = \mathcal{F}/(\text{GL}_n \times G).\]
 The second isomorphism is a consequence of \cite[Theorem 4.1]{romagny2005}, the third isomorphism comes from \cite[Remark 2.4]{romagny2005}, as $\text{GL}_n \subset \text{GL}_n \times G$ is a normal subgroup with quotient group $G$. But now $\mathcal{F}$ is a smooth algebraic space and the action of $G \times \text{GL}_n$ has finite, reduced stabilizers at geometric points. Then by \cite[Proposition 5.27]{algstacks}, the quotient $\mathcal{F}/(\text{GL}_n \times G)$ is again a Deligne-Mumford stack and as $\mathcal{F}$ is smooth and a locally trivial $\text{GL}_n \times G$-torsor over it, it is also smooth. 
\end{proof}

\begin{Lem} \label{Lem:quotcoarsemod}
 Let $\mathcal{M}$ be an Deligne-Mumford stack with a strict action $\mu_{\mathcal{M}}$ of flat, separated group scheme $G$ of finite presentation. Let $M$ be a coarse moduli space carrying the induced action $\mu_{M}$ of $G$ and assume $M$ is locally Noetherian. Let $N$ be an algebraic space and $[M/G] \to N$ a tame moduli space (in the sense of \cite{alpergood}). Then there are natural maps \[\mathcal{M}/G \to [M/G] \to N\]
 making $N$ a coarse moduli space for $\mathcal{M}/G$ and $[M/G]$.
\end{Lem}
\begin{proof}
 By definition, we have 
 \[\text{Hom}(\mathcal{M}/G,S) = \text{Hom}_G(\mathcal{M},S) \]
 for any algebraic space $S$, where the right means homomorphisms of $G$-stacks with the trivial $G$ action on $S$. 
 But any morphism $\psi: \mathcal{M} \to S$ must factor uniquely through a morphism $\tilde \psi: M \to S$.
 Because $S$ is an algebraic space, the category $\text{Hom}_G(\mathcal{M},S)$ is actually a set(oid) and by definition it is given by those morphisms $f:\mathcal{M} \to S$ of stacks such that 
 \[f \circ \mu_{\mathcal{M}} = f \circ \pi_{\mathcal{M}} : G \times \mathcal{M} \to S.\]
 Now as $\mathcal{M}$ is Deligne-Mumford, it is tame and so the formation of its moduli space $M$ commutes with base change. 
 Hence $G \times M$ is a coarse moduli space for $G \times \mathcal{M}$, so the morphisms from these spaces to $S$ agree. Thus the above condition is equivalent to asking for a morphism $\tilde f: M \to S$ such that
 \[\tilde f \circ \mu_{M} = \tilde f \circ \pi_{M} : G \times M \to S.\]
 By now we have shown $\text{Hom}(\mathcal{M}/G,S) = \text{Hom}_G(M,S)$, but as $M \to [M/G]$ is the universal $G$-torsor over $[M/G]$ this clearly equals $\text{Hom}([M/G],S)$. As $M$ is locally Noetherian, so is $[M/G]$. Then, by \cite[Theorem 6.6]{alpergood}, the good moduli space $[M/G] \to N$ is universal among maps to algebraic spaces, so $\text{Hom}([M/G],S) = \text{Hom}(N,S)$.
 
 Note that the map $\mathcal{M}/G \to [M/G]$ is obtained from the morphism $\mathcal{M} \to M$ of $G$-stacks via Lemma \ref{Lem:quotfibreprod}, using that $M/G = [M/G]$ by Theorem \cite[Theorem 4.1]{romagny2005}.
 
 Finally, any geometric point $p \in \mathcal{M}/G$ corresponds to a (necessarily) trivial $G$-torsor $E \to p$ together with a $G$-equivariant map $E \to \mathcal{M}$. This data is equivalent to specifying an orbit of some geometric point $\widehat{p} \in \mathcal{M}$. But as the geometric points of $\mathcal{M}$ and $M$ agree, this is equivalent to specifying an orbit of a geometric point in $M$. In turn, this is equivalent to a geometric point of $[M/G]$. By the definition of a tame moduli space, the geometric points of $[M/G]$ and $N$ agree. This shows that geometric points of $\mathcal{M}/G, [M/G]$ and $N$ coincide.
\end{proof}
\begin{Rmk} \label{Rmk:quotcoarsemod}
 A sufficient condition for the morphism $[M/G] \to N$ to be a tame moduli space is to require that $G$ is a smooth, affine, linearly reductive group scheme over a field $k$ and $M=M^{ss}=M^s$ for some $G$-linearized line bundle $\mathcal{L}$ on $M$ and $N=M \sslash G$. This is the only situation in which we are going to use the above result.
 
 Note that in the Lemma above, it does not suffice to ask for instance $N$ to be a geometric quotient of $M$ by $G$, if $M$ is a scheme. Indeed, in Example 8.6 of \cite{alpergood} we have a geometric quotient $X \to \mathbb{A}^1$ of a scheme $X$ by $\text{SL}_2$ such that $X/G=[X/G]$ is the non-locally separated affine line, which is an algebraic space. Hence $\mathbb{A}^1$ is not universal for morphisms from $X/G$ to algebraic spaces.
\end{Rmk}


The following type of group action on a stack appears in the study of self-maps. 
\begin{Lem} \label{Lem:stackyaction}
 Let $X$ be a projective, algebraic scheme over $\mathbb{C}$ and $\beta \in A_1 X$. Let $G$ be a group scheme over $S= \text{Spec}(\mathbb{C})$ with multiplication $m: G \times G \to G$ and unit $e: S \to G$. Assume that $\sigma: G \times X \to X$ is an algebraic action leaving $\beta$ invariant. 
 
 Then there is an induced action of $G$ on $\mathcal{M}=\overline{\mathcal{M}}_{0,n}(X, \beta)$, which sends a $\mathbb{C}$-point $(g,(f:C \to X; p_1, \ldots, p_n))$ to $(\sigma(g,-) \circ f: C \to X; p_1, \ldots, p_n)$.
\end{Lem}
\begin{proof}
 By \cite[Definition 2.1]{romagny2005}, an action of $G$ on $\mathcal{M}$ is a morphism of stacks $\mu: G \times \mathcal{M} \to \mathcal{M}$ such that the diagram
 \begin{align}
  \begin{CD}
G \times G \times \mathcal{M}     @>m \times \text{id}_{\mathcal{M}}>>  G \times \mathcal{M}\\
@VV\text{id}_G \times \mu V        @VV\mu V\\
G \times \mathcal{M}     @>\mu>>  \mathcal{M}
\end{CD} \label{eqn:comdiagaction}
 \end{align}
 commutes and such that $\mu \circ (e \times \text{id}_{\mathcal{M}}) = \text{id}_{\mathcal{M}}$. Note that we want these equalities of $1$-morphisms to hold strictly, that is not up to a choice of $2$-morphism between them. We now proceed to construct $\mu$ and check the relations above.
 
 As all our fibre products are over the base category $\mathcal{S}=\text{Sch}/S$, an object in $G \times \mathcal{M}$ over a scheme $T \to S$ consists of a tuple
  \[\left(\begin{tikzcd} T \arrow{d}{g} \\G\end{tikzcd}\   , \  
  \begin{tikzcd} C \arrow{d}{\pi} \arrow{r}{f} & X \\ T & \end{tikzcd}  \  , p_1, \ldots, p_n : T \to C
 \right)\]
 
with $p_i$ sections of $\pi$. The functor $\mu$ assigns to this the stable family $(\pi : C \to T, g.f : C \to X, p_1, \ldots, p_n)$ where 
\begin{equation}
 g.f = \sigma \circ ((g \circ \pi)\times f).
\end{equation}
One checks that this still defines an element of $\mathcal{M}(T)$ as the $G$-action on $X$ preserves $\beta$. 
Now assume we have a morphism between objects
\begin{equation*}
\begin{array}{rccccr}
 P=(&T \xrightarrow{g} G, &C \xrightarrow{\pi} T, &C \xrightarrow{f} X, &p_1, \ldots, p_n : T \to C&),\\
 P'=(&T' \xrightarrow{g'} G, &C' \xrightarrow{\pi'} T',& C' \xrightarrow{f'} X, &p_1', \ldots, p_n' : T' \to C'&).
\end{array}
\end{equation*}
This morphism is given by the data $\varphi: T \to T'$, $\overline \varphi: C \to C'$ such that $g' \circ \varphi = g$, $f' \circ \overline \varphi = f$ and such that
\[\begin{CD}
C     @>\overline \varphi >>  C'\\
@VV\pi V        @VV\pi' V\\
T    @>\varphi >>  T'
\end{CD}\]
becomes a cartesian diagram. Then the induced morphism $\mu(P \to P')$ between $\mu(P)$ and $\mu(P')$ shall be given by the same data $(\varphi, \overline \varphi)$. To see that this is well-defined, we have to check that $(g'.f')\circ \overline \varphi = g.f$. Indeed
\begin{align*}
 (g'.f')\circ \overline \varphi &= \sigma \circ ((g' \circ \pi') \times f')\circ \overline \varphi\\
 &= \sigma \circ ((g' \circ \pi'\circ \overline \varphi) \times f'\circ \overline \varphi)\\
 &= \sigma \circ ((g' \circ \varphi \circ \pi) \times f)\\
 &= \sigma \circ ((g \circ \pi) \times f) = g.f.
\end{align*}
Now that $\mu$ is defined, we first check the commutative diagram in (\ref{eqn:comdiagaction}). First we start with an object
\[\left(\begin{tikzcd} T \arrow{d}{g} \\G\end{tikzcd}\ , \begin{tikzcd} T \arrow{d}{h} \\G \end{tikzcd}\   ,\  \begin{tikzcd} C \arrow{d}{\pi} \arrow{r}{f} & X \\ T & \end{tikzcd} \  , p_1, \ldots, p_n : T \to C
 \right) \in G \times G \times \mathcal{M}.\]
Then the images under the upper-right and lower-left corners of the diagram (\ref{eqn:comdiagaction}) are both stable families $\pi: C \to T$ with identical markings $p_i$ so we only have to show that the maps $C \to X$ coincide. Spelling out the definitions we obtain the equation
\[(m \circ (g \times h)).f = g.(h.f).\]
It is shown by the following diagram
\[\begin{CD}
C     @=  C\\
@VV((g \times h)\circ \pi) \times fV        @VV((g \times h)\circ \pi) \times fV\\
G \times G \times X     @=  G \times G \times X\\
@VV m \times \text{id}V        @VV\text{id} \times \mu V\\
G \times  X     @.  G \times  X\\
@VV \mu V        @VV\mu V\\
X     @=  X
\end{CD}\]
where the left and right vertical side are the desired morphisms. Here the lower diagram commutes because $\sigma$ is an action. The fact that $\mu$ acts identically on morphisms is clear, because by definition $\mu$ does not change the data $(\varphi, \overline \varphi)$ of the morphism at all.

The fact that the action $\mu$ is compatible with the identity $e$ is also straightforward.
\end{proof}
\end{appendix}

\bibliographystyle{alpha} 
\bibliography{Biblio}

\end{document}